\documentclass[12pt]{amsart}
\usepackage[margin=2.0cm]{geometry}

\usepackage{amssymb}
\usepackage{amsfonts}
\usepackage{mathrsfs}
\usepackage{cite}
\usepackage{graphicx}
\usepackage{mathtools}

\usepackage{float}

\usepackage{todonotes}

\usepackage{pgfplots}
\pgfplotsset{compat=1.15}
\usepackage{mathrsfs}
\usetikzlibrary{arrows}

\definecolor{ffqqqq}{rgb}{1,0,0}
\definecolor{qqzzqq}{rgb}{0,0.6,0}

\newcommand{\R}{{\mathbb R}}

\newcommand{\hyp}{\mathbb H}

\newtheorem{theorem}{Theorem}[section]

\newtheorem{lemma}[theorem]{Lemma}
\newtheorem{prop}[theorem]{Proposition}
\newtheorem{corollary}[theorem]{Corollary}

\newtheorem{de}[theorem]{Definition}
\newtheorem{example}[theorem]{Example}

\newcommand{\M}{\mathcal{M}}

\DeclareMathOperator{\arccosh}{\mathrm{arcosh}}

\makeatletter
\DeclareFontFamily{U}{tipa}{}
\DeclareFontShape{U}{tipa}{m}{n}{<->tipa10}{}
\newcommand{\arc@char}{{\usefont{U}{tipa}{m}{n}\symbol{62}}}%

\newcommand{\arc}[1]{\mathpalette\arc@arc{#1}}

\newcommand{\arc@arc}[2]{%
  \sbox0{$\m@th#1#2$}%
  \vbox{
    \hbox{\resizebox{\wd0}{\height}{\arc@char}}
    \nointerlineskip
    \box0
  }%
}
\makeatother

\title{P\'al's isominwidth problem in the hyperbolic space}
\author{K\'aroly J. B\"or\"oczky, Ansgar Freyer, \'Ad\'am Sagmeister }
\address{Alfr\'ed R\'enyi Institute of Mathematics,  Realtanoda u. 13-15, H-1053 Budapest, Hungary} \email{boroczky.karoly.j@renyi.hu}
\thanks{K\'aroly J. B\"or\"oczky is funded by grant ADVANCED\_24 150613 of the  Ministry of Culture and Innovation of Hungary from the National
Research, Development and Innovation Fund.}
\thanks{Ansgar Freyer is funded by the Deutsche Forschungsgemeinschaft (DFG, German Research Foundation) - 539867386.}
\thanks{\'Ad\'am Sagmeister is funded by the grant NKFIH 150151. Project no. 150151 has been implemented with the support provided by
the Ministry of Culture and Innovation of Hungary from the National
Research, Development and Innovation Fund, financed under the
ADVANCED\_24 funding scheme.}
\address{FU Berlin, Fachbereich Mathematik und Informatik, Arnimallee 2, D-14195 Berlin, Germany} \email{a.freyer@fu-berlin.de}
\address{Bolyai Institute, University of Szeged, Aradi vértanúk tere 1, H-6720 Szeged,
Hungary} \email{sagmeister.adam@gmail.com }

\mathtoolsset{showonlyrefs=true}

\begin{document}

\begin{abstract}
The paper focuses on possible hyperbolic versions of the classical
P\'al isominwidth inequality in $\R^2$ from 1921, which states that for a fixed minimal width, the regular triangle has minimal area. We note that the isominwidth problem is still wide open in $\R^n$ for $n\geq 3$. Recent work on the isominwidth problem on the
sphere $S^2$ shows that the solution is the regular spherical triangle when the width is at most $\frac{\pi}{2}$ according to Bezdek and Blekherman,  while Freyer and Sagmeister proved that the minimizer is the polar of a spherical Reuleaux triangle when the minimal width is greater than $\frac{\pi}{2}$.  

In this paper, the hyperbolic isominwidth problem is discussed with respect to the probably most natural notion of width due to Lassak in the hyperbolic space $\hyp^n$ where strips bounded by a supporting hyperplane and a corresponding hypersphere are considered. On the one hand, we show that the volume of a convex body of given minimal Lassak width $w>0$ in $\hyp^n$ might be arbitrarily small; therefore, the isominwidth problem for convex bodies in $\hyp^n$ does not make sense. On the other hand, in the two-dimensional case, we prove that among horocyclically convex bodies of given Lassak width in $\hyp^2$, the area is minimized by the regular  horocyclic triangle. 
In addition, we also verify a stability version of the last result.
\end{abstract}

\maketitle

\section{Introduction}
\label{sec:intro}


We call a compact convex set with non-empty interior a convex body.
The minimal width, or sometimes also called thickness of a convex body $K\subset\R^n$, is the minimum distance of two parallel hyperplanes.
More formally, it is defined as
\begin{equation}
\label{intro:eq:width}
w(K) = \min_{u\in S^{n-1}} \max_{x,y\in K} \langle u, x-y\rangle,
\end{equation}
The \emph{isominwidth} problem asks for the minimal volume of a convex body $K\subset \R^n$ with $w(K) = 1$. One can view the isominwidth problem as a convex version of Kakeya's needle problem \cite{Kak17}, which is concerned with the minimum volume of measurable sets that contain a unit interval in any direction.

In the plane, the isominwidth problem was solved by P\'al in a 1921 paper \cite{Pal21}; namely, among convex bodies in $\R^2$ of given minimal width $w>0$, the regular triangles minimize the volume. This beautiful result may be regarded as a dual version of the isodiametric inequality due to Bieberbach (in the plane) \cite {Bie15} and Urysohn \cite{Ury24} (in higher dimensions), which states among all convex bodies $K\subset\R^n$ of fixed diameter $d$, the Euclidean $n$-balls of diameter $d$ uniquely maximize the volume.
While the isodiametric inequality holds in all dimensions, a higher-dimensional analogue of P\'al's result is not known. 


Apart from the Euclidean setting, an isominwidth (or isodiametric) problem may be posed in each kind of geometry that allows for a definition of width (diameter) and volume. For spaces of constant curvature, the isodiametric problem has been solved by Schmidt \cite{Sch48,Sch49} (see also
B\"or\"oczky, Sagmeister \cite{BoS20,BoS23}). While there exists a natural definition of minimal width in terms of the smallest lune enclosing a convex body in the spherical case $S^n$, there are various different notions of minimal width  
in the hyperbolic space $\hyp^n$ starting with an approach by Santal\'o \cite{S45} in 1945
(see, for example, Fillmore \cite{Fil70},
Horv\'ath \cite{Hor21},
Jer\'onimo-Castro, Jimenez-Lopez \cite{JCJL17},
Lassak \cite{Las23},
Leichtweiss \cite{Lei05},
and the surveys
\cite{BoCsS, Hor21}). 

On the 2-sphere $S^2$, the isominwidth inequality has been proven by Bezdek and Blekherman \cite{Bez00} for width $w\leq\frac{\pi}{2}$. Here the minimal width (thickness) of a convex body $K\subset S^2$ is the minimal width of a lune containing $K$; namely, the minimal $\alpha\in (0,\pi)$ such that
there exist $u,v\in S^2$ with angle $\pi-\alpha$ and satisfying that
$\langle x,u\rangle\leq 0$
and $\langle x,v\rangle\leq 0$ hold for any $x\in K$.

\begin{theorem}[Pál \cite{Pal21}, Bezdek--Blekherman\cite{Bez00}]\label{isominwidth}
Let $\M^2$ be either $\R^2$, or $S^2$,  and let $w>0$ where
 $w\leq\frac{\pi}{2}$ if $\M^2=S^2$.
If $K\subset\mathcal{M}^2$ is a convex body of minimal width $w$, and $\widetilde{T}_{w}\subset\mathcal{M}^2$ denote an equilateral triangle of minimal width $w$, then
$$
V_{\mathcal{M}^2}\left(K\right)\geq V_{\mathcal{M}^2}\left(\widetilde{T}_{w}\right),
$$
with equality if and only if $K$ is congruent with $\widetilde{T}_{w}$
where $V_{\mathcal{M}^2}(\cdot)$ is the Lebesgue measure in $\mathcal{M}^2$. 
\end{theorem}

We note that Freyer, Sagmeister \cite{FrS} proved that if the minimal width is larger than $\frac{\pi}{2}$ in the spherical case, then the area minimizer convex body is the polar of a spherical Reuleaux triangle.
In addition, stability versions are due to Lucardesi, Zucco \cite{LuZ} in the case of the classical P\'al theorem in $\R^2$, and to Freyer, Sagmeister \cite{FrS} in the spherical versions of the P\'al theorem.\\

This paper focuses on possible hyperbolic versions of P\'al's isominwidth inequality.
We use a recently introduced  natural width function from 
Lassak \cite{Las23}; namely, given a supporting hyperplane $H$ to  a convex body $K\subset \hyp^n$, the Lassak width of $K$ with respect to $H$ is the maximal distance of the points of $K$ from $H$. These notions are introduced in full detail in Section~\ref{sec:convexity}. In this case, the minimal Lassak width; namely, the ``Lassak thickness'' $w(K)$ of $K$ is
the smallest value of the Lassak width with respect to a supporting hyperplane of $K$, and $w(K)$ is actually
the minimal width of a strip that contains $K$ and is bounded by a hyperplane and a corresponding hypersphere.  The fundamental properties of Lassak width  in  $\hyp^n$ are discussed in Section \ref{secLassak-width}.

Surprisingly enough, Pál's problem does not have a solution in the hyperbolic space for general convex bodies with respect to the Lassak width. For a convex body $K\subset \hyp^n$ (compact convex set with non-empty interior, see Section~\ref{secHypConvex}), we write $V_{\hyp^n}(K)$ to denote its hyperbolic volume.

\begin{theorem}
\label{intro:thm:nopal}
For $w>0$ and $n\geq 2$, we have
$$
\inf\left\{V_{\hyp^n}\left(K\right)\colon K\subset \hyp^n\text{ convex body, }w\left(K\right)\geq w\right\}=0.
$$
\end{theorem}

As a corollary, we deduce that no hyperbolic analogue of
Steinhagen's theorem about the inradius and the minimal width of convex bodies in $\R^n$ exists (cf. 
Corollary~\ref{thm:nopal-radius}).

However, we can obtain a planar result within a distinguished class of convex bodies introduced by Santal\'o \cite{San68} in 1968. 
A closed set $X\subset \hyp^n$ is h-convex or horocyclically convex (sometimes called horoconvex) if for any $x,y\in X$, $x\neq y$, $\sigma\subset X$ holds for any horocyclic arc $\sigma$ connecting $x$ and $y$ (see Section \ref{sech-convex}); or equivalently, $X$ is the intersection of horoballs
 (see Corollary~\ref{HoroballIntersection}). We call a compact h-convex set with non-empty interior an h-convex body. As natural analogues of Euclidean convex sets, h-convex sets have been studied in $\hyp^n$ from the point of view of conformal geometry, integral geometry, curvature flows, fundamental gap, etc  (see, for example,
 Andrews, Chen, Wei \cite{ACW21}, Assouline, Klartag \cite{AK24},
Gallego, Naveira, Solanes \cite{GNS04}, Gallego, Reventos, Solanes, Teufel \cite{GRST08},  
 Grossi, Provenzano \cite{GrP24},
 Hu, Li, Wei \cite{HLW22}
Mej\'{\i}a, Pommerenke \cite{MeP05}
Nguyen, Stancu, Wei \cite{NSW22},
Santal\'o \cite{San68} for studies on h-convex sets or on their boundaries). 

Given three vertices $q_1,q_2,q_3\in \hyp^2$ of a regular triangle, the corresponding regular horocyclic triangle $T$ is the intersection of the three horoballs $\Xi_1,\Xi_2,\Xi_3$ where $q_m,q_k\in\partial \Xi_j$, $\{j,m,k\}=\{1,2,3\}$ and $\Xi_j$ contains $q_j$  (see
Section~\ref{secRegHoroTriangle}, $T$ is actually the h-convex hull of $q_1,q_2,q_3$). As we will see in Section~\ref{secRegHoroTriangle}, the minimal Lassak width  $w=w(T)$  is the distance of $q_j$ from $\partial\Xi_j$, that is, the length of intersection of $T$ and the perpendicular bisector of the segment $[q_m,q_k]$, $\{j,m,k\}=\{1,2,3\}$. In particular, we frequently write $T=T_w$.

\begin{theorem}
\label{intro:thm:pal:hconv}
For $w>0$, if $K\subset \hyp^2$ is any h-convex body of minimal Lassak width at least $w>0$, and $T_w$ is a regular horocyclic triangle  of minimal Lassak width $w$, then 
$$
V_{\hyp^2}\left(K\right)\geq V_{\hyp^2}\left(T_w\right)
$$
with equality if and only if $K$ is congruent to $T_w$.
\end{theorem}

We also prove a stability version of 
Theorem~\ref{intro:thm:pal:hconv}.
We write $d(x,y)$ to denote the hyperbolic geodesic distance of  $x,y\in \hyp^n$, and given compact $X,Y\subset \hyp^n$, their hyperbolic Hausdorff distance is
\begin{equation}
\label{Hausdorff-hyp0}
\delta(X,Y)=\max\left\{\max_{x\in X}\min_{y\in Y}d(x,y),~\max_{y\in Y}\min_{x\in X}d(x,y)\right\}.
\end{equation}
We note that the Hausdorff distance is a metric on the space of compact subsets of $\hyp^n$.

\begin{theorem}
\label{thm:pal_hconv-stab0}
For $w>0$, if
$K\subset \hyp^2$ is an h-convex body of minimal Lassak width at least $w$ and $V_{\hyp^2}(K)\leq (1+\varepsilon)V_{\hyp^2}(T_w)$ for $\varepsilon\in[0,1]$ and a  regular horocyclic triangle  $T_w$ of minimal Lassak width $w$, then there exists an isometry $\Phi$ of $\hyp^2$ such that
$$
\delta(K,\Phi T_w)\leq c\sqrt{\varepsilon}
$$
where $c>0$ is an explicitly calculable constant depending on $w$.
\end{theorem}

Most probably, the error term of order $\sqrt{\varepsilon}$ in Theorem~\ref{thm:pal_hconv-stab0} can be improved to an error term of order $\varepsilon$. The order of the error term cannot be less than $\varepsilon$, as taking the h-convex hull of $T_w$ and a point $p_\varepsilon$ of distance $c_0\varepsilon$ from $T_w$  for suitable $c_0>0$ depending on $w$ shows.

The paper is organized as follows. In the upcoming Section~ \ref{sec:convexity}, we recall the basic terms and concepts in hyperbolic geometry that are needed for the study of convex and h-convex sets. 
In particular, we introduce the notion of Lassak width $w$, and establish its basic properties.
In
Section~\ref{sec:pal:hyp}, we prove Theorem~\ref{intro:thm:nopal}, i.e., that the isominwidth problem for the minimal Lassak width of convex bodies does not make sense in the hyperbolic space $\hyp^n$ (cf.
Theorem~\ref{thm:nopal}).

Concerning positive results in $\hyp^2$, Section~\ref{secRegHoroTriangle}
proves the two-dimensional hyperbolic analogue of Steinhagen's theorem among h-convex domains; namely, the extremality of the horocyclic regular triangles 
with respect to minimal Lassak width and inradius (cf. Theorem~\ref{Blaschke:horocyclic}).
Finally, 
Theorem~\ref{intro:thm:pal:hconv}
is proved in Section~\ref{sec-isomin}, and
Theorem~\ref{thm:pal_hconv-stab0} is verified in Section~\ref{sec-isomin-stab}.

\section{Life in the hyperbolic space}
\label{sec:convexity}

For a background on hyperbolic geometry, see Ratcliffe \cite{Rat19} Chapters 3 and 4, or Greenberg \cite{Gre93}.
In this paper, we use the Poincar\'e ball model to represent the Hyperbolic space $\hyp^n$. We mostly survey the main notions and their properties, and only prove some simple technical estimates that we will need in the sequel.

\subsection{The Poincar\'e ball model of the hyperbolic space}

In the Poincar\'e ball model, the hyperbolic space $\hyp^n$ is identified with the interior of the unit Euclidean ball $B^n$ in $\R^n$, and the set of ideal points is just $\partial B^n$. For $p,q\in {\rm int}B^n$, their hyperbolic distance is defined to be (cf. Ratcliffe \cite{Rat19} Chapter~4)
\begin{equation}
\label{hypdist}
d(p,q)=\arccosh\left(1+\frac{2\|p-q\|^2}{(1-\|p\|^2)(1-\|q\|^2)}\right),
\end{equation}
where $\|p\|$ denotes the Euclidean norm.
In the special case, when $p=o$, then the center of $B^n$, the hyperbolic distance is (cf. Ratcliffe \cite{Rat19} Chapter~4)
\begin{equation}
\label{hypdist-origin}
d(o,q)=\ln\frac{1+\|q\|}{1-\|q\|}\mbox{ \ and \ }
\|q\|=\frac{e^d-1}{e^d+1}\mbox{ \ for $d=d(o,q)$}.
\end{equation}
For $p\in {\rm int}B^n$ and radius $r>0$, the hyperbolic ball of center $p$ and radius $r$ is
$B(p,r)=\{q\in {\rm int}B^n: d(p,q)\leq r\}$. Actually, $B(p,r)$ is a Euclidean ball as well, only its Euclidean center is different from $p$ unless $p=o$. We observe that the hyperbolic and the Euclidean topologies on  ${\rm int}B^n$ coincide.  When we want to emphasize the intrinsic hyperbolic geometry, then we simply write $\hyp^n$ for ${\rm int}B^n$.

In the following, a $k$-sphere is the relative boundary of a solid $(k+1)$-dimensional Euclidean ball in $\R^n$, $k=1,\ldots,n-1$. If the center and the radius of an $(n-1)$-sphere $\Sigma$ is $p$ and $r$, respectively, and the center and the radius of a $k$-dimensional sphere $C$ are $q$ and $\varrho$, respectively, then we say that $C$ is orthogonal to $\Sigma$ if and only if $\|p-q\|^2=r^2+\varrho^2$. For example, if $k=1$, then the circle $C\subset\R^n$ is orthogonal to the $(n-1)$-sphere $\Sigma\subset \R^n$ if there exists $x\in C\cap \Sigma$, and the exterior normals to $\Sigma$ at $x$ are tangent to $C$. In addition, if $k\geq 2$, then $C$ is orthogonal to $\Sigma$ if and only if $C$ contains a great circle (whose center is $q$) that is orthogonal to $\Sigma$.

Let us discuss the basic objects in the Poincar\'e ball model of the hyperbolic space. A hyperbolic line in the Poincar\'e ball model is the intersection of ${\rm int}B^n$ and either a Euclidean circle that is orthogonal to $\partial B^n$ at the two intersection points, or a Euclidean line containing $o$, and hence any hyperbolic line connects two ideal points. These hyperbolic lines are the geodesics with respect to the hyperbolic distance in \eqref{hypdist}. If $\ell$ is a hyperbolic line and $x,y\in\ell$, then $[x,y]\subset \ell$ denotes the unique hyperbolic (geodesic) segment connecting $x$ and $y$.

A particular property of 
the Poincar\'e ball model is that it preserves Euclidean angles, namely, if $\sigma_1,\sigma_2\subset {\rm int}B^n$ are closed Euclidean circular arcs or segments meeting at a common endpoint $p$, then the hyperbolic angle of $\sigma_1$ and $\sigma_2$ at $p$ coincides with the  Euclidean angle of the two tangent vectors at $p$. In addition, $\sigma_1$ is orthogonal to a part $X$ of an $(n-1)$-sphere or Euclidean hyperplane  containing $p$ in terms of the hyperbolic geometry if and only if it is orthogonal in the Euclidean sense. For pairwise different $p,q,r\in{\rm int}B^n$, we write $\angle(p,q,r)$ to denote the angle of the hyperbolic half-lines emanating from $q$ and passing through $p$ or $r$, and hence $0\leq \angle(p,q,r)\leq \pi$.

In general, if $k=1,\ldots,n-1$, then a hyperbolic $k$-space $\Pi$ in $\hyp^n$ is the intersection of
${\rm int}B^n$ and either a $k$-sphere orthogonal to $\partial B^n$, or a Euclidean $k$-subspace of $\R^n$ containing $o$. We observe that $\Pi$ equipped with the restriction of the hyperbolic metric \eqref{hypdist} is isomorphic to $\hyp^k$.  

In particular, a hyperbolic hyperplane $H$ in the Poincar\'e ball model is the intersection of ${\rm int}B^n$ either with a Euclidean $(n-1)$-sphere orthogonal to $\partial B^n$ at the $(n-2)$-sphere $S$ of intersection points, or with a Euclidean hyperplane containing $o$ that intersects $\partial B^n$ in a great sphere $S$. Actually, $H$ is the unique hyperbolic hyperplane whose set of ideal points is $S$. We observe that ${\rm int}B^n\backslash H$ has two connected components, and the closure of either of them is called a closed {\it half-space}, and its boundary is $H$. For $\varrho>0$, the surface of points in one of the half-spaces whose hyperbolic distance from $H$ is $\varrho$ is a so-called hypersphere, and it is of the from $\Sigma\cap {\rm int}B^n$ for an $(n-1)$-sphere  $\Sigma$
where  $S\subset \Sigma$, $\Sigma\neq \partial B^n$ and  $H\not\subset \Sigma$. 
In addition,
any hyperbolic line $\ell$ orthogonal to the hyperbolic  hyperplane $H$ is also orthogonal to any hypersphere corresponding to $H$. If $n=2$, then hyperspheres are frequently called hypercycles.

In the Poincar\'e ball model, a horoball $\Xi$ at an ideal point $i\in \partial B^n$ is of the form  $G\setminus\left\{i\right\}$ for a Euclidean $n$-ball $G\subset B^n$ of radius less than one and touching $\partial B^n$ at $i$, and the corresponding horosphere 
$\partial \Xi$ at $i\in \partial B^n$
is  $\partial G\setminus\left\{i\right\}$. It also follows that for a hyperbolic line $\ell$, $i$ is an ideal point of  $\ell$ if and only if  $\ell$ is orthogonal to $\partial\Xi$. If $n=2$, then horospheres are frequently called horocycles.  Actually, even if $n\geq 3$, if $C\subset B^n$ is a Euclidean circle such that $C\cap\partial B^n=\{i\}$, then $C\backslash \{i\}$ is a horocycle with ideal point $i$.

In summary: let $G\subset\R^n$ be a  Euclidean ball such that $\partial G\cap{\rm int}B^n\neq \emptyset$. Then $G$ is a hyperbolic ball if $G\subset {\rm int}B^n$, and $G\backslash\{i\}$ is a horoball if $G$ touches $\partial B^n$ in the ideal point $i$. In the other cases, 
$\partial G\cap\partial B^n$ is an $(n-2)$-sphere $S$, and 
for the hyperbolic hyperplane $H$ whose set of ideal points is $S$,
either $\partial G\cap{\rm int}B^n=H$  provided that $\partial G$ is orthogonal to $\partial B^n$, or $\partial G\cap{\rm int}B^n$ is a hypersphere corresponding to $H$; namely, there exists $\varrho>0$ and a half-space $H^+$ of $\hyp^n$ bounded by $H$ such that $\partial G\cap{\rm int}B^n$ is the set of points of $H^+$ whose distance from $H$ is $\varrho$. 

Concerning properties and examples of isometries of the Poincar\'e ball model, 
 the restrictions of the orthogonal transformations of $\R^n$ to ${\rm int}B^n$ form the isometries of the Poincar\'e ball model fixing $o$. 
For a hyperplane $H\subset \hyp^n$, a fundamental example of isometries is the orientation reversing reflection through $H$, which leaves the points of $H$ fixed, and for $x\in \hyp^n\backslash H$, the image $x'$ satisfies that $H$ is the perpendicular bisector of
 the segment $[x,x']$. Given $p,q\in \hyp^n$, $p\neq q$,
  there is a unique orientation preserving ``hyperbolic translation'' $\Phi$ that maps $p$ into $q$, where $\Phi$ maps any point of the line $\ell$ of $p$ and $q$ into a  point of $\ell$, and if $x\not\in\ell$, then $\Phi x$ lies on the hypercycle $\sigma$ corresponding to $\ell$ and passing through $x$ where $\sigma$ is contained in the hyperperbolic $2$-space spanned by $\ell$ and $x$.

  The following lemma allows us to choose a specific point in $\hyp^n$ occurring in a problem as the center $o$ of the Poincar\'e ball model ${\rm int}\,B^n$ (cf. Ratcliffe \cite{Rat19} Chapters~3 and 4).

 \begin{lemma}
\label{HypIsometries}
Isometries of the hyperbolic space $\hyp^n$ are transitive on hyperplanes and also transitive on horospheres, even allowing  to fix a point on these objects. 
In addition, they keep angles, and map hyperspheres into hyperspheres.  
 \end{lemma}
 
Two hyperbolic lines $\ell_1,\ell_2\subset \hyp^n$ are parallel if they have a common ideal point $i$. For $p_j\in\ell_j$, $j=1,2$, the half-lines $g_1$ and $g_2$ where $g_j\subset \ell_j$ connects $p_j$ to $i$ are also called parallel.
A characteristic property of parallel lines is that for a half-line $g$ emanating from $p_2$ and lying on the same side of the line of $[p_1,p_2]$ as $g_1,g_2$, we have
\begin{align}
\label{parallel-half-lines}
g\cap g_1\neq\emptyset\mbox{ if and only if }&\mbox{the angle of $[p_1,p_2]$ and $g$ at $p_2$ is less}\\ 
\nonumber
&\mbox{than the angle of $[p_1,p_2]$ and $g_2$ at $p_2$.}
\end{align}
The law of cosines for angles (cf. Lemma~\ref{triangle}) yields that if $g_2$ is orthogonal to segment $[p_1,p_2]$, then the angle $\alpha$ of 
$[p_1,p_2]$ and the half-line $g_1$ parallel to $g_2$ satisfies 
\begin{equation}
\label{parallel-angle}
\sin\alpha=\frac1{\cosh a} \mbox{ \ for }a=d(p_1,p_2).
\end{equation}

 On the other hand, two hyperbolic hyperplanes $H_1,H_2\subset \hyp^n$ are \emph{ultraparallel}, if 
 $H_1\cap H_2=\emptyset$ and they do not have a common ideal point either. Since the hyperspheres corresponding to $H_1$ have the same sets of ideal points as $H_1$, we deduce that there exist $x_1\in H_1$ and $x_2\in H_2$ such that $d(x_1,x_2)$ minimizes the distance between points of $H_1$ and $H_2$, and hence the line $\ell$ of $x_1$ and $x_2$ is the unique line orthogonal to both $H_1$ and $H_2$. In particular, two hyperplanes $H_1,H_2\subset \hyp^n$ are ultraparallel if and only if there exists a (unique) line $\ell$ orthogonal to both $H_1$ and $H_2$ at some $p_j=\ell\cap H_j$, $j=1,2$. Concerning the distance $d(x,H_2)$ of an $x\in H_1$ from $H_2$,
 \begin{equation}
\label{ultraparallel-distance-infty}
d(x,H_2)\mbox{ tends to infinity if $d(x,p_1)$ tends to infinity for }x\in H_1.
\end{equation}

We will need estimates between the hyperbolic and the Euclidean distance of points $p,q\in{\rm int}B^n$. On the one hand, 
if $d(o,q)\leq D$, then \eqref{hypdist-origin} yields that
\begin{equation}
\label{hypdist-origin-est}
\|q\|\leq \frac{e^D-1}{e^D+1}.
\end{equation}

On the other hand, we have the following lemma.
\begin{lemma}
\label{hyp-Euc-dist}
If $\|p\|,\|q\|\leq \theta<1$, then
$$
2\|p-q\|\leq d(p,q) \leq \frac{2}{1-\theta^2}\cdot \|p-q\|.
$$
\end{lemma}
\begin{proof}
 If $t\geq 0$, then comparing the Taylor series of $\cosh t$ and $1+\frac{t^2}2$  yields that
\begin{equation}
\label{acosh-est}
\arccosh\left(1+\frac{t^2}2\right)\leq t.
\end{equation}
We deduce from \eqref{hypdist} and \eqref{acosh-est} that if $\|p\|,\|q\|\leq \theta<1$, then
$$
d(p,q)\leq \arccosh\left(1+\frac{2\|p-q\|^2}{(1-\theta^2)^2}\right)\leq \frac{2}{1-\theta^2}\cdot \|p-q\|.
$$

For the lower bound on $d(p,q)$,  if $t\in[0,1)$, then comparing the Taylor series of $e^{2t}$ and $\frac{1+t}{1-t}=1+\sum_{k=1}^\infty 2t^k$ implies that
\begin{equation}
\label{lnplusminus-est}
\ln \frac{1+t}{1-t}\geq 2t.
\end{equation}
For $p,q\in{\rm int}B^n$, we deduce from \eqref{hypdist} and  the AM-GM inequality that
$$
d(p,q)=\arccosh\left(1+\frac{2\|p-q\|^2}{(1-\|p\|^2)(1-\|q\|^2)}\right)
\geq \arccosh\left(1+\frac{2\|p-q\|^2}{\left(1-\frac{\|p\|^2+\|q\|^2}2\right)^2}\right).
$$
For $z=\frac{p+q}2$ and $y=\frac{p-q}2$, we have $p-q=2y$ and
$\|p\|^2+\|q\|^2=2(\|z\|^2+\|y\|^2)\geq \|y\|^2+\|-y\|^2$. Therefore,  \eqref{lnplusminus-est} yields that
\begin{align*}
d(p,q)\geq &\arccosh\left(1+\frac{2\|y-(-y)\|^2}{\left(1-\frac{\|y\|^2+\|-y\|^2}2\right)^2}\right)
=\arccosh\left(1+\frac{2\|y-(-y)\|^2}{(1-\|y\|^2)(1-\|-y\|^2)}\right)\\
=&d(y,-y)=2d(o,y)=2\cdot \ln\frac{1+\|y\|}{1-\|y\|}\geq 4\|y\|=2\|p-q\|.
\end{align*}
\end{proof}

Given compact $X,Y\subset {\rm int}B^n$, their hyperbolic Hausdorff distance is
\begin{equation}
\label{Hausdorff-hyp}
\delta(X,Y)=\max\left\{\max_{x\in X}\min_{y\in Y}d(x,y),\max_{y\in Y}\min_{x\in X}d(x,y)\right\}.
\end{equation}
In addition, the Euclidean Hausdorff distance of $X$ and $Y$ is
\begin{equation}
\label{Hausdorff-Euc}
\delta_{\rm Euc}(X,Y)=\max\left\{\max_{x\in X}\min_{y\in Y}\|x-y\|,\max_{y\in Y}\min_{x\in X}\|x-y\|\right\}.
\end{equation}
Both $\delta(\cdot,\cdot)$ and $\delta_{\rm Euc}(\cdot,\cdot)$ are metrics on compact subsets of ${\rm int}B^n$. We deduce from 
 Lemma~\ref{hyp-Euc-dist} the following.

\begin{lemma}
\label{hyp-Euc-Hausdorff-dist}
If $\theta\in(0,1)$ and $X,Y\subset \theta B^n$ are  compact, then
$$
2\cdot \delta_{\rm Euc}(X,Y)\leq \delta(X,Y) \leq \frac{2}{1-\theta^2}\cdot \delta_{\rm Euc}(X,Y).
$$
\end{lemma}

The Euclidean Hausdorff metric in \eqref{Hausdorff-Euc} can be naturally extended to compact subsets of $B^n$, and the resulting metric space on compact subsets of $B^n$ is also compact according to Hausdorff's theorem from 1911. We also observe that if a sequence of closures of hyperbolic hyperplanes - that are intersections of $B^n$ and  $(n-1)$-spheres or hyperplanes orthogonally intersecting $\partial B^n$ - tends to a compact set in $\R^n$ that intersects ${\rm int}B^n$, then the limit is also the closure (in $\R^n$) of a hyperbolic hyperplane. The same holds for horospheres, and if the limit of
a sequence of closures (in $\R^n$) of hyperbolic hyperspheres intersects ${\rm int}B^n$, then the limit is either a hypersphere, or a horosphere, or a hyperplane.
We consider the space of hyperbolic hyperplanes, hyperspheres and horospheres with this ``standard topology''. 
Let us provide an intrinsic approach.

\begin{lemma}
\label{space-hyperplanes-compact}
For  a (non-empty) compact $K\subset {\rm int}\,B^n$, and a sequence $\{X_j\}$  where
each $X_j\subset {\rm int}B^n$ 
intersects $K$ and 
\begin{description}
\item[(a)] either each $X_j$ is a hyperbolic hyperplane,
 \item[(b)] or each $X_j$ is a horosphere,
 \item[(c)] or each $X_j$ is a hypersphere,
\end{description}  
there exists a subsequence $\{X_{j'}\}\subset \{X_j\}$ that converge to some $X$ such that
$X$ is a hyperbolic hyperplane in the case of (a), $X$ is a horosphere in the case of (b), and $X$ is a hyperplane, horosphere or hypersphere in the case of (c) depending on whether the distance $\varrho_{j'}$ of $X_{j'}$ to its corresponding hyperplane $H_{j'}$ tends to  $0$, $\infty$ or a positive finite number.

In particular, if $K\subset {\rm int}\,B(p,R)$ for $R>0$ and $p\in \hyp^n$, then
$\{X_{j'}\cap B(p,R)\}$ tends to $X\cap B(p,R)$ with respect to the Hausdorff metric.
\end{lemma}

For a non-empty compact $X\subset \hyp^n$ that is not a singleton, let the {\it circumradius} $R(X)$ be the  minimal radius of balls containing $X$. Since the intersection of two balls of radius $R$ in $\hyp^n$ is contained in a ball of smaller radius, we have the following.

\begin{lemma}
\label{circumscribed-ball}
If $X\subset \hyp^n$ is compact and is neither empty nor a singleton, then there exists a unique ball of radius $R(X)$ containing $X$ that is called the circumscribed ball of $X$.
\end{lemma}

Concerning volume of a Borel measurable $X\subset {\rm int}B^n$, we write $|X|$ to denote the Euclidean Lebesgues measure, and $V(X)=V_{\hyp^n}(X)$ to denote the hyperbolic volume (cf. Ratcliffe \cite{Rat19} Chapter~4)
\begin{equation}
\label{Poincare-volume}
V(X)=\int_X\left(\frac{2}{1-\|x\|^2}\right)^n\,dx
\end{equation}
where the integration is with respect to the Euclidean Lebesgue measure. We deduce the following estimate.

\begin{lemma}
\label{hyp-Euc-volume}
If $\theta\in(0,1)$ and $X\subset \theta B^n$ is Borel measurable, then
$$
2^n|X|\leq V(X)\leq \left(\frac{2}{1-\theta^2}\right)^n |X|.
$$
\end{lemma}

If $n=2$, then some exact formulas are known for hyperbolic triangles and circular disks (cf. Ratcliffe \cite{Rat19} Chapter~3).
For $A,B,C\in{\rm int}B^2$ that are not contained in a hyperbolic line, the corresponding triangle $T$ with vertices $A,B,C$ is the intersection of the three half-planes containing $A,B,C$ such that exactly two of the points $A,B,C$ lie on the boundary. Then  $\alpha=\angle(B,A,C)$,
$\beta=\angle(A,B,C)$ and
$\gamma=\angle(A,C,B)$ are the angles of $T$ at $A,B,C$,
and $a=d(B,C)$, $b=d(A,C)$
and $c=d(A,B)$ are the lengths of the sides of $T$ opposite to
$A,B,C$. We recall that if $t\in\R$, then $\cosh t=\frac{e^t+e^{-t}}2$ and
$\sinh t=\frac{e^t-e^{-t}}2$.

\begin{lemma}
\label{triangle}
Using the notation as above, we have the following.
\begin{description}
\item[Area as angle deficit] $V(T)=\pi-\alpha-\beta-\gamma$.
\item[Law of sines] $\frac{\sin\alpha}{\sinh a}=\frac{\sin\beta}{\sinh b}=\frac{\sin\gamma}{\sinh c}$.
\item[Law of cosines for sides]
$\cosh a=\cosh b\cosh c-\sinh b\sinh c\cos\alpha$.
\item[Law of cosines for angles]
$\cos \alpha=-\cos \beta\cos \gamma+\sin \beta\sin \gamma\cosh a$.
\end{description}
\end{lemma}
\noindent{\bf Remark.} The Area formula and the  law of cosines for Angles also hold if $A$ is an ideal point, and hence the sides $BA$ and $CA$ of the triangle are parallel half-lines and $\alpha=0$.\\

The area of a circular disc can be easily expressed.

\begin{lemma}
\label{circularDisk}
If $r>0$ and $p\in \hyp^2$, then
$$
V(B(p,r))=2\pi(\cosh r-1),
$$
and hence $V(B(p,r))<2\pi r^2$ if
$r\in(0,2)$.

\end{lemma}

\subsection{Some basic properties of horoballs and horospheres}

In this section, we survey some properties of horoballs and horospheres that are frequently used throughout the paper (cf. Ratcliffe \cite{Rat19} Chapters~3 and 4).  The first two properties readily follow from the representation of horoballs in the Poincar\'e ball model in ${\rm int}B^n$ as Euclidean spheres touching $\partial B^n$.

\begin{lemma}
\label{horocycle-intersection} 
If $p,q\in \hyp^2$, $p\neq q$, then there exist exactly two horocyclic arcs connecting $p$ and $q$, and their ideal points correspond to the perpendicular bisector of $[p,q]$.
\end{lemma}

\begin{lemma}
\label{HoroballsIntersect}
If $\Xi_1,\Xi_2\subset \hyp^n$ are horoballs corresponding to different ideal points, and the interiors of $\Xi_1$ and $\Xi_2$ intersect, then $\Xi_1\cap\Xi_2$ is compact, and $\partial \Xi_1\cap \partial \Xi_2$ is an $(n-2)$-dimensional sphere. 
In particular, if $n=2$, then the two horocycles intersect in two points $p$ and $q$, and the two horocyclic arcs between $p$ and $q$ bound
$\Xi_1\cap\Xi_2$.
\end{lemma}

Intersection patterns of horospheres and hyperspheres are important for our paper. Since in the Poincar\'e ball model a hypersphere or a horosphere $X$ is a subset of a Euclidean $(n-1)$-sphere $\Sigma$, it makes sense to speak about an open spherical cap on $X$ that is of the form $X\cap {\rm int}\,G\neq\emptyset$ where $G$ is a Euclidean $n$-ball such that $G\cap{\rm cl}X\subset {\rm int}B^n$.
Lemma~\ref{horo-hyper-sphere}
and
Lemma~\ref{horo-sphere-sphere}
below
follow from the intersection patterns of Euclidean spheres used in the Poincar\'e ball model.
For Lemma~\ref{horo-hyper-sphere}, we observe that in the Poincar\'e ball model, a Euclidean sphere $\Sigma$ lies in $B^n$ if $\Sigma$ represents a horosphere, and has points outside of $ B^n$, if  $\Sigma$ represents a hypersphere.

\begin{lemma}
\label{horo-hyper-sphere}
Let $X\subset  \hyp^n$ be a hypersphere and $\Xi \subset  \hyp^n$ be a horoball. If $X$ has a common point $p$ with the horosphere $\partial \Xi$, then either they are tangent at $p$ and $X \cap \Xi=\{p\}$, or ${\rm int}\,\Xi$ intersects $X$ in an open spherical cap.

In particular, if $n=2$ and the hypercycle $X$ and the horocycle $\partial \Xi$  are not tangent at the intersection point $p$, then  $X \cap {\rm int}\Xi$ is an open hypercycle arc emanating from $p$ whose other endpoint is either a point of $X$ or the ideal point of $\Xi$.
\end{lemma}

For Lemma~\ref{horo-sphere-sphere}, we observe that in the Poincar\'e ball model, a Euclidean sphere $\Sigma$ lies in ${\rm int}B^n$ if $\Sigma$ represents a hyperbolic sphere, and has a point in $\partial B^n$ (the ideal point), if  $\Sigma$ represents a horosphere.

\begin{lemma}
\label{horo-sphere-sphere}
Let $B(z,r)\subset  \hyp^n$ be a hyperbolic ball, $r>0$, and let and $\Xi \subset  \hyp^n$ be a horoball intersecting
${\rm int}B(z,r)$. If $\partial B(z,r)$ has a common point $p$ with $\partial \Xi$, then either they are tangent at $p$ with $B(z,r)\subset \Xi$ and
$\partial B(z,r)\cap \partial\Xi=\{p\}$, or ${\rm int}\,B(z,r)$ intersects $\partial \Xi$ in an open spherical cap.

In particular, if $n=2$ and the circle $\partial B(z,r)$ and the horocycle $\partial \Xi$  are not tangent at the intersection point $p$, then  they intersect in another point $q$, 
and the open arc of $\partial \Xi$ between $p$ and $q$ is contained in ${\rm int}B(z,r)$ and intersects the perpendicular bisector of $[p,q]$.
\end{lemma}

Lemma~\ref{horo-hyper-sphere}
and Lemma~\ref{horo-sphere-sphere} also follow from the curvature of the corresponding curves, as the geodesic curvature at each point is $1$ for a horocycle, $1/\tanh r>1$ for a hyperbolic circle of radius $r>0$, and $\tanh r<1$ for a hypercycle whose points are of distance $r$ from the corresponding line.

A fundamental property of horospheres is that they are symmetric through the perperpendicular bisector of any secant.

\begin{lemma}
\label{HoroballSymmetry} 
Let $\Xi\subset \hyp^n$ be a horoball with the ideal point $i$.
\begin{description}
\item[(a)] If $x,y\in\partial\Xi$ with $x\neq y$, and $H\subset \hyp^n$ is the hyperplane through the midpoint $p$ of $[x,y]$ and $H$ is perpendicular to $[x,y]$, then $\Xi$ is symmetric through $H$ and $i$ is an ideal point of $H$.

\item[(b)] If $i$ is an ideal point of a hyperplane $H\subset \hyp^n$, then
$\Xi$ is symmetric through $H$.
\end{description}
\end{lemma}
\begin{proof}
We use the Poincar\'e ball model.
For both (a) and (b), we
may assume that $H$ is part of a Euclidean hyperplane $\widetilde{H}$ containing $o$ by Lemma~\ref{HypIsometries}.

For (b), since $i\in \widetilde{H}$ and $\Xi$ as a Euclidean ball touches $B^n$ at $i$, the Euclidean center of $\Xi$ is contained in the line passing through $o$ and $i$; therefore, $\Xi$ is symmetric through 
$\widetilde{H}$.

For (a), we may also assume that $p=o$, and hence $[x,y]$ is a Euclidean segment whose Euclidean perpendicular bisector is $\widetilde{H}$. In this case $\Xi$ as a Euclidean ball is symmetric through $\widetilde{H}$, and hence
$i\in\widetilde{H}$.
\end{proof}

We deduce from Lemma~\ref{horo-hyper-sphere} the following extremal property of horospheres.

\begin{lemma}
\label{HoroballHyperplane}
For a half-space $H^+\subset \hyp^n$ bounded by the hyperplane $H$ and $y\in H^+\backslash H$, let $\Xi$ be the horoball with $y\in\partial \Xi$ such that the center of $\Xi$ is the ideal point of the half-line emanating from $y$ and orthogonal to $H$. Then the unique farthest point of $H^+\cap \Xi$ from $H$ is $y$.
\end{lemma}

\subsection{Convex sets in the hyperbolic space}
\label{secHypConvex}

We say that an $X\subset \hyp^n$ is convex if $[x,y]\subset X$ for any $x,y\in X$. Readily, the intersection of convex sets in $\hyp^n$ is convex; therefore, for any $X\subset \hyp^n$, we can consider its {\it convex hull} ${\rm conv}X$, that is the minimal convex set containing $X$. A compact convex set $X\subset \hyp^n$ with non-empty interior is called a {\it convex body}

Lines and hyperplanes are convex by definition, and let us see some full dimensional examples.

\begin{lemma}
\label{convex-examples}
The following closed subsets of $\hyp^n$ are convex.
\begin{description}
\item[(a)]  Half-spaces.
\item[(b)] Horoballs.
\item[(c)] For a hypersphere $X$ corresponding to a hyperplane $H$, the part of $\hyp^n$ bounded by $X$ and containing $H$.
\end{description}
\end{lemma}
\begin{proof}
Let $Z$ be either of the sets in (a), (b) or (c). It is sufficient to prove that if $x\in{\rm int}Z$ and $y\in Z$, then $[x,y]\subset Z$. We may assume after an isometry that $x=o$ in the Poincar\'e ball model by Lemma~\ref{HypIsometries}, and 
hence $[x,y]$ is a Euclidean segment in the Poincar\'e ball model.
Since $Z=G\cap{\rm int}B^n$ in this case for a Euclidean $n$-ball $G$, we conclude that $[x,y]\subset Z$. 
\end{proof}

 A characteristic property of convex sets is the unique closest point to an exterior point, and the existence of supporting hyperplane at boundary points. We say that a hyperplane 
$H\subset \hyp^n$ separates $X,Y\subset \hyp^n$ if $X$ and $Y$ lie in different closed half-spaces bounded by $H$.

\begin{lemma}
\label{supporting-hyperplane-convex}
Let $K\subset \hyp^n$ be a closed convex set.
\begin{description}
\item[(a)]  For $p\in \hyp^n\backslash K$, there exists a unique $z\in K$ closest to $p$, and the hyperplane $H\subset \hyp^n$ passing through $z$ and orthogonal to $[p,z]$ separates $p$ and $K$.
\item[(b)] For any $z\in\partial K$, there exists a so-called supporting hyperplane $H\subset \hyp^n$ to $K$ containing $z$; namely, $K$ lies in one of the closed half-spaces bounded by $H$.
\item[(c)] For any $z\in\partial K$ and supporting hyperplane $H$ to $K$ at $z$, if $K\subset H^+$ for a closed half-space $H^+$ bounded by $H$, and $p\in\ell^-$ for the half-line $\ell^-$ emanating from $z$, orthogonal to $H$ and not contained in $H^+$, then $z$ is the closest point of $K$ to $p$. 
\end{description}
\end{lemma}
\begin{proof}
For (a), let $r=d(p,z)$, and hence $K\cap {\rm int}B(p,r)=\emptyset$. Using the Poincar\'e ball model, we may assume that $z=o$ by Lemma~\ref{HypIsometries}, and hence the symmetries of the Poincar\'e ball model yield that $B(p,r)$ is a Euclidean ball whose center lies on the line of $z$ and $p$. Let $\widetilde{H}$ be the Euclidean hyperplane tangent to $B(p,r)$ at $z=o$, thus it is orthogonal to 
$[p,o]=[p,z]$. For any $x\in K$, the hyperbolic segment $[x,o]$, that is a Euclidean segment, as well, avoids ${\rm int}B(p,r)$, and hence it is separated from $p$ by $\widetilde{H}$. Thus $H=\widetilde{H}\cap {\rm int}B^n$. 

For (b), let $p_m\in \hyp^n\backslash K$ be a sequence of points tending to $z$, and let $z_m\in K$ be the closest point of $K$ to $p_m$, and hence $z_m$ also tends to $z$. According to (a), there exists a (hyperbolic) half-space $H_m^+\supset K$ with $z_m\in\partial H_m^+$. Taking a convergent subsequence of $\{H_m^+\}$ leads to the supporting hyperplane $H$ at $z$.

For (c), $H$ is tangent  to the ball $B(p,r)$ (is a common supporting hyperplane separating $B(p,r)$ and $K$) where $r=d(p,z)$.
\end{proof}

\begin{corollary}
\label{closed-half-space-intersection}
A closed set $X\subset \hyp^n$, $X\neq \hyp^n$, is convex if and only if it is the intersection of closed half-spaces. 
\end{corollary}

The following properties follow from using the  Bertrami--Cayley--Klein model of the hyperbolic space (where the hyperbolic universe is still ${\rm int}B^n$, but hyperbolic convex sets coincide with the Euclidean ones),
and properties of convex sets in $\R^n$.

\begin{lemma}[Carath\'eodory]
\label{Caratheodory}
Let $X\subset \hyp^n$ be compact.
\begin{description}
\item[(a)] The convex hull of $X$ is compact.
\item[(b)] If $z$  lies in the convex hull of $X$, then $z$  lies in the convex hull of some $x_1,\ldots,x_k\in X$
for $k\leq n+1$.
\end{description}
\end{lemma}

\subsection{h-convex (horocyclically convex) sets in the hyperbolic space}
\label{sech-convex}

We say that a hyperplane, a hypersphere or a horosphere \emph{supports} a convex body $K$ if it intersects $K$, and $K$ is contained in one of the two closed regions of $\hyp^n$ bounded by the corresponding hypersurface.
The following definition---that is a core notion of our paper---was introduced by Santal\'o \cite{San68} in 1968.

\begin{de}[h-convex sets]
An $X\subset \hyp^n$ is h-convex if for any $x,y\in X$, $x\neq y$, $\sigma\subset X$ holds for any horocyclic arc $\sigma$ connecting $x$ and $y$ (cf. Lemma~\ref{horocycle-intersection}).
\end{de}
\noindent{\bf Remark.}
These sets are also called horocyclically convex or horoconvex.

\begin{lemma} \mbox{ }
\begin{description}
\item[(i)] Arbitrary intersection of h-convex sets is h-convex,
\item[(ii)]  h-convex sets are convex,
\item[(iii)] horoballs are h-convex, and hence convex.
\end{description}
\end{lemma}
\begin{proof} (i) directly follows from the definition. For (ii) and (iii), we use the Poincar\'e ball model.

For (ii), let $X\subset {\rm int}B^n$ be h-convex, and let $x,y\in X$, $x\neq y$. We may assume that the hyperbolic midpoint of the hyperbolic segment $s$ connecting $x$ and $y$ is $o$ by Lemma~\ref{HypIsometries}, and hence $s$ is a Euclidean segment.  
Let $\Pi$ be any hyperbolic two-plane containing $x$ and $y$, and let $\sigma,\sigma'\subset \Pi$ be the two horocyclic arcs in $\Pi$ connecting $x$ and $y$. As any horocyclic arc connecting a point of $\sigma$ and a point of $\sigma'$ is part of $X$, $X$ contains the intersection of the two horoballs whose boundaries contain $\sigma$ and $\sigma'$.   In particular, $s\subset X$.

For (iii), let $x,y\in \Xi$, $x\neq y$, for a horoball  $\Xi\subset {\rm int}B^n$, and let $\sigma$ be a horocyclic arc connecting $x$ and $y$. Let $\tilde{\sigma}$ be the Euclidean circle containing $\sigma$. Since the Euclidean ball $\Xi$ contains $x$ and $y$, either $\sigma\subset\Xi$, or $\sigma$ contains the arc of $\tilde{\sigma}$ lying outside of $\Xi$. However, in the second case, $\sigma$ would contain its ideal point, which is a contradiction; therefore, $\sigma\subset\Xi$.
\end{proof}

\begin{lemma}
\label{HoroballsClosestPoint}
If $X\subset \hyp^n$ is h-convex and closed and $y\not \in X$, then for the unique point $z\in X$ closest to $y$, $X\subset \Xi$ for the horoball $\Xi$ such that $z\in \partial \Xi$ and the center of $\Xi$ is the ideal point of the half-line emanating from $y$ and passing through $z$.
\end{lemma}
\begin{proof} The proof is indirect; therefore, we suppose that there exists an $x_0\in X\backslash \Xi$. Let $\Pi$ be the two-plane containing $x_0$, $y$, $z$, and let $\sigma\subset\Pi$ be the horocyclic arc connecting $x_0$ and $z$ such that $\sigma\cap\Xi=\{z\}$. It follows that the tangent line to $\sigma$ at $z$ is different from $\Pi\cap H$, where $H$ denotes the supporting hyperplane of $X$ at $z$ orthogonal to $[z,y]$. Therefore, there exists an $x\in \sigma$, and hence $x\in X$ such that it lies in the same open half-space bounded by $H$ as $y$, this $\varphi=\angle(y,z,x)<\frac{\pi}2$. Setting $t_0=d(x,z)$, for any $t\in(0,t_0)$, there exists a $w_t$ in the geodesic segment between $x$ and $z$ with $d(w_t,z)=t$. As $X$ is convex, we have $w_t\in X$, and the law of cosines for sides (cf. Lemma~\ref{triangle}) yields that
\begin{equation}
\label{wtz}
\cosh{d(w_t,y)}=\cosh{t}\cosh{d(z,y)}-\sinh{t}\sinh{d(z,y)}\cdot \cos\varphi.
\end{equation}
Differentiating the right hand side of \eqref{wtz} with respect to $t$ implies that $d(w_t,y)<d(z,y)$ for small $t>0$, which is a contradiction proving 
Lemma~\ref{HoroballsClosestPoint}.
\end{proof}

Lemma~\ref{HoroballsClosestPoint} yields the following statements:

\begin{corollary}
\label{HoroballIntersection}
A closed set $K\subset \hyp^n$ is h-convex if and only if $K$ is the intersection of horoballs.   
\end{corollary}
 
\begin{corollary}
\label{HoroballSupporting}
Let $K\subset \hyp^n$ be h-convex, and let $H$ be a supporting hyperplane at $z\in\partial K$. 
\begin{description}
\item[(i)] There exists a horoball $\Xi\supset K$ such that $H$ is a supporting hyperplane to $\Xi$.
\item[(ii)] $H\cap \Xi=\{z\}$.
\end{description}
\end{corollary}
\begin{proof} (ii) directly follows from (i).
For (i), take a point $y\neq z$ that is separated from $K$ by $H$ and the segment connecting $y$ and $z$ is orthogonal to $H$. Then $z$ is the closest point of $K$ to $y$, and hence Lemma~\ref{HoroballsClosestPoint} yields (i).
\end{proof}

A compact h-convex set $K\subset \hyp^n$ is called an {\it h-convex body} if it is neither empty nor a singleton; or equivalently, if it has non-empty interior. The largest radius of a ball contained in $K$ is the {\rm inradius} $r(K)$.

\begin{lemma}[Inscribed ball of an h-convex body]
\label{inscribed-ball}
Let $K\subset \hyp^n$ be an h-convex body.
\begin{description}
\item[(a)] There exists a unique ball of radius $r(K)$ contained in $K$, the so-called inscribed ball.

\item[(b)] A ball
$B(p,r)\subset K$ is the inscribed ball of $K$ if and only if there exists $x_1,\ldots,x_k\in\partial B(p,r)\cap \partial K$ such that $k\leq n+1$ and $p$ lies in the convex hull of
$x_1,\ldots,x_k$.
\end{description}
\end{lemma}
\begin{proof} We note that an inscribed ball exists by the compactness of $K$. The uniqueness part of (a) is proved by contradiction, so we suppose that there exists $q\neq p$ such that $B(p,r),B(q,r)\subset K$ for $r=r(K)$. Let $\ell$ be the line through $p$ and $q$, and let $m$ be the midpoint of $[p,q]$. 

First let $n=2$. We write $\ell'\subset \hyp^2$ to denote the line passing through $m$ and orthogonal to $\ell$, and $i_1$ and $i_2$ to denote the ideal points of $\ell'$. For $j=1,2$, there exists a horoball $\Xi_j$ with ideal point $i_j$ containing $B(p,r)$ and $B(q,r)$ such that $\partial \Xi_j$ is tangent to $B(p,r)$ at $p_j$ and $B(q,r)$ at $q_j$ by the symmetry of horocycles (cf. Lemma~\ref{HoroballSymmetry}) and the intersection pattern of horoballs and spheres (cf. Lemma~\ref{horo-sphere-sphere}). As $K$ is h-convex, the horocyclic arc of $\partial \Xi_j$ between $p_j$ and $q_j$ is contained in $K$. As $K$ is convex, Lemma~\ref{horo-sphere-sphere} also yields that $B(m,\varrho)\subset K$ where $\varrho$ is the common distance of $\ell'\cap \partial\Xi_1$
and $\ell'\cap \partial\Xi_2$ from $m$ (and hence from $\ell$). However, $\varrho>r$ by 
Lemma~\ref{HoroballHyperplane},
which is absurd, proving (a) if $n=2$.

If $n\geq 3$, then (a) follows 
by assuming $m=o$ in the  Poincar\'e ball model, and applying
the two-dimensional case for $\Pi\cap K$ where $\Pi$ is any Euclidean two-plane containing $\ell$.

For (b), first we assume that 
$B(p,r)\subset K$ is the inscribed ball of $K$. The maximality of $r$ yields that $\partial B(p,r)\cap\partial K\neq \emptyset$. Here the convex hull
$C$ of $\partial B(p,r)\cap\partial K$ is compact according to Lemma~\ref{Caratheodory} (a).
We claim that 
\begin{equation}
\label{p-in-convex-hull}
p\in C.
\end{equation}
We suppose that 
$p\not\in C$, and seek a  contradiction. For the closest point $z$ of $C$, the hyperplane $\widetilde{H}$ containing $z$ and orthogonal to $[p,z]$ separates $p$ and $\partial B(p,r)\cap\partial K$
according to Lemma~\ref{supporting-hyperplane-convex}. Let $H$ be the hyperplane  containing $p$ and orthogonal to $[p,z]$, and let $H^+$ be the half-space bounded by $H$ and not containing $z$. It follows that $H^+\cap C=\emptyset$, and hence there exists $\varrho>r$ such that
$d(p,x)\geq \varrho$ for any 
$x\in \partial K\cap H^+$. Choose 
a $q\in \partial K\cap {\rm int}H^+$ that lies on the line of $p$ and $z$---that is orthogonal to $H$---and $d(q,p)<\varrho-r$. It follows from the triangle inequality that 
$d(q,x)>r$ for any 
$x\in \partial K\cap H^+$. On the other hand, if $x\in \partial K\backslash H^+$, then $\angle(q,p,x)>\frac{\pi}2$, and the law of cosines for sides (cf. Lemma~\ref{triangle}) yields that
$d(q,x)>d(p,x)\geq r$. Therefore, $B(q,r)\subset {\rm int} K$, which contradicts the maximality of $r$, and proves
\eqref{p-in-convex-hull}.
In turn, combining \eqref{p-in-convex-hull} and Lemma~\ref{Caratheodory} (b) implies the existence of $x_1,\ldots,x_k\in C$, $k\leq n+1$, whose convex hull contains $p$.

Finally, we assume that 
$B(p,r)\subset K$, and  $p$ lies in the convex hull of
$x_1,\ldots,x_k$ for some
$x_1,\ldots,x_k\in\partial B(p,r)\cap \partial K$.
We suppose that $B(p,r)$ is not the inscribed ball of $K$, and hence there exists $B(q,r)\subset K$ with $q\neq p$, and seek a contradiction. For the half-space $H^+$ that contains $q$ and whose bounding hyperplane $H$ contains $p$ and is orthogonal to $[p,q]$, we claim that
\begin{equation}
\label{hemisphere-in-interior}
H^+\cap B(p,r)\subset {\rm int}K.
\end{equation}
If $n=2$ and $z\in \partial B(p,r)\cap \partial K$, then there exists a horoball $\Xi\supset K\subset B(p,r)$ such that $z\in \partial \Xi$ by Corollary~\ref{HoroballSupporting}. For the secant $s$ of $\Xi$ orthogonal to $[z,p]$ and passing through $p$, Lemma~\ref{HoroballHyperplane} yields that
$z$ is the farthest point from $s$ of the cap of $\Xi$ cut off by $s$. As
$B(q,r)\subset\Xi$, we deduce that
$\angle(z,p,q)>\frac{\pi}2$, yielding \eqref{hemisphere-in-interior}. If $n\geq 3$, then we use rotational symmetry of $B(p,r)\cup B(q,r)$ around the line passing through $p$ and $q$
as in (a) to verify \eqref{hemisphere-in-interior}.

Now $p$ lies in the convex hull of $x_1,\ldots,x_k$; therefore,
there exists an $x_i\in H^+$. This is absurd by
\eqref{hemisphere-in-interior}, completing the proof of 
Lemma~\ref{inscribed-ball}.
\end{proof}

Lemma~\ref{inscribed-ball} (a) does not hold if $K$ is only assumed to be convex. For example, we can take $K$ to be the convex hull of two balls of radius $r$.

\subsection{Lassak width in the hyperbolic space}\label{secLassak-width}

In $\R^n$, the width function of a convex body with respect to a certain direction is the distance of the two parallel supporting hyperplanes orthogonal to this direction. In the hyperbolic space, one can interpret the Euclidean concept in many ways as the distance of a pair of supporting hypersurfaces (see Fillmore \cite{Fil70}, Santal\'o \cite{S04}, Leichtweiss \cite{Lei05}, Jer\'onimo-Castro--Jimenez-Lopez \cite{JCJL17}, G. Horv\'ath \cite{Hor21}, B\"or\"oczky--Cs\'epai--Sagmeister \cite{BoCsS}, Lassak \cite{Las23}), and they all give different width functions. For a comparison of width functions, see G. Horv\'ath \cite{Hor21} and B\"or\"oczky, Cs\'epai, Sagmeister \cite{BoCsS}. It is verified in \cite{BoCsS} that all width functions are continuous, and their maximal value equals the diameter of the respective convex body. We also expect the minimal width to be a monotone function with respect to containment, and also we expect the minimal width of a lower dimensional convex body to be zero. 
Hence, in this paper, we consider the definition of width recently given by Lassak \cite{Las23}:

\begin{de}[Lassak width]
Let $K\subset \hyp^n$ be a convex body.
\begin{description}
\item[Lassak width with respect to a supporting hyperplane] For a supporting hyperplane $H\subset \hyp^n$, the Lassak width  with respect to $H$ is the maximal distance of the points of $K$ from $H$.
\item[Minimal Lassak width] The minimal Lassak width $w(K)$ of $K$ is the minimum of Lassak widths with respect to supporting hyperplanes, whose minimum is achieved by Lemma~\ref{space-hyperplanes-compact}.
\end{description}
\end{de}
\noindent{\bf Remark.} For a supporting hyperplane $H$, $\varrho$ is the maximal distance of the points of $K$ from $H$ if and only if $K$ lies in the convex ``strip'' bounded by $H$ and the hypersphere $X$ of distance $\varrho$ from $H$ where $X$ is a supporting hypersphere (see Lemma~\ref{convex-examples} for the convexity of the strip). In particular, Lemma~\ref{space-hyperplanes-compact} yields that
\begin{align*}
w(K)=&\mbox{minimal width of a strip that contains $K$,}\\
&\mbox{and is bounded by a hypersphere and the corresponding hyperplane,}
\end{align*}
 and the width of a strip is the common length of the secants orthogonal to the boundary of the strip.\\

We may extend the notion of minimal Lassak width  to a compact convex set $K\subset \hyp^n$ contained in a hyperplane by setting
$w(K)=0$. Therefore, a simple argument based on Lemma~\ref{space-hyperplanes-compact} yields the following. 

\begin{lemma}[Lassak \cite{Las23}]\label{lemma:Lassakwidth:continuity-monotonicity}\mbox{ }
\begin{itemize}
\item 
The minimal Lassak width $w(\,\cdot\,)$ is a continuous function on the space of convex bodies in $\hyp^n$.
\item If $K_1\subset K_2\subset \hyp^n$ are compact convex sets, then
$w(K_1)\leq w(K_2)$.
\end{itemize}
\end{lemma}

We observe that any supporting hyperplane to a closed h-convex set intersects the set in a single point. Therefore we have
Proposition~\ref{width-opposite}, due to Lassak \cite{Las23}, that is useful in order to determine the Lassak  width of an h-convex body.

\begin{prop}[Lassak \cite{Las23}]
\label{width-opposite}
Let $K\subset \hyp^n$ be an h-convex body.
 If $H$ is a supporting hyperplane such that the maximal distance of points of $K$ from $H$ is $w(K)$, and $X$ is the supporting hypersphere of $K$ corresponding to $H$, then  $X\cap K$ and $H\cap K$ are unique points that lie on a line 
 orthogonal to $X$ and $H$,
and hence the distance of these points is $w(K)$.
\end{prop}

Proposition~\ref{width-opposite} does not hold if we assume only that $K$ is convex.

\begin{example}[Convex body $K\subset \hyp^2$ where Proposition~\ref{width-opposite} does not hold
(see Figure~1)]
\label{example:noProp2.27}
For small $\varepsilon>0$, we will define $K$ 
be  the convex hull of 
$p_1,p_2,q_1,q_2\in \hyp^2$ and the hypercycle arc $\sigma\subset \hyp^2$ where
\begin{itemize}
\item $p_1,p_2\in[q'_1,q'_2]$ in a way such that $d(p_j,q'_j)=\varepsilon$, $j=1,2$ and   $d(p_1,p_2)=1/\varepsilon$,
\item $q_1,q_2$ lie on the same side of of the line $\ell$ of $[q'_1,q'_2]$, and are of distance $\varepsilon$ from $\ell$, and the orthogonal projection of $q_j$ into $\ell$ is $q'_j$.
\end{itemize}  
Let $\sigma'$ be the hypercycle arc of distance $\varepsilon$ from $\ell$ passing through $q_1$ and $q_2$, and let $\ell_0$ be the perpendicular bisector of the segment $[p_1,p_2]$ (and hence of $[q'_1,q'_2]$). For small  $\eta\in(0,\varepsilon)$, we consider the line $\ell_\eta$ orthogonal to $\ell_0$, of distance $\eta$ from $\ell$, and separating $p_1,p_2$ on the one hand from $q_1,q_2$ on the other hand ($\eta$ must be small enough for this property). Now, for any line tangent to $\sigma'$, the distance of either $p_1$ or $p_2$ from this tangent line is strictly larger than $\varepsilon$. Therefore, we can choose $\eta$ small enough in a way such that if $\sigma_\eta$ is the hypercycle arc corresponding to $\ell_\eta$ and passing through $q_1$ and $q_2$, then   for any line tangent to $\sigma_\eta$, the distance of either $p_1$ or $p_2$ from this tangent line is strictly larger than $\varepsilon$. 

We define $K$ be the convex hull of 
$p_1,p_2,q_1,q_2$ and $\sigma=\sigma_\eta$. Then the Lassak width of $K$ is $\varepsilon$, realized by the distance of $q_i$ and the support line $\ell$ for $i=1,2$, but the orthogonal projections of $q_1$ and $q_2$ do not lie in $K$.  

\end{example}

\begin{figure}[H]
\centering
\begin{tikzpicture}[line cap=round,line join=round,>=triangle 45,x=1cm,y=1cm,scale=5]
\clip(-1.1,-1.1) rectangle (1.1,1.1);
\draw [line width=0.8pt,dash pattern=on 1pt off 1pt] (0,0) circle (1cm);
\draw [shift={(1.127331490004463,-0.21120874058422445)},line width=2pt]  plot[domain=2.3753297080004594:2.7560847857514994,variable=\t]({1*0.5616808884543404*cos(\t r)+0*0.5616808884543404*sin(\t r)},{0*0.5616808884543404*cos(\t r)+1*0.5616808884543404*sin(\t r)});
\draw [shift={(0,-1.2508175819708593)},line width=0.8pt,dotted]  plot[domain=0.8963743137439145:2.2452183398458785,variable=\t]({1*1.6014195650632685*cos(\t r)+0*1.6014195650632685*sin(\t r)},{0*1.6014195650632685*cos(\t r)+1*1.6014195650632685*sin(\t r)});
\draw [shift={(0,4.358101402254348)},line width=0.8pt,dash pattern=on 1pt off 1pt on 1pt off 4pt] plot[domain=4.5411940487716:4.883583911997778,variable=\t]({1*4.24182128717504*cos(\t r)+0*4.24182128717504*sin(\t r)},{0*4.24182128717504*cos(\t r)+1*4.24182128717504*sin(\t r)});
\draw [line width=0.8pt,dash pattern=on 1pt off 1pt on 1pt off 4pt] (-0.6801232890394878,0)-- (-0.6068737219236293,0);
\draw [line width=2pt] (-0.6068737219236293,0)-- (0.6068737219236293,0);
\draw [line width=0.8pt,dash pattern=on 1pt off 1pt on 1pt off 4pt] (0.6068737219236293,0)-- (0.6801232890394878,0);
\draw [shift={(1.0752224720604844,0)},line width=0.8pt,dash pattern=on 1pt off 1pt on 1pt off 4pt] plot[domain=2.6734300154997994:3.141592653589793,variable=\t]({1*0.3950991830209966*cos(\t r)+0*0.3950991830209966*sin(\t r)},{0*0.3950991830209966*cos(\t r)+1*0.3950991830209966*sin(\t r)});
\draw [shift={(-1.0752224720604844,0)},line width=0.8pt,dash pattern=on 1pt off 1pt on 1pt off 4pt] plot[domain=2.6734300154997994:3.141592653589793,variable=\t]({-1*0.3950991830209966*cos(\t r)+0*0.3950991830209966*sin(\t r)},{0*0.3950991830209966*cos(\t r)+1*0.3950991830209966*sin(\t r)});
\draw [shift={(-1.127331490004463,-0.2112087405842243)},line width=2pt]  plot[domain=2.3753297080004594:2.7560847857514994,variable=\t]({-1*0.5616808884543404*cos(\t r)+0*0.5616808884543404*sin(\t r)},{0*0.5616808884543404*cos(\t r)+1*0.5616808884543404*sin(\t r)});
\draw [line width=0.8pt,dash pattern=on 1pt off 1pt on 1pt off 4pt] (-1,0)-- (-0.6801232890394878,0);
\draw [line width=0.8pt,dash pattern=on 1pt off 1pt on 1pt off 4pt] (0.6801232890394878,0)-- (1,0);
\draw [shift={(1.0752224720604844,0)},line width=0.8pt,dash pattern=on 1pt off 1pt on 1pt off 4pt] plot[domain=1.9470707351455765:2.6734300154997994,variable=\t]({1*0.39509918302099656*cos(\t r)+0*0.39509918302099656*sin(\t r)},{0*0.39509918302099656*cos(\t r)+1*0.39509918302099656*sin(\t r)});
\draw [shift={(1.0752224720604844,0)},line width=0.8pt,dash pattern=on 1pt off 1pt on 1pt off 4pt] plot[domain=3.141592653589793:4.33611457203401,variable=\t]({1*0.39509918302099656*cos(\t r)+0*0.39509918302099656*sin(\t r)},{0*0.39509918302099656*cos(\t r)+1*0.39509918302099656*sin(\t r)});
\draw [shift={(0,4.358101402254348)},line width=0.8pt,dash pattern=on 1pt off 1pt on 1pt off 4pt] plot[domain=4.480868482209441:4.5411940487716,variable=\t]({1*4.24182128717504*cos(\t r)+0*4.24182128717504*sin(\t r)},{0*4.24182128717504*cos(\t r)+1*4.24182128717504*sin(\t r)});
\draw [shift={(0,4.358101402254348)},line width=0.8pt,dash pattern=on 1pt off 1pt on 1pt off 4pt] plot[domain=4.883583911997778:4.943909478559938,variable=\t]({1*4.24182128717504*cos(\t r)+0*4.24182128717504*sin(\t r)},{0*4.24182128717504*cos(\t r)+1*4.24182128717504*sin(\t r)});
\draw [shift={(1.127331490004463,-0.21120874058422445)},line width=0.8pt]  plot[domain=1.8973575655729522:2.3753297080004594,variable=\t]({1*0.5616808884543406*cos(\t r)+0*0.5616808884543406*sin(\t r)},{0*0.5616808884543406*cos(\t r)+1*0.5616808884543406*sin(\t r)});
\draw [shift={(1.127331490004463,-0.21120874058422445)},line width=0.8pt,dash pattern=on 1pt off 1pt on 1pt off 4pt] plot[domain=2.7560847857514994:4.015416200309656,variable=\t]({1*0.5616808884543405*cos(\t r)+0*0.5616808884543405*sin(\t r)},{0*0.5616808884543405*cos(\t r)+1*0.5616808884543405*sin(\t r)});
\draw [shift={(-1.127331490004463,-0.2112087405842243)},line width=0.8pt,dash pattern=on 1pt off 1pt on 1pt off 4pt] plot[domain=2.7560847857514994:4.015416200309656,variable=\t]({-1*0.5616808884543405*cos(\t r)+0*0.5616808884543405*sin(\t r)},{0*0.5616808884543405*cos(\t r)+1*0.5616808884543405*sin(\t r)});
\draw [shift={(-1.0752224720604844,0)},line width=0.8pt,dash pattern=on 1pt off 1pt on 1pt off 4pt] plot[domain=1.9470707351455765:2.6734300154997994,variable=\t]({-1*0.39509918302099656*cos(\t r)+0*0.39509918302099656*sin(\t r)},{0*0.39509918302099656*cos(\t r)+1*0.39509918302099656*sin(\t r)});
\draw [shift={(-1.127331490004463,-0.2112087405842243)},line width=0.8pt,dash pattern=on 1pt off 1pt on 1pt off 4pt] plot[domain=1.8973575655729522:2.3753297080004594,variable=\t]({-1*0.5616808884543406*cos(\t r)+0*0.5616808884543406*sin(\t r)},{0*0.5616808884543406*cos(\t r)+1*0.5616808884543406*sin(\t r)});
\draw [shift={(-1.0752224720604844,0)},line width=0.8pt,dash pattern=on 1pt off 1pt on 1pt off 4pt] plot[domain=3.141592653589793:4.33611457203401,variable=\t]({-1*0.39509918302099656*cos(\t r)+0*0.39509918302099656*sin(\t r)},{0*0.39509918302099656*cos(\t r)+1*0.39509918302099656*sin(\t r)});
\draw [shift={(0,-2.0784204926954546)},line width=2pt]  plot[domain=1.2608964946927324:1.880696158897061,variable=\t]({1*2.3695852400853465*cos(\t r)+0*2.3695852400853465*sin(\t r)},{0*2.3695852400853465*cos(\t r)+1*2.3695852400853465*sin(\t r)});
\draw (0.64,0.28) node[anchor=north west] {$q_2$};
\draw (0.68,0.01) node[anchor=north west] {$q_2'$};
\draw (0.53,0) node[anchor=north] {$p_2$};
\draw (-0.53,0) node[anchor=north] {$p_1$};
\draw (-0.68,0.01) node[anchor=north east] {$q_1'$};
\draw (-0.64,0.28) node[anchor=north east] {$q_1$};
\draw (0,0) node[anchor=north] {$\ell$};
\draw (0,0.46) node[anchor=north] {$\sigma'$};
\draw [shift={(0,-2.0784204926954546)},line width=0.8pt,dotted]  plot[domain=1.1363039589642288:1.2608964946927324,variable=\t]({1*2.3695852400853465*cos(\t r)+0*2.3695852400853465*sin(\t r)},{0*2.3695852400853465*cos(\t r)+1*2.3695852400853465*sin(\t r)});
\draw [shift={(0,-2.0784204926954546)},line width=0.8pt,dotted]  plot[domain=1.880696158897061:2.0052886946255644,variable=\t]({1*2.3695852400853465*cos(\t r)+0*2.3695852400853465*sin(\t r)},{0*2.3695852400853465*cos(\t r)+1*2.3695852400853465*sin(\t r)});
\draw [shift={(0,14.086093714285717)},line width=0.8pt,dash pattern=on 1pt off 1pt on 1pt off 4pt] plot[domain=4.641337210571465:4.783440750197913,variable=\t]({1*14.050552876226599*cos(\t r)+0*14.050552876226599*sin(\t r)},{0*14.050552876226599*cos(\t r)+1*14.050552876226599*sin(\t r)});
\draw (0,0.3) node[anchor=north] {$\sigma$};
\draw (0,0.14) node[anchor=north] {$\ell_\eta$};
\begin{scriptsize}
\draw [fill=black] (0.6801232890394878,0) circle (.2pt);
\draw [fill=black] (0.6068737219236293,0) circle (.2pt);
\draw [fill=black] (0.7226363962095523,0.1782874744335669) circle (.2pt);
\draw [fill=black] (-0.6068737219236293,0) circle (.2pt);
\draw [fill=black] (-0.6801232890394878,0) circle (.2pt);
\draw [fill=black] (-0.7226363962095523,0.178287474433567) circle (.2pt);
\end{scriptsize}
\end{tikzpicture}
\label{fig:noProp2.27}
\caption{The shape in Example~\ref{example:noProp2.27}}

\end{figure}

\section{The isominwidth problem for convex bodies in the hyperbolic space}
\label{sec:pal:hyp}

 In this section, we verify that
 among convex bodies of fixed Lassak width in $\hyp^n$, the infimum of volume and the infimum of the inradius are both zero.

First we state a claim that directly follows from the use of the Poincar\'e ball model in a way such that the center $o$ of $B^n$ is contained in the orthogonal projection of $\ell$ into $H$.

\begin{lemma}
	\label{line-orthogonal-projection}
	For a line $\ell\subset \hyp^n$ and a hyperplane $H\subset \hyp^n$ such that $\ell$ is not orthogonal to $H$ and $\ell\not\subset H$, the orthogonal projection of $\ell$ into $H$ is an open segment contained in a line $\ell'\subset H$, and $\ell$ and $\ell'$ span hyperbolic two-dimensional subspace $\Pi$ orthogonal to $H$.
	
\end{lemma}
 
\begin{theorem}\label{thm:nopal}
Let $w>0$ be a fixed positive number. Then, for $n\geq 2$,
$$
\inf\left\{V\left(K\right)\colon K\subset \hyp^n\text{ convex body, }w\left(K\right)\geq w\right\}=0.
$$
\end{theorem}
\begin{proof} 
It follows from \eqref{ultraparallel-distance-infty} that for $r\in(0,\frac12)$, there exists a $g(r)>1$ with the following property: If $a_1,a_2\in \hyp^2$ satisfies $d(a_1,a_2)=r$, and $\ell_1,\ell_2\subset \hyp^2$ are the ultraparallel lines such that $\ell_j$, $j=1,2$, passes through $a_j$ and is orthogonal to the segment $[a_1,a_2]$, then a point $b_1\in\ell_1$ with
$d(b_1,a_1)=g(r)$ satisfies that
the distance $d(b_1,\ell_2)$ of $b_1$ from $\ell_2$ is
\begin{equation}
\label{b1ell2gr}
d(b_1,\ell_2)\geq\frac1r.
\end{equation}
We note that if we keep the condition that $\ell_2$ is orthogonal to the segment $[a_1,a_2]$, but the other conditions are changed into
$d(a_1,a_2)\geq r$, $d(b_1,a_1)\geq g(r)$, and $\angle(a_2,a_1,b_1)\geq \frac{\pi}2$, then \eqref{b1ell2gr} still holds.

For $n\geq 2$, let us fix a line $\ell\subset \hyp^n$, and a point $m\in \ell$, and let $H_0\subset \hyp^n$ be the hyperplane containing $m$ and orthogonal to $\ell$.
To define a convex body $K_r$
for $r\in\left(0,\frac12\right)$, we  consider $p_r,q_r\in\ell$ such that $m$ is the midpoint of the segment $[p_r,q_r]$ and $d(p_r,m)=d(q_r,m)=g(r)$, and define $K_r$ as the convex hull of $B(m,r)$, $p_r$ and $q_r$ (see Figure~\ref{fig:spike}). We observe that $K_r$ is symmetric through $H_0$, and claim that
\begin{gather}
\label{nopal-wLKr}
w(K_r)\geq \mbox{$\frac1r$ for $r\in\left(0,\frac12\right)$, and hence }\lim_{r\to 0^+}w(K_r)=\infty,\\
\label{nopal-VKr}
\lim_{r\to 0^+}V(K_r)=0.
\end{gather}

\begin{figure}[H]
\centering
\begin{tikzpicture}[line cap=round,line join=round,>=triangle 45,x=1cm,y=1cm,scale=4]
\clip(-2.276478078557974,-1.0167959098620705) rectangle (1.9862150949743307,1.010664286668644);
\draw [line width=0.8pt,dashed] (0,0) circle (1cm);
\draw [line width=0.8pt,dashdotted] (-1,0)-- (1,0);
\draw [shift={(0,2.6)},line width=0.8pt]  plot[domain=4.317597860684928:5.1071801000844514,variable=\t]({1*2.4*cos(\t r)+0*2.4*sin(\t r)},{0*2.4*cos(\t r)+1*2.4*sin(\t r)});
\draw [shift={(-1.0095355671727764,-2.396004578170347)},line width=1.5pt]  plot[domain=1.1720287494198045:1.5130863414427642,variable=\t]({1*2.4*cos(\t r)+0*2.4*sin(\t r)},{0*2.4*cos(\t r)+1*2.4*sin(\t r)});
\draw [shift={(1.0095355671727728,-2.396004578170331)},line width=1.5pt]  plot[domain=1.6285063121470282:1.9695639041699902,variable=\t]({1*2.4*cos(\t r)+0*2.4*sin(\t r)},{0*2.4*cos(\t r)+1*2.4*sin(\t r)});
\draw [shift={(1.0095355671727788,2.396004578170358)},line width=1.5pt]  plot[domain=4.3136214030095985:4.654678995032556,variable=\t]({1*2.4*cos(\t r)+0*2.4*sin(\t r)},{0*2.4*cos(\t r)+1*2.4*sin(\t r)});
\draw [shift={(-1.009535567172775,2.396004578170341)},line width=1.5pt]  plot[domain=4.770098965736821:5.111156557759783,variable=\t]({1*2.4*cos(\t r)+0*2.4*sin(\t r)},{0*2.4*cos(\t r)+1*2.4*sin(\t r)});
\draw [shift={(0,0.3135210448119356)},line width=0.8pt]  plot[domain=-3.2184597718814203:0.07686711829162733,variable=\t]({1*0.9258106777680023*cos(\t r)+0*0.9258106777680023*sin(\t r)},{0*0.9258106777680023*cos(\t r)+1*0.9258106777680023*sin(\t r)});
\draw [shift={(0,0)},line width=1.5pt]  plot[domain=4.313621403009595:5.1111565577597835,variable=\t]({1*0.2*cos(\t r)+0*0.2*sin(\t r)},{0*0.2*cos(\t r)+1*0.2*sin(\t r)});
\draw [shift={(0,0)},line width=1.5pt]  plot[domain=1.1720287494198025:1.9695639041699906,variable=\t]({1*0.2*cos(\t r)+0*0.2*sin(\t r)},{0*0.2*cos(\t r)+1*0.2*sin(\t r)});
\draw [shift={(0,0)},line width=0.8pt,dotted]  plot[domain=1.9695639041699906:4.313621403009595,variable=\t]({1*0.2*cos(\t r)+0*0.2*sin(\t r)},{0*0.2*cos(\t r)+1*0.2*sin(\t r)});
\draw [shift={(0,0)},line width=0.8pt,dotted]  plot[domain=-1.1720287494198027:1.1720287494198025,variable=\t]({1*0.2*cos(\t r)+0*0.2*sin(\t r)},{0*0.2*cos(\t r)+1*0.2*sin(\t r)});
\draw (0,0.09) node[anchor=north west] {$m$};
\draw (-0.88,0.136) node[anchor=north west] {$p_r$};
\draw (0.88,0.136) node[anchor=north east] {$q_r$};
\draw (-1.06,0.485) node[anchor=north west] {$i_1$};
\draw (1.06,0.485) node[anchor=north east] {$i_2$};
\draw (-1.1,0.052) node[anchor=north west] {$i$};
\draw (1.0009699746767695,0.052) node[anchor=north west] {$-i$};
\draw (0.3508422422355219,-0.0851695716215109) node[anchor=north west] {$K_r$};
\draw (-0.4,0.38) node[anchor=north west] {$\ell_r$};
\begin{scriptsize}
\draw [fill=black] (0,0) circle (0.3pt);
\draw [fill=black] (-1,0) circle (0.3pt);
\draw [fill=black] (1,0) circle (0.3pt);
\draw [fill=black] (-0.8711084694395298,0) circle (0.3pt);
\draw [fill=black] (0.8711084694395297,0) circle (0.3pt);
\draw [fill=black] (0,5) circle (0.3pt);
\draw[color=black] (-2.2496686875294687,1.0492027862721203) node {$G'$};
\draw [fill=black] (0,2.6) circle (0.3pt);
\draw[color=black] (-2.2630733830437215,1.0492027862721203) node {$H$};
\draw [fill=black] (0.9230769230769229,0.38461538461538475) circle (0.3pt);
\draw [fill=black] (-0.923076923076923,0.38461538461538475) circle (0.3pt);
\end{scriptsize}
\end{tikzpicture}
\caption{The convex body $K_r\subset\hyp^2$ with ``large'' width and ``small'' area.}
\label{fig:spike}
\end{figure}

Before proving \eqref{nopal-wLKr} and \eqref{nopal-VKr}, we observe that these two inequalities verify 
Theorem~\ref{thm:nopal}, as for fixed $w>0$, $w(K_r)>w$ for small $r>0$ by \eqref{nopal-wLKr}.

For \eqref{nopal-wLKr}, let $H$ be a supporting hyperplane of $K_r$.
We write $m'$ to denote the orthogonal projection of $m$ into $H$. Possibly after interchanging $p_r$ and $q_r$, we may assume that $\angle(m',m,q_r)\geq \frac{\pi}2$. Since $d(m',m)\geq r$ as $B(m,r)\subset K_r$ and $d(m,q_r)=g(r)$, combining Lemma~\ref{line-orthogonal-projection} and \eqref{b1ell2gr} yields \eqref{nopal-wLKr}.

For \eqref{nopal-VKr}, we write $\alpha_s\in\left(0,\frac{\pi}2\right)$ to denote the ``angle of parallelism'' corresponding to $s>0$ as defined in \eqref{parallel-angle}:
\begin{equation}
\label{parallel-angle-s} 
\sin\alpha_s=\frac1{\cosh a}.
\end{equation}
In addition, let $H_0^+\subset \hyp^n$ be the half-space bounded by $H_0$ and containing $p_r$, and let $i$ be the ideal point of $\ell$ ``in $H_0^+$.''
First, we have $K_r\subset \widetilde{K}_r$  for the closure of convex hull $\widetilde{K}_r$ of $\ell$ and $B(m,r)$. Here $\widetilde{K}_r$ is symmetric through $H_0$, and the closure of $(\partial \widetilde{K}_r\cap H_0^+)\backslash B(m,r)$ is the union of half-lines parallel to $\ell$ of the following form: We choose an $x\in\partial B(m,r)$ such that $\angle(x,m,p_r)=\alpha_r$, and consider the half-line connecting $x$ and $i$, that is actually tangent to $B(m,r)$ at $x$. Since
$\lim_{r\to 0^+}\alpha_r=\frac{\pi}2$ 
by \eqref{parallel-angle-s}, we have $\widetilde{K}_r\subset C_r$ for small $r>0$ where $C_r$ is the closure of the convex hull of $H_0\cap B(m,2r)$ and $\ell$.
Now $C_r$ is symmetric through $H_0$ and has rotational symmetry around $\ell$, and $H_0^+\cap \partial C_r$ is the union of half-lines parallel to $\ell$ that connect an $x\in H_0\cap \partial B(m,2r)$ to $i$. We observe that the angle of a half-line connecting an $x\in H_0\cap \partial B(m,2r)$ to $i$ and $[x,m]$ is $\alpha_{2r}$.
Readily, 
\eqref{nopal-VKr} follows if
\begin{equation}
\label{nopal-VCr}
\lim_{r\to 0^+}V(C_r)=0.
\end{equation}
If $n=2$, then $V(C_r)=4(\frac{\pi}2-\alpha_{2r})$
by Lemma~\ref{triangle}; therefore,
\eqref{nopal-VCr}, and in turn 
Theorem~\ref{thm:nopal} follow from  \eqref{parallel-angle-s}.

In order to verify \eqref{nopal-VCr} for any dimension $n\geq 2$, we use the Poincar\'e ball model where $m=o$, the center of $B^n$, and hence $\ell$ is an open Euclidean segment having $i$ as an endpoint.
For $t\in(0,1]$, let $z_t\in\ell$ such that $\|z_t-i\|=t$, and let $\varrho(r,t)$ be the Euclidean radius of the $(n-1)$-dimensional Euclidean ball that is the intersection of $C_r$ and the Euclidean hyperplane $\widetilde{H}_t$ passing through $z_t$ and orthogonal to $\ell$. We observe that $z_1=o$ and $\frac12\,r\leq \varrho(r,1)\leq r$ if $r\in(0,\frac18)$ by Lemma~\ref{hyp-Euc-Hausdorff-dist}.

To estimate $\varrho(r,t)$ in terms of $r$ and $t$, we may assume that $n=2$.
If $x_r$ is one of the endpoints of $H_0\cap C_r$, then the hyperbolic half-line connecting $x_r$ to $i$ is the circular arc of Euclidean radius $R_r\geq 2$ assuming $r\in(0,\frac18)$. Now if $t=R_r\sin\varphi_{r,t}$ for $\varphi_{r,t}\in(0,\frac{\pi}2)$, then  $\varrho(r,t)=R_r(1-\cos \varphi_{r,t})=R_r\left(1-\sqrt{1-\frac{t^2}{R_r^2}}\right)$, and hence
$$
\frac{t^2}{2R_r}<\varrho(r,t)<\frac{t^2}{R_r}
$$
assuming $r\in(0,\frac18)$ and $t\in(0,1]$. In particular, 
$r\geq \varrho(r,1)>\frac{1}{2R_r}$ yields that $R_r>\frac{1}{2r}$, and hence if $r\in(0,\frac18)$ and $t\in(0,1]$, then
\begin{equation}
\label{rtvarrhoRr}
\varrho(r,t)<2rt^2.
\end{equation}

Finally, let $n\geq 2$, and we estimate  $V(C_r)$ using the formula \eqref{Poincare-volume} for the hyperbolic volume in the Poincar\'e ball model. Now if $x\in \widetilde{H}_t\cap C_r$, then
\eqref{rtvarrhoRr} implies that
$$
\|x\|^2\leq (1-t)^2+\varrho(r,t)^2<1-\frac{t}2.
$$
Therefore, writing $\omega_n=|B^n|$, \eqref{Poincare-volume} and \eqref{rtvarrhoRr} yield that
$$
V(C_r)<2\int_0^1\omega_{n-1}\varrho(r,t)^{n-1}\left(\frac{4}{t}\right)^n\,dt<r^{n-1}2^{3n}\omega_{n-1}\int_0^1t^{n-1}\,dt=
r^{n-1}\cdot \frac{2^{3n}\omega_{n-1}}{n}.
$$
We conclude \eqref{nopal-VCr}, and in turn 
Theorem~\ref{thm:nopal}.
\end{proof}

\noindent{\textbf{Remark.} The bodies $K_r$ constructed in the proof of Theorem \ref{thm:nopal} are not \emph{reduced}, i.e., we can find a convex body $K'_r\subsetneq K_r$ with $w(K'_r) = w(K_r)$. For $n=2$, a natural choice is the rhombus $P_r\subsetneq K_r$ given as the convex hull of the points $p_r$ and $q_r$ and the two points $x_r$ and $y_r$ on the orthogonal bisector of $[p_r,q_r]$ that are of distance $r$ to $o$, if $y_r$ is closer to the line $\ell$ through $x_r$ orthogonal to $\left[x_r,y_r\right]$ than $p_r$. One checks that $w(P_r)=w(K_r)$ and that there exists no convex body $K'_r\subset P_r$ of the same width. Thus, $P_r$ is called a \emph{reduction} of $K_r$.}\\
\\
Let us consider the relation between the minimal width and the inradius $r(K)$ of a convex body $K$ where $r(K)$ the maximal radius  of a ball contained in $K$.
We recall that according to 
Steinhagen's theorem (cf. 
P. Steinhagen \cite{Ste22} or Eggleston \cite{Egg58}), among convex bodies in $\R^n$ of given minimal width $w>0$, the regular simplex minimizes the inradius.
However, we conclude from Theorem~\ref{thm:nopal} that
no analogue of Steinhagen's theorem hold in $\hyp^n$ with respect to the minimal Lassak width.

\begin{corollary}\label{thm:nopal-radius}
Let $w>0$ be a fixed positive number. Then,
$$
\inf\left\{r\left(K\right)\colon K\subset \hyp^n\text{ convex body, }w\left(K\right)\geq w\right\}=0.
$$
\end{corollary}

\section{The regular horocyclic triangle and its extremality with respect to width and inradius}
\label{secRegHoroTriangle}

The goal of this section is to show that regular horocyclic triangles minimize the inradius among h-convex bodies of fixed Lassak width.

For horoballs $\Xi_1,\Xi_2,\Xi_3\subset \hyp^2$ 
 whose ideal points are different, and $\partial \Xi_j$ intersects the non-empty interior of the intersection of
 $\Xi_m\cap \Xi_k$, $\{j,k,m\}=\{1,2,3\}$,
 the h-convex body $T=\Xi_1\cap\Xi_2\cap\Xi_3$ is called a horocyclic triangle.   Then $T$ is compact, and there exist $q_j=\partial \Xi_m\cap\partial \Xi_k\cap\Xi_j$, $\{j,k,m\}=\{1,2,3\}$
according to Lemma~\ref{HoroballsIntersect}.
We call $q_1,q_2,q_3$  the vertices of $T$ where $T$ is bounded by the three horocyclic arcs obtained as the arc of $\partial \Xi_j$ between $q_m$ and $q_k$, $\{j,k,m\}=\{1,2,3\}$, and  $T$ is actually the h-convex hull of $q_1,q_2,q_3$. Actually, for any three non-collinear points of $\hyp^2$, their h-convex hull is a horocyclic triangle.

 For a hyperbolic circular disc $B(\tilde{p},r)\subset \hyp^2$, choose three equally spaced points $\tilde{z}_1,\tilde{z}_2,\tilde{z}_3\in \partial B(\tilde p,r)$, and hence $\angle(\tilde{z}_i,\tilde{p},\tilde{z}_j)=\frac{2\pi}3$ for $i\neq j$.
Then the corresponding {\it regular horocyclic triangle} $T$ with inradius $r$ is obtained as the intersection of the three horoballs
$\widetilde{\Xi}_j$, $j=1,2,3$ containing  
$B(\tilde{p},r)$ such that 
$\partial \widetilde{\Xi}_j$ is tangent to $B(\tilde{p},r)$ at
$\tilde{z}_j$, $j=1,2,3$ (cf. Lemma~\ref{horo-sphere-sphere}).
Here $T$ is symmetric through the common line of $\tilde{z}_j$, $\tilde{q}_j$, $\tilde{p}$ for $j=1,2,3$  by the symmetry of the arrangement of
$\tilde{z}_1,\tilde{z}_2,\tilde{z}_3$
and of the horoballs (cf. Lemma~\ref{HoroballSymmetry}), and hence $T$ has threefold rotational symmetry through $\tilde{p}$.

For $\{j,k,m\}=\{1,2,3\}$ and $\tilde{q}_j=\partial \widetilde{\Xi}_k\cap \partial\widetilde{\Xi}_m\cap\widetilde{\Xi}_j$,  (cf.\ Lemma~\ref{HoroballsIntersect}) we have:
\begin{itemize}
\item $\tilde{q}_1,\tilde{q}_2,\tilde{q}_3$ are called the {\it vertices} of $T$ where $T$ is the
h-convex hull of $\tilde{q}_1,\tilde{q}_2,\tilde{q}_3$,
\item the arc of $\partial \widetilde{\Xi}_j$ between $\tilde{q}_k$ and $\tilde{q}_m$ is called the {\it side} of $T$ opposite to $\tilde{q}_j$ where the midpoint of this side is $\tilde{z}_j$. We observe that $\partial T$ is the union of its three horocycle sides.
\end{itemize}

Since the circumscribed circular disk of $T$ is unique (cf.\ Lemma~\ref{circumscribed-ball}), the threefold rotational symmetry of $T$ around $\tilde{p}$ yields that $\tilde{p}$ is the center of the circumscribed circular disk, and  Lemma~\ref{horo-sphere-sphere} yields that
$$
R(T)=d(\tilde{p},\tilde{q}_i),\,i=1,2,3.
$$
We also observe that if
$j\neq k$, then  $\tilde{p}\in{\rm int}\widetilde{\Xi}_k$ and $[\tilde{z}_k,\tilde{p}]$ is orthogonal to $\partial\widetilde{\Xi}_k$, and hence according to \eqref{parallel-half-lines}, there exists an acute angle $\aleph(T)\in\left(0,\frac{\pi}2\right)$ depending on $r$ such that the horocyclic arc of $\partial\widetilde{\Xi}_k$ emanating from $\tilde{q}_j$ and passing through $\tilde{z}_k$  encloses an angle $\aleph(T)$  with the half-line emanating from $\tilde{q}_j$ and passing through $\tilde{p}$.

\begin{lemma}
\label{ReuleauxTriangleWidth}
Let $T$ be the regular horocyclic  triangle of inradius $r>0$.
\begin{description}
\item[(a)]  $w(T)=r+R(T)=d(\tilde{q}_j,\tilde{z}_j)$, $j=1,2,3$, where $\tilde{q}_1,\tilde{q}_2,\tilde{q}_3$ are the vertices of $T$, and $\tilde{z}_j$ is the midpoint of the side of $T$ opposite to $\tilde{q}_j$, $j=1,2,3$.
\item[(b)] If $K\subset T$ is h-convex and $K\neq T$, then $w(K)<w(T)$.
\item[(c)] If $T'$ is a regular horocyclic triangle of inradius $r'>r$, then $w(T')>w(T)$.
\end{description}
\end{lemma}
\begin{proof}
We use the notation as before Lemma~\ref{ReuleauxTriangleWidth}.
For $j=1,2,3$, let $h_j$ be the half-line emanating from $\tilde{z}_j$ and passing through $\tilde{q}_j$, and hence $h_j$ is orthogonal to $\partial \Xi_j$.

As
$\aleph(T)<\frac{\pi}2$, we deduce that the line $\tilde{\ell}_j$
passing through $\tilde{q}_j$ and orthogonal to  $[\tilde{q}_j,\tilde{z}_j]$
is a supporting line of $T$. Now
Lemma~\ref{horo-hyper-sphere}
yields that the hypercycle $X_j$ corresponding to $\tilde{\ell}_j$ and passing through $\tilde{z}_j$ is a supporting 
hypercycle to $T$; therefore, 
\begin{equation}
\label{ReuleauxTriangleWidth-up}
 w(T)\leq d(\tilde{q}_j,\tilde{z}_j).
\end{equation}
On the other hand, Proposition~\ref{width-opposite} yields the existence of a line $\ell_0$ such that the length of  $\ell_0\cap K$ is $w(T)$, and the two lines orthogonal to $\ell_0$ at the endpoints of
$\ell_0\cap K$ are supporting lines to $T$.

We claim that if $\ell$ is a line intersecting ${\rm int}T$
such that the two lines orthogonal to $\ell$ at the endpoints of
$\ell\cap T$ are supporting lines to $T$, then
\begin{equation}
\label{line-ort-boundary-T}
\mbox{either $\ell\cap T=[\tilde{q}_j,\tilde{q}_m]$, or $\ell\cap K=[\tilde{q}_j,\tilde{z}_j]$, for some $j,m\in\{1,2,3\}$, $j\neq m$.}
\end{equation}
We suppose that \eqref{line-ort-boundary-T} does not hold, and seek a contradiction. Then we may assume that one endpoint $p$ of $\ell\cap T$ lies on the open arc
of $\partial \Xi_1$ between $\tilde{q}_2$ and $\tilde{z}_1$, and hence $\ell$ - being orthogonal to $\Xi_1$ - is parallel
to $h_1$, and 
the other endpoint $q$ of $\ell\cap T$ lies on the open arc
of $\partial \Xi_3$ between $\tilde{q}_1$ and $\tilde{q}_2$.
Since $\ell$ is orthogonal to $\Xi_3$ at $q$, we deduce that $\ell$ is also parallel to $h_3$, which is a contradiction, because $h_1$ and $h_3$ intersect in $\tilde{p}$. In turn, we conclude 
\eqref{line-ort-boundary-T}.

In \eqref{line-ort-boundary-T},
Lemma~\ref{horo-sphere-sphere} yields that
$d(\tilde{q}_j,\tilde{q}_m)>d(\tilde{q}_j,\tilde{z}_j)$ for  $j,m\in\{1,2,3\}$, $j\neq m$.
Therefore, $w(T)=d(\tilde{q}_j,\tilde{z}_j)$, $j=1,2,3$, by Proposition~\ref{width-opposite}.

Turning to (b), since 
$T$ is the
h-convex hull of $\tilde{q}_1,\tilde{q}_2,\tilde{q}_3$,  $K\subset T$ is h-convex and $K\neq T$, we deduce the existence a $\tilde{q}_j\not\in K$, $j\in\{1,2,3\}$. By compactness, there exists a supporting line $\tilde{\ell}'_j$ of $K$ intersecting 
$[\tilde{q}_j,\tilde{z}_j]$ in a point $q'_j$, $q'_j\neq \tilde{q}_j,\tilde{z}_j$, such that $K$ lies between 
$\tilde{\ell}'_j$ and the hypercycle $X_j$ corresponding to $\ell_j$ above. According to Lemma~\ref{horo-hyper-sphere},
 the hypercycle $X'_j$ corresponding to $\tilde{\ell}'_j$ and passing through $\tilde{z}_j$ intersects $\Xi_j$ in $\tilde{z}_j$, thus $K\subset T$ lies between
 $\tilde{\ell}'_j$ and $X'_j$.
 We conclude that $w(K)\leq d(\tilde{q}'_j,\tilde{z}_j)<w(T)$.

 Finally, (c) follows by containment and Lemma~\ref{lemma:Lassakwidth:continuity-monotonicity}. 
\end{proof}

Given a circular disk $B(p,r)\subset \hyp^2$ for $r>0$ and $u\not\in B(p,r)$, the spike with apex $u$ corresponding to $B(p,r)$ is $C\backslash B(p,r)$ where $C$ is the h-convex hull of $u$ and $B(p,r)$. It follows that the spike is neither closed nor open, and is bounded by three arcs: the two hypercycle arcs $\sigma_j$, $j=1,2$,  emanating from $u$ and ending at the $x_j\in\partial B(p,r)$ such that $\sigma_1$ and $\sigma_2$ are subsets of the supporting horocycles to $B(p,r)$ at $x_1$ and $x_2$, and the third arc bounding the spike is one of the arcs of $\partial B(p,r)$ connecting $x_1$ and $x_2$.

\begin{example}
\label{triangle-spike}
If $T$ is a regular horocyclic triangle with incircle $B(p,r)$, then $T\backslash B(p,r)$ is the union of three congruent spikes whose apexes are the three vertices of $T$, and the closure of each spike contains one third of $\partial B^n$.
\end{example}

The first property in Lemma~\ref{spike-properties} holds by rotation around $p$, while the second property
follows from Lemma~\ref{HoroballsIntersect}.

\begin{lemma}
\label{spike-properties}
Let $u\not\in B(p,r)$, $r>0$.
For the spikes corresponding to $B(p,r)$ we have:
\begin{itemize}
\item if $d(v,p)=d(u,p)$, then the spikes with apex $u$ and with apex $v$ are congruent,
\item if the spike $\Sigma_u$ with apex $u$ contains $v\neq u$, then it contains the whole spike $\Sigma_v$ with apex $v$, and the  circular arc of $\partial B(p,r)$ bounding $\Sigma_v$ is strictly contained in the  circular arc of $\partial B(p,r)$ bounding $\Sigma_u$.
\end{itemize}
\end{lemma}

 The following technical statement has a key role in understanding the width of an h-convex set whose incircle is given.

\begin{lemma}
\label{third-triangle}
For $p\in \hyp^2$ and $r>0$, let $z_1,z_2,z_3\in \partial B(p,r)$ be such that no two of $z_1,z_2,z_3$ are opposite, and $p$ lies in the convex hull of 
$z_1,z_2,z_3$. 
We write 
$q_1,q_2.q_3$ to denote the vertices of the horocyclic triangle
$T=\Xi_1\cap\Xi_2\cap\Xi_3$  where $\Xi_j$ is the horoball containing $B(p,r)$ and satisfying $z_j\in\partial \Xi_j$
and $q_m=\partial \Xi_j\cap\partial \Xi_k\cap\Xi_m$, $\{j,k,m\}=\{1,2,3\}$. 

If $\{j,k,m\}=\{1,2,3\}$, $\ell_j$ is the tangent line at $z_j$ to $\Xi_j$, then any point of the spike $\Sigma_k$ with apex $q_k$ corresponding to $B(p,r)$ is of distance less than $2r$ from $\ell_j$.
\end{lemma}
\begin{proof}
As no two of $z_1,z_2,z_3$ are opposite, and $p$ lies in the convex hull of 
$z_1,z_2,z_3$, we deduce that the arc of $\partial B(p,r)$ bounding $\Sigma_k$ is the shorter arc between $z_j$ and $z_m$. In particular, the point  $z'_j\in\partial B(p,r)$ that is opposite to $z_j$ is not contained in the closure of the spike $\Sigma_k$. Let $\Omega=B(p,r)\cup \Sigma_k$ be the h-convex hull of $q_k$ and $B(p,r)$, and let $\partial \Xi$ be the common supporting horocycle at $z'_j\in\partial \Omega$ to $\Omega$ and $B(p,r)$ for the horoball $\Xi\supset\Omega$.
Writing $\ell_j^+$ to denote the half-plane bounded by $\ell_j$ and containing $\Xi_j$, we have
$\Omega\subset \ell_j^+\cap \Xi$,
and Lemma~\ref{third-triangle} follows as $z'_j$ is the unique farthest point of $\ell_j^+\cap \Xi$ from $\ell_j$ by 
Lemma~\ref{HoroballHyperplane}.
\end{proof}

In Proposition~\ref{three-spikes}, we use Lemma~\ref{third-triangle} to identify a core part of an h-convex body $K\subset \hyp^2$ of Lassak width $w>0$  that can be naturally compared to the regular horocyclic triangle $T_w$.

\begin{prop}
\label{three-spikes}  
  Let $K\subset \hyp^2$ be an h-convex body of minimal Lassak width at least $w$  and of inradius $\varrho<w/2$ for $w>0$.
If $B(p,\varrho)$ is the incircle, then
there exist $u_1,u_2,u_3\in K$ with $d(u_j,p)=w-\varrho$, $j=1,2,3$, such that the spikes with apexes $u_1,u_2,u_3$ are pairwise disjoint.
\end{prop}
\begin{proof}
As $\varrho<w/2$, $\partial K\cap\partial B(p,\varrho)$ contains no pair of opposite points of $\partial B(p,\varrho)$.
Therefore, Lemma~\ref{inscribed-ball} yields that $p$ is contained in the convex hull of $z_1,z_2,z_3\in\partial K\cap\partial B(p,\varrho)$, and no two of $z_1,z_2,z_3$ are opposite.

Let $T=\Xi_1\cap\Xi_2\cap\Xi_3$ be the horocyclic triangle, where $\Xi_j$ is the horoball containing $B(p,\varrho)$ and satisfying $z_j\in\partial \Xi_j$, and hence $K\subset T$ as $K$ is h-convex.
We write $q_1,q_2,q_3$ to denote the vertices of $T$ where $q_m=\partial \Xi_j\cap\partial \Xi_k\cap\Xi_m$, $\{j,k,m\}=\{1,2,3\}$. 
If $\{j,k,m\}=\{1,2,3\}$, then let $\ell_j$ be the tangent line at $z_j$ to $\Xi_j$, and let $\Sigma_m\subset T$ be the spike with apex $q_m$ corresponding to $B(p,\varrho)$. In particular, 
$T=B(p,\varrho)\cup\Sigma_1\cup\Sigma_2\cup\Sigma_3$, and 
Lemma~\ref{third-triangle} yields
that the distance of any point of $\Sigma_k$ from $\ell_j$ is less than $2\varrho$ for $k\neq j$. 

As $w(K)\geq w>2\varrho$ and $K\subset T$, there exists an $x_j\in \Sigma_j\cap K$ whose distance from $\ell_j$ is at least $w$ for $j=1,2,3$, and hence the convexity of $(\Sigma_j\cup B(p,\varrho))\cap K$, the triangle inequality and $w-\varrho>\varrho$ imply the existence of a $u_j\in\Sigma_j\cap K$ such that $d(u_j,p)=w-\varrho$. Finally,  
Lemma~\ref{spike-properties}
and the h-convexity of $K$ yield that the spike with apex $u_j$ is contained in $\Sigma_j\cap K$.
\end{proof}

For $w>0$, we write $T_w$ to denote a regular horocyclic triangle with $w(T_w)=w$ (cf. Lemma~\ref{ReuleauxTriangleWidth}). We now prove the two-dimensional hyperbolic analogue of Steinhagen's theorem among h-convex domains; namely, the extremality of the horocyclic regular triangle among h-convex domains
with respect to minimal Lassak width and inradius.

\begin{theorem}
\label{Blaschke:horocyclic}
For $w>0$, among h-convex bodies of  minimal Lassak width at least $w$ in $\hyp^2$, the regular horocyclic triangle $T_w$ of minimal Lassak width $w$ minimizes the inradius, and any minimizer is congruent to $T_w$.
\end{theorem}
\begin{proof} For $w>0$, let $r=r(T_w)$, and let
$K\subset \hyp^2$ be an h-convex body  of  minimal Lassak width $w(K)\geq w$, and let $B(p,\varrho)\subset K$ be the incircle. 
In particular, $p$ is contained in the convex hull of $\partial K\cap\partial B(p,\varrho)$ according to Lemma~\ref{inscribed-ball}. We note that $r<w/2$ according to Lemma~\ref{ReuleauxTriangleWidth}.

For the inradius $\varrho=r(K)$ of $K$, let $B(p,\varrho)$ be the incircle of $K$.
If $\varrho>r$, then we are done; therefore, we assume that $\varrho\leq r$. In this case $\varrho<w/2$, and hence 
Proposition~\ref{three-spikes}
yields the existence of $u_1,u_2,u_3\in K$ with $d(u_j,p)=w-\varrho$, $j=1,2,3$, such that the spikes with apexes $u_1,u_2,u_3$ are pairwise disjoint.

Let $w'$ be the minimal Lassak width of a regular horocyclic triangle $T'$ of inradius $\varrho$ whose incircle is also $B(p,\varrho)$ and whose vertices are $q'_1,q'_2,q'_3$. We observe that $d(q'_j,p)=w'-\varrho$, and
the closure of a spike with apex $q'_j$ covers one third of
$\partial B(p,\varrho)$ (cf. Example~\ref{triangle-spike}). On the other hand, the congruent spikes with apexes $u_1,u_2,u_3$ are pairwise disjoint; therefore, their closures cover at most the one third of
$\partial B(p,\varrho)$. We
deduce from Lemma~\ref{spike-properties} that 
$d(u_j,p)\leq d(q'_j,p)$, $j=1,2,3$, and hence $w\leq w'$. In turn, we conclude that $r(T_w)=r\leq \varrho=r(K)$ by 
Lemma~\ref{ReuleauxTriangleWidth}.

If $r(K)=r(T_w)$, then the argument above shows that $u_1,u_2,u_3$ are vertices of a regular horocyclic triangle $\widetilde{T}\subset K$ whose incircle is $B(p,r)$. As $B(p,r)$ is the incircle of the h-convex body $K$, the three horocycles bounding the three horoballs whose intersection is  $\widetilde{T}$ are supporting horocycles of $K$; therefore, $\widetilde{T}=K$.
\end{proof}

\section{The Isominwidth Theorem for h-convex domains}
\label{sec-isomin}

This section proves the hyperbolic version of P\'al's theorem.

\begin{theorem}
\label{thm:pal_hconv}
For $w>0$, the regular horocyclic  triangle is the unique h-convex body with the smallest area among h-convex bodies in $\hyp^2$ whose minimal Lassak width is $w$.
\end{theorem}

The main idea of the proof of Theorem~\ref{thm:pal_hconv}
comes from Proposition~\ref{three-spikes}. 
  Let $K\subset \hyp^2$ be an h-convex body of minimal Lassak width at least $w$, and hence
$r(K)\geq r(T_w)$ by Theorem~\ref{Blaschke:horocyclic}. For this sketch of ideas, let us assume  
  that $r(K)=\varrho<w/2$.
If $B(p,\varrho)$ is the incircle, then
Proposition~\ref{three-spikes} yields the existence of $u_1,u_2,u_3\in K$ with $d(u_j,p)=w-\varrho$, $j=1,2,3$, such that the spikes $\Sigma_1,\Sigma_2,\Sigma_3$ with apexes $u_1,u_2,u_3$, respectively,  are pairwise non-overlapping. Now the core statement is that  
\begin{equation}
\label{thm:pal_hconv-idea}  
V(B(p,\varrho)\cup\Sigma_1\cup \Sigma_2\cup \Sigma_3)\geq V(T_w)
\end{equation}
where the union on the left hand side is the h-convex hull of $u_1,u_2,u_3$ and $B(p,\varrho)$.
Actually, we prove a somewhat stronger statement than \eqref{thm:pal_hconv-idea}
(cf. Proposition~\ref{Trhoareaincrease}), because technically that is easier to handle.

\subsection{The geometric setup}
\label{ssec:setup}

Let us introduce some notations that we use throughout Sections \ref{sec-isomin} and \ref{sec-isomin-stab} (see Figure~\ref{fig:cap-domain}). We fix a point $p\in \hyp^2$, and three half-lines $f_1,f_2,f_3\subset \hyp^2$ emanating from $p$ such that the angle of $f_i$ and $f_j$ is $\frac{2\pi}3$ for $i\neq j$.
For $\{i,j,k\}=\{1,2,3\}$, we also consider the half-line $\tilde{f}_i$ that is collinear with $f_i$ and 
$\tilde{f}_i\cap f_i=\{p\}$, and hence $\tilde{f}_i$ bisects the angle of $f_j$ and $f_k$.

We fix a $w>0$, and $r=r(T_w)<w/2$. If $\varrho\in\left[r,\frac{w}2\right]$ and $i=1,2,3$, then let 
$m_i(\varrho)\in \tilde{f}_i$ and $q_i(\varrho)\in f_i$ satisfy that $d(m_i(\varrho),p)=\varrho$ and $d(q_i(\varrho),p)=w-\varrho$.
As $\varrho\in\left[r,\frac{w}2\right]$, it follows that $m_i(\varrho)\in\partial C_w(\varrho)$ for the h-convex hull $C_w(\varrho)$ of $q_1(\varrho),q_2(\varrho),q_3(\varrho)$ and $B(p,\varrho)$.
We observe that $C_w\left(\frac{w}2\right)=B\left(p,\frac{w}2\right)$,
$C_w(r)$ is congruent to the regular horocyclic triangle $T_w$ of Lassak width $w$, and if 
$\varrho\in\left(r,\frac{w}{2}\right)$, then 
$C_w(\varrho)$ is the disjoint union of $B(p,\varrho)$ and the three spikes with apexes $q_1,q_2,q_3$.
Using the notion of 
\eqref{thm:pal_hconv-idea}, if $\varrho\in[r,\frac{w}2)$, then
(cf. Lemma~\ref{spike-properties})
\begin{equation}
\label{thm:pal_hconv-C}  
V(B(p,\varrho)\cup\Sigma_1\cup \Sigma_2\cup \Sigma_3)=V(C_w(\varrho)). 
\end{equation}
We set $f_1=f$, $q_1(\varrho)=q(\varrho)\in f$, $\tilde{f}_2=\tilde{f}$ and
$m_2(\varrho)=m(\varrho)\in\tilde{f}$, and
  hence the convex intersection $\Gamma_w(\varrho)$ of 
$C_w(\varrho)$ and the convex cone bounded by $f$ and $\tilde{f}$ satisfies that
\begin{equation}
\label{C6Gamma}
V(C_w(\varrho))=6\,V(\Gamma_w(\varrho)).
\end{equation}
If $\varrho\in[r,\frac{w}2)$, then let $v(\varrho)\in\partial\Gamma_w(\varrho)\cap\partial B(p,\varrho)$ be the point such that the supporting horocycle to $B(p,\varrho)$ at $v(\varrho)$ passes through $q(\varrho)$,
and hence $\Gamma_w(\varrho)$ is bounded by the segments $[p,q(\varrho)]$, $[p,m(\varrho)]$, the horocyclic arc between $q(\varrho)$
and $v(\varrho)$, and the shorter circular arc of $\partial B(p,\varrho)$ between $v(\varrho)$ and $m(\varrho)$.

We replace $\Gamma_w(\varrho)$ by a somewhat smaller set $\Delta_w(\varrho)$ whose boundary structure is simpler.
We note that $B(p,r)$ lies in the horoball whose boundary contains the horocyclic arc on $\partial \Gamma_w(\varrho)$ connecting $q(\varrho)$ to $v(\varrho)$.
According to Lemma~\ref{HoroballsIntersect}
and
Lemma~\ref{horo-sphere-sphere}
on the intersection patterns of horocycles and circles,
there exists a horoball
 $\Xi(\varrho)$ such that 
 $p\in \Xi(\varrho)$, and the bounding horocycle $h(\varrho)$ contains $q(\varrho)$ and $m(\varrho)$, and if $\varrho\in\left(r,\frac{w}2\right]$, then
 the open horocyclic arc of $h(\varrho)$ between $q(\varrho)$ and $m(\varrho)$ lies in ${\rm int}\Gamma_m(\varrho)$.
We define $\Delta_w(\varrho)$ to be the intersection of 
$\Xi(\varrho)$ and the convex cone bounded by $f$ and $\tilde{f}$. Thus, for any $\varrho\in\left[r,\frac{w}2\right]$,
 $\Delta_w(\varrho)$ is bounded by the segments $[p,q(\varrho)]$, $[p,m(\varrho)]$, and the horocyclic arc of $h(\varrho)$ between $q(\varrho)$
and $v(\varrho)$, and
\begin{equation}
\label{DeltainGamma}
\Delta_w(\varrho)\subset \Gamma_w(\varrho)
\mbox{ \ with equality if and only if $\varrho=r$.}
\end{equation}
In particular, $V(T_w)=6 V(\Delta_w(r))$.
\begin{figure}[H]
\centering
\begin{tikzpicture}[line cap=round,line join=round,>=triangle 45,x=1cm,y=1cm,scale=9]
\clip(-0.75,-0.52) rectangle (0.75,0.52);
\draw [shift={(0.1390169332771399,-0.23727471896984564)},line width=2pt]  plot[domain=1.7637386979085985:2.1007763685713607,variable=\t]({1*0.725*cos(\t r)+0*0.725*sin(\t r)},{0*0.725*cos(\t r)+1*0.725*sin(\t r)});
\draw [shift={(-0.13901693327714024,-0.23727471896984548)},line width=2pt]  plot[domain=1.0408162850184315:1.3778539556811935,variable=\t]({1*0.725*cos(\t r)+0*0.725*sin(\t r)},{0*0.725*cos(\t r)+1*0.725*sin(\t r)});
\draw [shift={(0.13597746766512986,0.23902955525913222)},line width=2pt]  plot[domain=1.7637386979085985:2.1007763685713607,variable=\t]({-0.5*0.725*cos(\t r)+-0.8660254037844387*0.725*sin(\t r)},{0.8660254037844387*0.725*cos(\t r)+-0.5*0.725*sin(\t r)});
\draw [shift={(-0.27499440094226973,-0.0017548362892866876)},line width=2pt]  plot[domain=1.7637386979085985:2.1007763685713607,variable=\t]({-0.5*0.725*cos(\t r)+0.8660254037844387*0.725*sin(\t r)},{-0.8660254037844387*0.725*cos(\t r)+-0.5*0.725*sin(\t r)});
\draw [shift={(0.27499440094226973,-0.0017548362892870623)},line width=2pt]  plot[domain=1.0408162850184315:1.3778539556811935,variable=\t]({-0.5*0.725*cos(\t r)+-0.8660254037844387*0.725*sin(\t r)},{0.8660254037844387*0.725*cos(\t r)+-0.5*0.725*sin(\t r)});
\draw [shift={(-0.13597746766512958,0.23902955525913244)},line width=2pt]  plot[domain=1.0408162850184315:1.3778539556811935,variable=\t]({-0.5*0.725*cos(\t r)+0.8660254037844387*0.725*sin(\t r)},{-0.8660254037844387*0.725*cos(\t r)+-0.5*0.725*sin(\t r)});
\draw [shift={(0,0.3)},line width=0.8pt,dotted]  plot[domain=4.045442635881026:5.379335324888354,variable=\t]({1*0.7*cos(\t r)+0*0.7*sin(\t r)},{0*0.7*cos(\t r)+1*0.7*sin(\t r)});
\draw [shift={(-0.25980762113533157,-0.15)},line width=0.8pt,dotted]  plot[domain=4.045442635881026:5.379335324888354,variable=\t]({-0.5*0.7*cos(\t r)+-0.8660254037844387*0.7*sin(\t r)},{0.8660254037844387*0.7*cos(\t r)+-0.5*0.7*sin(\t r)});
\draw [shift={(0.25980762113533157,-0.15)},line width=0.8pt,dotted]  plot[domain=4.045442635881026:5.379335324888354,variable=\t]({-0.5*0.7*cos(\t r)+0.8660254037844387*0.7*sin(\t r)},{-0.8660254037844387*0.7*cos(\t r)+-0.5*0.7*sin(\t r)});
\draw [shift={(0,0)},line width=2pt]  plot[domain=2.1007763685713607:3.1352113874116267,variable=\t]({1*0.45*cos(\t r)+0*0.45*sin(\t r)},{0*0.45*cos(\t r)+1*0.45*sin(\t r)});
\draw [shift={(0,0)},line width=2pt]  plot[domain=4.195171470964556:5.229606489804823,variable=\t]({1*0.45*cos(\t r)+0*0.45*sin(\t r)},{0*0.45*cos(\t r)+1*0.45*sin(\t r)});
\draw [shift={(0,0)},line width=2pt]  plot[domain=0.00638126617816511:1.0408162850184315,variable=\t]({1*0.45*cos(\t r)+0*0.45*sin(\t r)},{0*0.45*cos(\t r)+1*0.45*sin(\t r)});
\draw [shift={(0,0)},line width=0.8pt,dash pattern=on 1pt off 1pt]  plot[domain=1.0408162850184315:2.1007763685713607,variable=\t]({1*0.45*cos(\t r)+0*0.45*sin(\t r)},{0*0.45*cos(\t r)+1*0.45*sin(\t r)});
\draw [shift={(0,0)},line width=0.8pt,dash pattern=on 1pt off 1pt]  plot[domain=3.1352113874116267:4.195171470964556,variable=\t]({1*0.45*cos(\t r)+0*0.45*sin(\t r)},{0*0.45*cos(\t r)+1*0.45*sin(\t r)});
\draw [shift={(0,0)},line width=0.8pt,dash pattern=on 1pt off 1pt]  plot[domain=-1.0535788173747633:0.00638126617816511,variable=\t]({1*0.45*cos(\t r)+0*0.45*sin(\t r)},{0*0.45*cos(\t r)+1*0.45*sin(\t r)});
\draw [shift={(0.10775460776681849,0.2606145783765225)},line width=2pt,color=ffqqqq]  plot[domain=3.906589188359715:4.561899616739305,variable=\t]({1*0.7187378760690324*cos(\t r)+0*0.7187378760690324*sin(\t r)},{0*0.7187378760690324*cos(\t r)+1*0.7187378760690324*sin(\t r)});
\draw [shift={(-0.10775460776681851,0.2606145783765225)},line width=2pt,color=ffqqqq]  plot[domain=3.906589188359715:4.561899616739305,variable=\t]({-1*0.7187378760690324*cos(\t r)+0*0.7187378760690324*sin(\t r)},{0*0.7187378760690324*cos(\t r)+1*0.7187378760690324*sin(\t r)});
\draw [shift={(-0.2795761493540484,-0.03698906148736841)},line width=2pt,color=ffqqqq]  plot[domain=3.906589188359715:4.561899616739305,variable=\t]({-0.5*0.7187378760690324*cos(\t r)+-0.8660254037844387*0.7187378760690324*sin(\t r)},{0.8660254037844387*0.7187378760690324*cos(\t r)+-0.5*0.7187378760690324*sin(\t r)});
\draw [shift={(0.17182154158722995,-0.223625516889154)},line width=2pt,color=ffqqqq]  plot[domain=3.906589188359715:4.561899616739305,variable=\t]({-0.5*0.7187378760690324*cos(\t r)+0.8660254037844387*0.7187378760690324*sin(\t r)},{-0.8660254037844387*0.7187378760690324*cos(\t r)+-0.5*0.7187378760690324*sin(\t r)});
\draw [shift={(-0.17182154158722995,-0.22362551688915402)},line width=2pt,color=ffqqqq]  plot[domain=3.906589188359715:4.561899616739305,variable=\t]({0.5*0.7187378760690324*cos(\t r)+-0.8660254037844386*0.7187378760690324*sin(\t r)},{-0.8660254037844386*0.7187378760690324*cos(\t r)+-0.5*0.7187378760690324*sin(\t r)});
\draw [shift={(0.2795761493540484,-0.036989061487368385)},line width=2pt,color=ffqqqq]  plot[domain=3.906589188359715:4.561899616739305,variable=\t]({0.5*0.7187378760690324*cos(\t r)+0.8660254037844388*0.7187378760690324*sin(\t r)},{0.8660254037844388*0.7187378760690324*cos(\t r)+-0.5*0.7187378760690324*sin(\t r)});
\draw (-0.0011776984782996458,0.038921159475356085) node[anchor=north west] {$p$};
\draw (0,0.46) node[anchor=south] {$q_1(\varrho)=q\left(\varrho\right)$};
\draw (0,-0.44) node[anchor=north] {$m_1(\varrho)$};
\draw (-0.3904673621026664,0.2634194930992227) node[anchor=north west] {$m_3(\varrho)$};
\draw (-0.40957360326214454,-0.199906855018119) node[anchor=north west] {$q_2(\varrho)$};
\draw (0.40960648645048003,-0.199906855018119) node[anchor=north west] {$q_3(\varrho)$};
\draw (0.3905002452910019,0.2634194930992227) node[anchor=north west] {$m_2(\varrho)=m\left(\varrho\right)$};
\draw (0.218,0.365) node[anchor=south west] {$v\left(\varrho\right)$};
\begin{scriptsize}
\draw [fill=black] (0,0) circle (.2pt);
\draw [fill=black] (0,0.4742723927959315) circle (.2pt);
\draw [fill=black] (-0.41073194047490863,-0.23713619639796538) circle (.2pt);
\draw [fill=black] (0.4107319404749083,-0.23713619639796588) circle (.2pt);
\draw [fill=black] (0.38971143170299744,0.225) circle (.2pt);
\draw [fill=black] (-0.38971143170299727,0.225) circle (.2pt);
\draw [fill=black] (0,-0.45) circle (.2pt);
\draw [fill=black] (0.22748225445350226,0.38826772195065634) circle (.2pt);
\draw [fill=black] (-0.22748225445350176,0.38826772195065673) circle (.2pt);
\draw [fill=black] (-0.22250858345203084,-0.39113927224221656) circle (.2pt);
\draw [fill=black] (0.4499908379055325,0.002871550291559999) circle (.2pt);
\draw [fill=black] (-0.4499908379055324,0.0028715502915606372) circle (.2pt);
\draw [fill=black] (0.22250858345203028,-0.3911392722422168) circle (.2pt);
\draw [fill=red] (0.4384693581041919,0.8987461387976574) circle (.2pt);
\draw[color=red] (0.31885184094295893,0.7207751408542274) node {$O_1$};
\end{scriptsize}
\end{tikzpicture}
\caption{The cap-domain $C_w\left(\varrho\right)$ and the included h-convex hexagon drawn with red lines whose area is considered later in the proof of Theorem~\ref{thm:pal_hconv}}\label{fig:Cwrho}
\label{fig:cap-domain}
\end{figure}

\subsection{Proof of the h-convex isominwidth theorem}

After our preparations, the following proposition implies Theorem~\ref{thm:pal_hconv}.

\begin{prop}
\label{Trhoareaincrease}
For given $w>0$, 
$V(\Delta_w(\varrho))$
is an increasing function of 
$\varrho\in\left[r,\frac{w}2\right]$.
\end{prop}

The proof of Proposition~\ref{Trhoareaincrease} is prepared by a series of lemmas. If $x,y\in h(\varrho)$ for some $\varrho\in[r,\tfrac w2]$, then we write $\arc{xy}$ to denote the horocyclic arc of $h(\varrho)$ connecting $x$ and $y$.

Roughly speaking, as $\varrho$ increases, we ``gain area'' at the ``vertex'' $m(\rho)$ of $\Delta_w(r)$, while we ``lose area'' at the vertex $q(\varrho)$ of $\Delta_w(r)$. In the remainder of the proof, we formalize this intuition and show that we actually gain more area than we lose.

We recall (cf. Lemma~\ref{triangle}) that the area of a hyperbolic triangle $T$ with angles $\alpha,\beta,\gamma$ is
\begin{equation}
\label{triangle-area}
V(T)=\pi-\alpha-\beta-\gamma.
\end{equation}
For $x,y,z\in \hyp^2$, we write $[x,y,z]$
to denote their convex hull.

\begin{de} 
\label{muell-def}
If $[x,y,z]\subset \hyp^2$ is an isosceles triangle with $\angle(y,x,z)=\varphi$ and $d(x,y)=d(x,z)=\ell$, then we set
$$
\mu(\varphi,\ell)=\angle(x,y,z)<\frac{\pi}2.
$$
\end{de}
\noindent{\bf Remark.} The law of cosines for the angles of the triangle $[x,y,x_0]$ (cf. Lemma~\ref{triangle}) where $x_0$ is the midpoint of $[y,z]$ shows that
\begin{equation}
\label{isosceles-def}
1=\tan \mu(\varphi,\ell) \cdot \tan \frac{\varphi}2\cdot \cosh \ell.
\end{equation}

We deduce from \eqref{isosceles-def} that if $\ell_2>\ell_1>0$ and $\ell>0$, then 
\begin{gather}
\label{isosceles-angle-monotone}
\mu(\varphi,\ell_2)<\mu(\varphi,\ell_1),\\
\label{isosceles-angle-small}
\lim_{\varphi\to 0^+}\mu(\varphi,\ell)=\frac{\pi}2.
\end{gather}

\begin{de} 
\label{xiell-def}
If $d(x,y)=\ell>0$ for $x,y\in \hyp^2$, and $\sigma$ is a horocyclic arc connecting $x$ and $y$, then the angle of $\sigma$ and $[x,y]$ at $x$ (or at $y$) is denoted by $\xi(\ell)$ (see Figure~\ref{fig:xi-ell}).
\end{de}

Since the half-line connecting $x$ to the ideal point of $\sigma$ is orthogonal to $\sigma$ and parallel to the perpendicular bisector of $[x,y]$, the law of cosines for the angles of the resulting asymptotic triangle with right angle yields 
\begin{equation}
\label{xiell-de-eq}
\xi(\ell)=\arccos \left(\cosh \frac{\ell}2\right)^{-1}<\frac{\pi}2.
\end{equation}

\begin{figure}[H]
\centering
\begin{tikzpicture}[line cap=round,line join=round,>=triangle 45,x=1cm,y=1cm,scale=4]
\clip(-1.1,-1.1) rectangle (1.1,1.1);
\draw [line width=0.8pt,dashed] (0,0) circle (1cm);
\draw [line width=0.8pt,dashdotted] (0,-0.6) circle (0.4cm);
\draw [shift={(0.265047478148601,-2.037065671362065)},line width=0.8pt]  plot[domain=1.1917604572388534:1.6163155423061932,variable=\t]({1*1.7944042786158072*cos(\t r)+0*1.7944042786158072*sin(\t r)},{0*1.7944042786158072*cos(\t r)+1*1.7944042786158072*sin(\t r)});
\draw [shift={(0.265047478148601,-2.037065671362065)},line width=0.8pt]  plot[domain=1.6163155423061932:1.8900509845509703,variable=\t]({1*1.7944042786158068*cos(\t r)+0*1.7944042786158068*sin(\t r)},{0*1.7944042786158068*cos(\t r)+1*1.7944042786158068*sin(\t r)});
\draw [shift={(0.265047478148601,-2.037065671362065)},line width=0.8pt,dotted]  plot[domain=1.8900509845509703:2.2086032408522405,variable=\t]({1*1.794404278615807*cos(\t r)+0*1.794404278615807*sin(\t r)},{0*1.794404278615807*cos(\t r)+1*1.794404278615807*sin(\t r)});
\draw [line width=0.8pt] (0.6666666666666667,0.7453559924999298)-- (-0.29814239699997197,-0.33333333333333326);
\draw [line width=0.8pt,dotted] (-0.29814239699997197,-0.33333333333333326)-- (-0.6666666666666669,-0.7453559924999299);
\draw (-0.2968044659594451,-0.33) node[anchor=north] {$x$};
\draw (0.18162276562950477,-0.24) node[anchor=north] {$y$};
\draw (-0.06898197472661183,-0.26) node[anchor=north] {$\ell$};
\draw (-0.23,-0.14) node[anchor=north west] {\color{red}{\pmb{$\xi(\ell)$}}};
\draw [shift={(-0.29814239699997197,-0.33333333333333326)},line width=1.2pt,red]  plot[domain=0.23564037584608052:0.8410686705679301,variable=\t]({1*0.2997261170764911*cos(\t r)+0*0.2997261170764911*sin(\t r)},{0*0.2997261170764911*cos(\t r)+1*0.2997261170764911*sin(\t r)});
\begin{scriptsize}
\draw [fill=black] (-0.29814239699997197,-0.33333333333333326) circle (.5pt);
\draw [fill=black] (0.18339580692346977,-0.24452007369910544) circle (.5pt);
\end{scriptsize}
\end{tikzpicture}
\caption{The definition of $\xi(\ell)$}
\label{fig:xi-ell}
\end{figure}

We deduce from \eqref{xiell-de-eq} that if $\ell_2>\ell_1>0$, then 
\begin{gather}
\label{horocycle-arc-angle-monotone}
\xi(\ell_2)>\xi(\ell_1),\\
\label{horocycle-arc-angle-small}
\lim_{\ell\to 0^+}\xi(\ell)=0,\\
\label{horocycle-arc-angle-large}
\lim_{\ell\to \infty}\xi(\ell)=\frac{\pi}2.
\end{gather}

We say that two horocycles $h$ and $\tilde{h}$ \emph{cross} at an $x\in \hyp^2$ if they intersect at $x$ and their tangent lines at $x$ are different. We deduce from Lemma~\ref{HoroballsIntersect} the following property.

\begin{lemma}
\label{horocycles-meeting}
Let $h$, $\tilde{h}$ be horocycles crossing at an $x$ and bounding the horoballs $H$ and $\widetilde{H}$, respectively.
Then one of the arcs of $\tilde{h}\backslash\{x\}$ avoids $H$, and the other arc of $\tilde{h}\backslash\{x\}$ intersects $H$ in a bounded arc.
\end{lemma}
\noindent{\bf Remark.} We call the angle $\varphi(h,\tilde{h})=\varphi(\tilde{h},h)$ of the arc of $\tilde{h}\backslash\{x\}$ avoiding $H$ and the arc of $h\backslash\{x\}$ intersecting $\widetilde{H}$ the angle of the horocycles $h$ and $\tilde{h}$.\\

It follows from \eqref{isosceles-angle-monotone}, \eqref{horocycle-arc-angle-monotone}, \eqref{horocycle-arc-angle-small} and \eqref{horocycle-arc-angle-large} that for any $\varphi\in\left(0,\frac{\pi}2\right)$, there exists
$\ell_0(\varphi)>0$ such that
\begin{align*}
\xi(\ell)<\mu(\varphi,\ell)&\mbox{ if }0<\ell<\ell_0(\varphi),\\
\xi(\ell)\geq \mu(\varphi,\ell)&\mbox{ if }\ell\geq \ell_0(\varphi).
\end{align*}
In addition, using \eqref{isosceles-angle-small}, we deduce that
\begin{equation}
\label{ell0-varphi-small-large}
\lim_{\varphi\to 0^+}\ell_0(\varphi)=\infty.
\end{equation}

Let $h\neq \tilde{h}$ be horocycles crossing at an $x\in \hyp^2$ and bounding the horoballs $H$ and $\widetilde{H}$, respectively, 
and let  $h_0$ be the arc of $h\backslash\{x\}$  avoiding  $\widetilde{H}$, and $\tilde{h}_0$ be the arc of $h\backslash\{x\}$  intersecting $H$. If $0<\ell<\ell_0(\varphi(h,\tilde{h}))$, then let  $y\in h_0$ and $\tilde{y}\in\tilde{h}_0$ such that 
$d(y_\ell,x)=d(\tilde{y}_\ell,x)=\ell$, and let
$\Omega(h,\tilde{h},x,\ell)$ be the compact set bounded by $[y,\tilde{y}]$, and the horocyclic arcs of $h$ between $x$ and $y$
and of $\tilde{h}$ between $x$ and $\tilde{y}$, and we call $x,y,\tilde{y}$ the vertices of $\Omega(h,\tilde{h},x,\ell)$ (see Figure~\ref{fig:Omega}). 
We observe that the angle of the two horocyclic arcs bounding $\Omega(h,\tilde{h},x,\ell)$ is $\varphi(h,\tilde{h})$.

\begin{figure}[H]
\centering
\begin{tikzpicture}[line cap=round,line join=round,>=triangle 45,x=1cm,y=1cm,scale=4]
\clip(-1.1,-1.1) rectangle (1.1,1.1);
\draw [line width=0.8pt,dashed] (0,0) circle (1cm);
\draw [shift={(-0.1706148784178758,-0.46998996081028527)},line width=0.8pt,dashdotted]  plot[domain=1.2225714595199548:4.364164113109748,variable=\t]({1*0.5*cos(\t r)+0*0.5*sin(\t r)},{0*0.5*cos(\t r)+1*0.5*sin(\t r)});
\draw [shift={(-0.1706148784178758,-0.46998996081028527)},line width=0.8pt,dashdotted]  plot[domain=-1.9190211940698383:0.11616413642560056,variable=\t]({1*0.5*cos(\t r)+0*0.5*sin(\t r)},{0*0.5*cos(\t r)+1*0.5*sin(\t r)});
\draw [shift={(-0.1706148784178758,-0.46998996081028527)},line width=1.2pt]  plot[domain=0.11616413642560056:1.2225714595199548,variable=\t]({1*0.5*cos(\t r)+0*0.5*sin(\t r)},{0*0.5*cos(\t r)+1*0.5*sin(\t r)});
\draw [shift={(0,-0.5)},line width=0.8pt,dashdotted]  plot[domain=1.5707963267948966:4.71238898038469,variable=\t]({1*0.5*cos(\t r)+0*0.5*sin(\t r)},{0*0.5*cos(\t r)+1*0.5*sin(\t r)});
\draw [shift={(0,-0.5)},line width=0.8pt,dashdotted]  plot[domain=-1.5707963267948966:0.46438900370054204,variable=\t]({1*0.5*cos(\t r)+0*0.5*sin(\t r)},{0*0.5*cos(\t r)+1*0.5*sin(\t r)});
\draw [shift={(0,-0.5)},line width=1.2pt]  plot[domain=0.46438900370054204:1.5707963267948966,variable=\t]({1*0.5*cos(\t r)+0*0.5*sin(\t r)},{0*0.5*cos(\t r)+1*0.5*sin(\t r)});
\draw [shift={(0.9209863944325009,-0.8197660437837556)},line width=1.2pt]  plot[domain=2.2877454907190127:2.5408076260912766,variable=\t]({1*0.7212713118311642*cos(\t r)+0*0.7212713118311642*sin(\t r)},{0*0.7212713118311642*cos(\t r)+1*0.7212713118311642*sin(\t r)});
\draw (-0.0015555992593657732,0.05536805202421601) node[anchor=north west] {$x$};
\draw (0.44633258948078497,-0.21971288279256837) node[anchor=north west] {$y$};
\draw (0.32642551532987846,-0.35725335020096055) node[anchor=north west] {$\widetilde{y}$};
\draw (0.4992327692532437,-0.5547473546848057) node[anchor=north west] {$h$};
\draw (-0.6010909700138982,-0.17033938167160706) node[anchor=north west] {$\widetilde{h}$};
\draw (0.21709847713346372,0.14706169696314414) node[anchor=north west] {$\Omega(h,\widetilde{h},x,\ell)$};
\draw (0.14656490410351872,-0.09275245133866789) node[anchor=north west] {$\ell$};
\draw (0.2805786928604142,-0.05395898617219829) node[anchor=north west] {$\ell$};
\draw [line width=0.8pt,dotted] (0,0)-- (0.32601538678898334,-0.41203843209390945);
\draw [line width=0.8pt,dotted] (0,0)-- (0.4470476917113975,-0.2760617019455777);
\draw [shift={(0,0)},line width=1.2pt,dotted,color=red]  plot[domain=5.381756778357468:5.729981645632408,variable=\t]({1*0.12524306339157673*cos(\t r)+0*0.12524306339157673*sin(\t r)},{0*0.12524306339157673*cos(\t r)+1*0.12524306339157673*sin(\t r)});
\draw [shift={(0,0)},line width=1.2pt,color=red]  plot[domain=5.809387620082046:6.157612487356987,variable=\t]({1*0.12524306339157673*cos(\t r)+0*0.12524306339157673*sin(\t r)},{0*0.12524306339157673*cos(\t r)+1*0.12524306339157673*sin(\t r)});
\begin{scriptsize}
\draw [fill=black] (0,0) circle (.5pt);
\draw [fill=black] (0.32601538678898334,-0.41203843209390945) circle (.5pt);
\draw [fill=black] (0.4470476917113975,-0.2760617019455777) circle (.5pt);
\end{scriptsize}
\end{tikzpicture}
\caption{The definition of $\Omega(h,\tilde{h},x,\ell)$}
\label{fig:Omega}
\end{figure}

As the cap of $H$ cut off by $[x,y]$ from $\Omega(h,\tilde{h},x,\ell)$ is congruent to the cap of $\widetilde{H}$ cut off by 
$[x,\tilde{y}]$, we deduce that (recall that we have $0<\ell<\ell_0(\varphi(h,\tilde{h}))$)
\begin{align}
\label{horocycle-triangle-straight-triangle-area}
V(\Omega(h,\tilde{h},x,\ell))&=V([x,y,\tilde{y}]),\\
\label{horocycle-triangle-straight-triangle-angle}
\angle(y,x,\tilde{y})&=\varphi(h,\tilde{h}).
\end{align}
It also follows that
\begin{equation}
\label{horocycle-triangle-area-commute}
\Omega(h,\tilde{h},x,\ell)\mbox{ \ and \ }\Omega(\tilde{h}, h,x,\ell)\mbox{ \  are congruent,}
\end{equation}
 and if $0<\ell_1<\ell_2<\ell_0(\varphi(h,\tilde{h})$, then the containment relation implies
\begin{equation}
\label{horocycle-triangle-area-monotone}
V(\Omega(h,\tilde{h},x,\ell_2))>V(\Omega(h,\tilde{h},x,\ell_1)).
\end{equation}

\begin{de} 
\label{horocycle-converge-def}
Let $h,h_n$, $n\geq 1$ be horocycles enclosing the horoballs $H,H_n$, respectively. We say that $h_n$ tends to $h$ if for any $p\in{\rm int}\,H$,  we have $p\in H_n$ for large $n$, and for any $q\not\in H$,  we have $q\not \in H_n$ for large $n$.
\end{de}

We observe that in the Poincar\'e disk model, the convergence of $h_n$ to $h$ is equivalent to saying that the Euclidean circular disk representing $H_n$ tends to the  Euclidean circular disk representing $H$ with respect to the Hausdorff distance in $\R^2$. In turn, we deduce Lemma~\ref{horocycle-converge-test} and Lemma~\ref{horocycle-converge-consequence}.

\begin{lemma}
\label{horocycle-converge-test}
If $\ell>0$,  and $h,h_n$, $n\geq 1$ are horocycles enclosing the horoballs $H,H_n$, respectively, and  $x_n,y\in h$ and $x_n,y_n\in h_n$ for $n\geq 1$ such that
$\ell\leq d(x_n,y)\leq 2\ell$ and $y_n$ tends to $y$, and the arc $\arc{y_nx_n}$ of $h_n$ between $y_n$ and $x_n$ avoids ${\rm int}\,H$, then $h_n$ tends to $h$.
\end{lemma}
\begin{proof} Using the Poincar\'e disk model, any convergent subsequence of the Euclidean closures of $\{H_n\}$ in $\R^2$ tends to the Euclidean closure of some horoball $H'$. Let $\{H_{n'}\}$ be  such a subsequence of $\{H_n\}$ where using $\ell\leq d(x_{n'},y)\leq 2\ell$, we may also assume that $\{x_{n'}\}$ tends to an $x\in h$ with $x\neq y$.
As the arc $\arc{y_{n'}x_{n'}}$ of $h_{n'}$ between $y_{n'}$ and $x_{n'}$ avoids ${\rm int}\,H$, it tends to the arc $\arc{yx}$ of $h$ between $y$ and $x$ with respect to the Hausdorff metric either in terms of the hyperbolic metric of Poincar\'e disk or the Euclidean metric of $\R^2$. Therefore $H'=H$.

Since any convergent subsequence of the Euclidean closures of $\{H_n\}$ in $\R^2$ tends to the Euclidean closure of $H$,
we deduce that $h_n$ tends to $h$. 
\end{proof}

\begin{lemma}
\label{horocycle-converge-consequence}
Let $x\in \hyp^2$ and $h,h_n$, $n\geq 1$ be horocycles enclosing the horoballs $H,H_n$, respectively, such that $h_n$ tends to $h$ and $x\in h\cap h_n$, and let $\ell_0>0$. In addition, let $p\in{\rm int}\, H$, and let
$y\in h$  lie on the arc of $h\backslash\{x\}$ avoiding $H_n$ for $n\geq 1$, and let $z\in h\cap{\rm int}\, H_n$ for $n\geq 1$. 
\begin{description}
\item[(i)] $h_n$ intersects $[p,y]$ in a unique point $y_n$ for large $n$, and $y_n$ tends to $y$,
and there exists a unique point $z_n\in h_n$ such that $z\in[p,z_n]$.

\item[(ii)] The angle of $[p,y_n]$ and the arc of $h_n$ connecting $x$ and $y_n$  tends to the angle of $[p,y]$ and the arc of $h$ connecting $x$ and $y$, and the angle of $[p,z_n]$ and the arc of $h_n$ connecting $x$ and $z_n$  tends to the angle of $[p,z]$ and the arc of $h$ connecting $x$ and $z$.

\item[(iii)] $\lim_{n\to\infty}\varphi(h,h_n)=0$, and hence $\ell_0<\ell_0(\varphi(h,h_n))$ for large $n$ (cf. \eqref{ell0-varphi-small-large}), and we can speak about $\Omega(h,h_n,x,\ell)$ and $\Omega(h_n,h,x,\ell)$ provided $\ell\in(0,\ell_0]$  and $n$ is large where $\varphi(h,h_n)$ is also the angle at $x$ of the horocycle arcs bounding $\Omega(h,h_n,x,\ell)$ and  of the horocycle arcs bounding $\Omega(h_n,h,x,\ell)$.

\item[(iv)] For $\ell\in(0,\ell_0]$, $\lim_{n\to\infty} V(\Omega(h,h_n,x,\ell))=\lim_{n\to\infty} V(\Omega(h_n,h,x,\ell))=0$.

\end{description}
\end{lemma}
\begin{proof}
The lemma follows from the facts that the Euclidean closure of $H_n$ in $\R^2$ tends to the Euclidean closure of $H$, and hyperbolic angle of two differentiable arcs meeting at point $w$ in the Poincar\'e disk model coincides with the Euclidean angle of the two arcs  in $\R^2$.
\end{proof}

Our last auxiliary statement is directly about $\Delta_w(\varrho)$.

\begin{lemma}
\label{delta+delta-lemma}
For $\varrho\in(r,w/2)$, if $\delta_{+}(\varrho)$ is the angle of the horocyclic arc $\arc{m(\varrho)q(\varrho)}$ and $[m(\varrho),p]$ at $m(\varrho)$, and $\delta_{-}(\varrho)$ is the angle of the horocyclic arc $\arc{m(\varrho)q(\varrho)}$ and $[q(\varrho),p]$ at $q(\varrho)$, then
\begin{equation}
\label{delta+delta-lemma-eq}
\delta_{-}(\varrho)<\delta_{+}(\varrho)<\pi/2.
\end{equation}
\end{lemma}
\begin{proof}
Let $\ell(\varrho)=d(m(\varrho),q(\varrho))$.
Since $d(m(\varrho),p)=\varrho<w-\varrho=d(q(\varrho),p)$, we have
$$
\delta_{-}(\varrho)=\angle(m(\varrho),q(\varrho),p)+\xi(\ell(\varrho))<
\angle(q(\varrho),m(\varrho),p)+\xi(\ell(\varrho))=\delta_{+}(\varrho).
$$

To prove $\delta_{+}(\varrho)<\pi/2$ in \eqref{delta+delta-lemma-eq}, we extend the definition of
$\delta_{+}(\varrho)$ to $\varrho=r$, and observe that $\delta_{+}(r)=\frac{\pi}2$ by the definition $T_w(r)\subset R_w$.
Therefore, it is sufficient to verify that 
\begin{equation}
\label{delta+delta-lemma-delta-dec}
\mbox{$\delta_{+}(\varrho)=\angle(q(\varrho),m(\varrho),p)+\xi(\ell(\varrho))$ 
is strictly monotone decreasing for $\varrho\in[r,w/2)$.}
\end{equation}
First we claim
\begin{equation}
\label{delta+delta-lemma-delta-dec1}
\mbox{$\angle(q(\varrho),m(\varrho),p)$ is strictly monotone decreasing for $\varrho\in[r,w/2)$.}
\end{equation}
For $r\leq \varrho_1<\varrho_2<w/2$, we observe that $[m(\varrho_1),q(\varrho_1)]$ and $[m(\varrho_2),q(\varrho_2)]$ intersect in a point $s$. As the sum of the angles of a hyperbolic triangle is less than $\pi$ according to \eqref{triangle-area}, we deduce that
$$
\angle(q(\varrho_2),m(\varrho_2),p)+\pi-\angle(q(\varrho_1),m(\varrho_1),p)=
\angle(s,m(\varrho_2),m(\varrho_1))+\angle(s,m(\varrho_1),m(\varrho_2))<\pi,
$$
proving \eqref{delta+delta-lemma-delta-dec1}. 

Next we prove that
\begin{equation}
\label{delta+delta-lemma-delta-dec2}
\mbox{$\ell(\varrho)=d(q(\varrho),m(\varrho))$ is strictly monotone decreasing for $\varrho\in[r,w/2)$.}
\end{equation}
According to the hyperbolic law of cosines for sides, we have
\begin{align*}
\frac{d}{d\varrho}\cosh \xi(\varrho)=&\frac{d}{d\varrho}\left(\cosh \varrho\cdot\cosh(w-\varrho)-\mbox{$\frac12$}\sinh \varrho\cdot\sinh(w-\varrho)\right)\\
=&\mbox{$\frac32$}\sinh \varrho\cdot\cosh(w-\varrho)-\mbox{$\frac32$}\cosh \varrho\cdot\sinh(w-\varrho)=\frac{3}{2}\sinh(2\varrho-w)<0,
\end{align*}
as $2\varrho<w$. We deduce \eqref{delta+delta-lemma-delta-dec2}, which together with
\eqref{delta+delta-lemma-delta-dec1} and \eqref{horocycle-arc-angle-monotone} yields \eqref{delta+delta-lemma-delta-dec}.
\end{proof}

\begin{proof}[Proof of Proposition~\ref{Trhoareaincrease}]
It is enough to prove that if $\varrho\in(r,w/2)$, and $\eta>0$ is small (in particular, $\varrho+\eta<w/2$), then
\begin{equation}
\label{Trhoareaincrease-eta0}
V(\Delta_w(\varrho+\eta))-V(\Delta_w(\varrho))\geq 0,
\end{equation}
 as \eqref{Trhoareaincrease-eta0} implies that 
$\frac{d}{d\varrho}\,V(\Delta_w(\varrho))\geq 0$. 

In order to prove \eqref{Trhoareaincrease-eta0}, we introduce some notation given a small $\eta>0$. Since $\varrho$ is kept fixed during the proof of
 \eqref{Trhoareaincrease-eta0}, we do not signal it in our notation.
We write $h_\eta=h(\varrho+\eta)$ and $h_0=h(\varrho)$; moreover,
$m_{0}=m(\varrho)$, $m_{\eta}=m(\varrho+\eta)$, $q_{0}=q(\varrho)$, $q_{\eta}=q(\varrho+\eta)$, and hence 
\begin{equation}
\label{Trhoareaincrease-etaq+q-}
\mbox{$m_{0}\in[p,m_{\eta}]$ and $q_{\eta}\in[p,q_{0}]$, and $d(m_{0},m_{\eta})=d(q_{0},q_{\eta})=\eta$.}
\end{equation}
It follows that the horocyclic arcs $\arc{m_{\eta}q_{\eta}}$ and $\arc{m_{0}q_{0}}$ intersect in a unique point $w_\eta$, and let $w_0$ be the midpoint of $\arc{m_{0}q_{0}}$. We write $\Theta_{\eta,+}$ to denote the compact set bounded by  $[m_{\eta},m_{0}]$, $\arc{m_{\eta}w_\eta}$  and $\arc{m_{0}w_\eta}$, and $\Theta_{\eta,-}$ to denote the compact set bounded by  $[q_{\eta},q_{0}]$, $\arc{q_{\eta}w_\eta}$  and $\arc{q_{0}w_\eta}$. Since
$\Theta_{\eta,+}={\rm cl}\,(\Delta_w(\varrho+\eta)\backslash \Delta_w(\varrho))$ and $\Theta_{\eta,-}={\rm cl}\,(\Delta_w(\varrho)\backslash \Delta_w(\varrho+\eta))$ (see Figure~\ref{fig:Trhoareaincrease}), the estimate \eqref{Trhoareaincrease-eta0} is equivalent to
\begin{equation}
\label{Trhoareaincrease-Theta}
V\left(\Theta_{\eta,+}\right)-V\left(\Theta_{\eta,-}\right)\geq 0.
\end{equation}

The core of the proof of \eqref{Trhoareaincrease-Theta}, and hence of  Proposition~\ref{Trhoareaincrease} is the claim that
\begin{equation}
\label{Trhoareaincrease-weta-w0}
w_\eta\in\arc{w_0q_{0}}.
\end{equation}

Let $\tilde{h}_\eta$ be the horocycle passing through $m_{\eta}$ and  $w_0$ such that the corresponding horoball contains  $[p,m_{\eta}]$. In this case, 
$\tilde{h}_\eta$ tends to $h_0$ according to Lemma~\ref{horocycle-converge-test}, and hence
Lemma~\ref{horocycle-converge-consequence}  yields the existence of a unique  intersection point
$\tilde{q}_{\eta}$ of $\tilde{h}_\eta$  and $[p,q_{0}]$ for small enough $\eta>0$.
It follows from Lemma~\ref{horocycles-meeting} that \eqref{Trhoareaincrease-weta-w0} is equivalent with the estimate
\begin{equation}
\label{Trhoareaincrease-weta-tildeqq0}
d(\tilde{q}_{\eta},q_{0})>\eta=d(q_{\eta},q_{0}).
\end{equation}

According to Lemma~\ref{delta+delta-lemma}, the angles $\delta_{0,+}=\delta_{+}(\varrho)$ and 
$\delta_{0,-}=\delta_{-}(\varrho)$  of $\arc{m_{0}q_{0}}$ with $[m_{0},p]$ at $m_{0}$, and 
with $[q_{0},p]$ at $q_{0}$, respectively, satisfy
\begin{equation}
\label{Trhoareaincrease-delta+-}
\delta_{0,-}<\delta_{0,+}<\pi/2.
\end{equation}
The way to prove \eqref{Trhoareaincrease-weta-tildeqq0} (and in turn \eqref{Trhoareaincrease-weta-w0}) is verifying the formula (cf. \eqref{Trhoareaincrease-etaq+q-})
\begin{equation}
\label{Trhoareaincrease-delta-eta-quotients}
\lim_{\eta\to 0^+} \frac{\sinh d(\tilde{q}_{\eta},q_{0})}{\sinh d(m_{\eta},m_{0})}=
\frac{\sin \delta_{0,+}}{\sin \delta_{0,-}}.
\end{equation}
Let $\ell=d(m_{0},w_0)=d(q_{0},w_0)$. Let $s_{\eta,+}\in\tilde{h}_\eta$ be the third vertex of $\Omega(\tilde{h}_\eta,h_0,w_0,\ell)$ besides $w_0$ and $m_{0}$, and let $s_{\eta,-}\in\tilde{h}_\eta$ be the third vertex of $\Omega(h_0,\tilde{h}_\eta,w_0,\ell)$ besides $w_0$ and $q_{0}$.
\begin{figure}[H]
\centering
\begin{tikzpicture}[line cap=round,line join=round,>=triangle 45,x=1cm,y=1cm,scale=8]
\clip(-0.85,-0.1) rectangle (0.25,1.2);
\draw [shift={(1.108017377814468,-0.13027451851532157)},line width=0.8pt]  plot[domain=2.3032547983717597:2.8875059690402285,variable=\t]({1*1.6569732726901767*cos(\t r)+0*1.6569732726901767*sin(\t r)},{0*1.6569732726901767*cos(\t r)+1*1.6569732726901767*sin(\t r)});
\draw [shift={(0.6057254235678058,-0.16984069648227362)},line width=0.8pt]  plot[domain=2.052344874542917:2.7373876667125545,variable=\t]({1*1.3078324055435435*cos(\t r)+0*1.3078324055435435*sin(\t r)},{0*1.3078324055435435*cos(\t r)+1*1.3078324055435435*sin(\t r)});
\draw [shift={(0.3260350602498085,-0.11015588095377526)},line width=0.8pt]  plot[domain=1.8933013072567746:2.6837754335372637,variable=\t]({1*1.028685479008219*cos(\t r)+0*1.028685479008219*sin(\t r)},{0*1.028685479008219*cos(\t r)+1*1.028685479008219*sin(\t r)});
\draw [shift={(-1.0703865377250317,0.2698096486072018)},line width=0.8pt]  plot[domain=0.0285586813786297:0.36241735471366004,variable=\t]({1*0.5748650470557976*cos(\t r)+0*0.5748650470557976*sin(\t r)},{0*0.5748650470557976*cos(\t r)+1*0.5748650470557976*sin(\t r)});
\draw [shift={(0.17873412953458717,1.2029728239317277)},line width=0.8pt]  plot[domain=3.656933715776594:4.184169816665124,variable=\t]({1*0.20541208894833177*cos(\t r)+0*0.20541208894833177*sin(\t r)},{0*0.20541208894833177*cos(\t r)+1*0.20541208894833177*sin(\t r)});
\draw [shift={(0.6057254235678058,-0.16984069648227362)},line width=0.8pt]  plot[domain=1.9884827986136815:2.052344874542917,variable=\t]({1*1.3078324055435433*cos(\t r)+0*1.3078324055435433*sin(\t r)},{0*1.3078324055435433*cos(\t r)+1*1.3078324055435433*sin(\t r)});
\draw [line width=0.8pt] (-0.5967157741579096,0.3445140128397651)-- (0,0);
\draw [line width=0.8pt] (0,1.1017392463715585)-- (0,0);
\draw (0,1.1) node[anchor=south] {$q_0$};
\draw (0,0) node[anchor=north] {$p$};
\draw (0,1) node[anchor=north west] {$q_\eta$};
\draw (0,0.9) node[anchor=north west] {$\tilde{q}_\eta$};
\draw (-0.24962949517825855,0.85) node[anchor=north west] {$w_\eta$};
\draw (-0.342222659666892,0.7) node[anchor=north west] {$w_0$};
\draw (-0.52,0.52) node[anchor=north east] {$s_{\eta,+}$};
\draw (0.075,1.056) node[anchor=north west] {$s_{\eta,-}$};
\draw (-0.48,0.3) node[anchor=north east] {$m_0$};
\draw (-0.59,0.3767836572091727) node[anchor=north east] {$m_\eta$};
\draw [shift={(-0.495755904208716,0.2862248047472486)},line width=0.8pt,color=ffqqqq]  plot[domain=-0.5235987755982991:1.2878204993201467,variable=\t]({1*0.0957237592023949*cos(\t r)+0*0.0957237592023949*sin(\t r)},{0*0.0957237592023949*cos(\t r)+1*0.0957237592023949*sin(\t r)});
\draw [color=ffqqqq](-0.5,0.345) node[anchor=north west] {$\delta_{0,+}$};
\draw [shift={(0,1.1017392463715585)},line width=0.8pt,color=ffqqqq]  plot[domain=3.902940268091842:4.71238898038469,variable=\t]({1*0.09572375920239501*cos(\t r)+0*0.09572375920239501*sin(\t r)},{0*0.09572375920239501*cos(\t r)+1*0.09572375920239501*sin(\t r)});
\draw [color=ffqqqq](-0.06,1.075) node[anchor=north west] {$\delta_{0,-}$};
\begin{scriptsize}
\draw [fill=black] (0,0) circle (.2pt);
\draw [fill=black] (0,0.9892634024809791) circle (.2pt);
\draw [fill=black] (-0.495755904208716,0.2862248047472486) circle (.2pt);
\draw [fill=black] (-0.5967157741579096,0.3445140128397651) circle (.2pt);
\draw [fill=black] (0,1.1017392463715585) circle (.2pt);
\draw [fill=black] (-0.24989524132546026,0.819269448444679) circle (.2pt);
\draw [fill=black] (-0.3414812364296802,0.6725418949403881) circle (.2pt);
\draw [fill=black] (0,0.865495160254234) circle (.2pt);
\draw [fill=black] (-0.5328632764863253,0.4736197709116883) circle (.2pt);
\draw [fill=black] (0.07520725377960791,1.0255571755730486) circle (.2pt);
\end{scriptsize}
\end{tikzpicture}
\caption{Notation in the proof of Proposition~\ref{Trhoareaincrease}}
\label{fig:Trhoareaincrease}
\end{figure}
In order to prove \eqref{Trhoareaincrease-delta-eta-quotients}, we compare the triangles
$[m_{\eta},m_{0},s_{\eta,+}]$ and $[\tilde{q}_{\eta},q_{0},s_{\eta,-}]$. We deduce from \eqref{horocycle-triangle-area-commute} that
\begin{equation}
\label{Trhoareaincrease-equal-sides}
d(m_{0},s_{\eta,+})=d(q_{0},s_{\eta,-}).
\end{equation}

First we consider the triangle $[m_{\eta},m_{0},s_{\eta,+}]$.
As $\lim_{\eta\to 0^+}\angle(m_{0},w_0,s_{\eta,+})=0$, and $d(w_0,s_{\eta,+})=d(w_0,m_{0})=\ell$, and
$\lim_{\eta\to 0^+} V([m_{0},w_0,s_{\eta,+}])=\lim_{\eta\to 0^+} V(\Omega(\tilde{h}_\eta,h_0,w_0,\ell))=0$
 by \eqref{horocycle-triangle-straight-triangle-area},
\eqref{horocycle-triangle-straight-triangle-angle} and Lemma~\ref{horocycle-converge-consequence}, the area formula \eqref{triangle-area} for $[m_{0},w_0,s_{\eta,+}]$ yields that 
\begin{equation}
\label{Trhoareaincrease-w0+q+s+q0+}
\lim_{\eta\to 0^+}\angle (w_0,s_{\eta,+},m_{0})=\lim_{\eta\to 0^+}\angle (w_0,m_{0},s_{\eta,+})=\frac{\pi}2.
\end{equation}
The first consequence of \eqref{Trhoareaincrease-w0+q+s+q0+} is that%
\begin{equation}
\label{Trhoareaincrease-position-s+}
\mbox{the angle of $[m_{0},s_{\eta,+}]$ and 
$\arc{m_{0}w_0}$ tends to $\frac{\pi}2-\xi(\ell)$}.
\end{equation}
Since $\delta_{0,+}<\pi/2$ according to \eqref{Trhoareaincrease-delta+-} for the angle $\pi-\delta_{0,+}>\frac{\pi}2$ of 
$[m_{0},m_{\eta}]$ and $\arc{m_{0}w_0}$, we deduce that $s_{\eta,+}\neq m_{\eta}$ lies in the arc $\arc{m_{\eta}w_0}$
of $\tilde{h}_\eta$ between $m_{\eta}$ and $w_0$. In other words, $\arc{s_{\eta,+}m_{\eta}}\subset \arc{w_0m_{\eta}}$ for the arc $\arc{s_{\eta,+}m_{\eta}}$ of $\tilde{h}_\eta$ between $s_{\eta,+}$ and $m_{\eta}$.

It follows from the definition of $\xi\left(\ell\right)$ and \eqref{Trhoareaincrease-w0+q+s+q0+} that the angle of $[s_{\eta,+},m_{0}]$ and $\arc{s_{\eta,+}w_0}$ tends to $\frac{\pi}2+\xi(\ell)$, and in turn 
\begin{equation}
\label{Trhoareaincrease-q-arc-s+q0+}
\mbox{the angle of $[s_{\eta,+},m_{0}]$ and $\arc{s_{\eta,+}m_{\eta}}$ tends to $\frac{\pi}2-\xi(\ell)$.}
\end{equation}
As $\lim_{\eta\to 0^+}d(s_{\eta,+},m_{\eta})=\lim_{\eta\to 0^+}d(s_{\eta,+},m_0)=0$ by the triangle inequality and Lemma~\ref{horocycle-converge-consequence}, we deduce from \eqref{Trhoareaincrease-q-arc-s+q0+} and using $\lim_{\eta\to 0^+}\xi\left(d(s_{\eta,+},m_0)\right)=0$ in \eqref{horocycle-arc-angle-small}  that
\begin{equation}
\label{Trhoareaincrease-q+s+q0+}
\lim_{\eta\to 0^+}\angle (m_{\eta},s_{\eta,+},m_{0})=\frac{\pi}2-\xi(\ell).
\end{equation}
As $\angle (s_{\eta,+},m_{\eta},m_{0})+\xi(d(s_{\eta,+},m_{\eta}))$ is the angle of
the horocyclic arc $\arc{m_{\eta}s_{\eta,+}}$ and $[m_{\eta},m_{0}]$, which equals
the angle of
the horocyclic arc $\arc{m_{\eta}w_0}$ and $[m_{\eta},p]$, which in turn tends to
$\delta_{0,+}$ according to Lemma~\ref{horocycle-converge-consequence}, we deduce
from  $\lim_{\eta\to 0^+}d(s_{\eta,+},m_{\eta})=0$ and \eqref{horocycle-arc-angle-small} that
\begin{equation}
\label{Trhoareaincrease-s+q+q0+}
\lim_{\eta\to 0^+}\angle (s_{\eta,+},m_{\eta},m_{0})=\delta_{0,+}.
\end{equation}
Combining \eqref{Trhoareaincrease-q+s+q0+} and \eqref{Trhoareaincrease-s+q+q0+} with the hyperbolic law of sines leads to
\begin{equation}
\label{Trhoareaincrease-positive-main}
\lim_{\eta\to 0^+}\frac{\sinh d(m_{\eta},m_{0})}{\sinh d(s_{\eta,+},m_{0})}=
\lim_{\eta\to 0^+}\frac{\sin\angle (m_{\eta},s_{\eta,+},m_{0})}{\sin\angle (s_{\eta,+},m_{\eta},m_{0})}=\frac{\sin(\frac{\pi}2-\xi(\ell))}{\sin\delta_{0,+}}.
\end{equation}

Next we consider the triangle $[\tilde{q}_{\eta},q_{0},s_{\eta,-}]$ on our way to verify  \eqref{Trhoareaincrease-delta-eta-quotients}.
 As $\lim_{\eta\to 0^+}\angle(q_{0},w_0,s_{\eta,-})=0$, and $d(w_0,s_{\eta,-})=d(w_0,q_{0})=\ell$, and
$\lim_{\eta\to 0^+} V([q_{0},w_0,s_{\eta,-}])=\lim_{\eta\to 0^+} V(\Omega(h_0,\tilde{h}_\eta,w_0,\ell))=0$
 by \eqref{horocycle-triangle-straight-triangle-area},
\eqref{horocycle-triangle-straight-triangle-angle} and Lemma~\ref{horocycle-converge-consequence}, the area formula \eqref{triangle-area} for $[q_{0},w_0,s_{\eta,-}]$ yields that 
\begin{equation}
\label{Trhoareaincrease-w0-tildeq-s-q0-}
\lim_{\eta\to 0^+}\angle (w_0,s_{\eta,-},q_{0})=\lim_{\eta\to 0^+}\angle (w_0,q_{0},s_{\eta,-})=\frac{\pi}2.
\end{equation}
The first consequence of \eqref{Trhoareaincrease-w0-tildeq-s-q0-} is that%
\begin{equation}
\label{Trhoareaincrease-position-s-}
\mbox{the angle of $[q_{0},s_{\eta,-}]$ and 
$\arc{q_{0}w_0}$ tends to $\frac{\pi}2+\xi(\ell)$}.
\end{equation}
As $\delta_{0,-}<\pi/2$ according to \eqref{Trhoareaincrease-delta+-} for the angle $\delta_{0,-}$ of 
$[q_{0},\tilde{q}_{\eta}]$ and $\arc{q_{0}w_0}$,  we deduce
from \eqref{Trhoareaincrease-position-s-}
 that $\tilde{q}_{\eta}\neq s_{\eta,-}$ lies in the arc $\arc{s_{\eta,-}w_0}$
of $\tilde{h}_\eta$ between $s_{\eta,-}$ and $w_0$. In other words, $\arc{s_{\eta,-}\tilde{q}_{\eta}}\subset \arc{s_{\eta,-}w_0}$ for the arc $\arc{s_{\eta,-}\tilde{q}_{\eta}}$ of $\tilde{h}_\eta$ between $s_{\eta,-}$ and $\tilde{q}_{\eta}$.

It follows from the definition of $\xi\left(\ell\right)$ and \eqref{Trhoareaincrease-w0-tildeq-s-q0-} that the angle of $[s_{\eta,-},q_{0}]$ and $\arc{s_{\eta,-}w_0}$, and in turn 
\begin{equation}
\label{Trhoareaincrease-tildeq-arc-s-q0-}
\mbox{the angle of $[s_{\eta,-},q_{0}]$ and $\arc{s_{\eta,-}\tilde{q}_{\eta}}$ tends to $\frac{\pi}2-\xi(\ell)$.}
\end{equation}
As $\lim_{\eta\to 0^+}d(s_{\eta,-},\tilde{q}_{\eta})=0$ by the triangle inequality and Lemma~\ref{horocycle-converge-consequence}, we deduce from \eqref{horocycle-arc-angle-small} and \eqref{Trhoareaincrease-tildeq-arc-s-q0-} that
\begin{equation}
\label{Trhoareaincrease-tildeq-s-q0-}
\lim_{\eta\to 0^+}\angle (\tilde{q}_{\eta},s_{\eta,-},q_{0})=\frac{\pi}2-\xi(\ell).
\end{equation}
As $\angle (s_{\eta,-},\tilde{q}_{\eta},q_{0})-\xi(d(s_{\eta,-},\tilde{q}_{\eta}))$ is the angle of
the horocyclic arc $\arc{\tilde{q}_{\eta}s_{\eta,-}}$ and $[\tilde{q}_{\eta},q_{0}]$, which equals
the angle of
the horocyclic arc $\arc{\tilde{q}_{\eta}w_0}$ and $[\tilde{q}_{\eta},p]$, which in turn tends to
$\delta_{0,-}$ according to Lemma~\ref{horocycle-converge-consequence}, we deduce
from  $\lim_{\eta\to 0^+}d(s_{\eta,-},\tilde{q}_{\eta})=0$ and \eqref{horocycle-arc-angle-small} that
\begin{equation}
\label{Trhoareaincrease-s-tildeq-q0-}
\lim_{\eta\to 0^+}\angle (s_{\eta,-},\tilde{q}_{\eta},q_{0})=\delta_{0,-}.
\end{equation}
Combining \eqref{Trhoareaincrease-tildeq-s-q0-} and \eqref{Trhoareaincrease-s-tildeq-q0-} with the hyperbolic law of sines leads to
\begin{equation}
\label{Trhoareaincrease-negative-main}
\lim_{\eta\to 0^+}\frac{\sinh d(\tilde{q}_{\eta},q_{0})}{\sinh d(s_{\eta,-},q_{0})}=
\lim_{\eta\to 0^+}\frac{\sin\angle (\tilde{q}_{\eta},s_{\eta,-},q_{0})}{\sin\angle (s_{\eta,-},\tilde{q}_{\eta},q_{0})}=\frac{\sin(\frac{\pi}2-\xi(\ell))}{\sin\delta_{0,-}}.
\end{equation}

We conclude \eqref{Trhoareaincrease-delta-eta-quotients} from \eqref{Trhoareaincrease-equal-sides}, \eqref{Trhoareaincrease-positive-main} 
and \eqref{Trhoareaincrease-negative-main}, which in turn implies \eqref{Trhoareaincrease-weta-tildeqq0}.
Therefore, we have $w_\eta\in\arc{w_0q_{0}}$, as it was claimed in \eqref{Trhoareaincrease-weta-w0}.

\mbox{ }

Finally, to prove \eqref{Trhoareaincrease-Theta}, we compare $\Theta_{\eta,+}$ to $\Omega(h_\eta,h_0,w_\eta,\ell_+)$ and
$\Theta_{\eta,-}$ to $\Omega(h_0,h_\eta,w_\eta,\ell_-)$ where $\ell_+=d(m_{0},w_\eta)$, and
$\ell_-=d(q_{0},w_\eta)<\ell$, and $\ell\leq \ell_+<2\ell$. 

It follows from Lemma~\ref{horocycle-converge-test} that $h_\eta$ tends to $h_0$ as $\eta$ tends to zero. We deduce from 
Lemma~\ref{horocycle-converge-consequence} that $V(\Omega(h_\eta,h_0,w_\eta,\ell_+))$ tends to zero, and the angle at $w_\eta$ of the horocyclic arcs bounding  $\Omega(h_\eta,h_0,w_\eta,\ell_+)$,---and hence the angle of the same size at $w_\eta$ of the horocyclic arcs bounding $\Omega(h_0,h_\eta,w_\eta,\ell_-)$---tends to zero.
As $\ell_-\leq\ell_+$, also $V(\Omega(h_0,h_\eta,w_\eta,\ell_-))$ tends to zero as $\eta$ tends to zero. We write $t_{\eta_+}$
to denote the third vertex of $\Omega(h_\eta,h_0,w_\eta,\ell_+)$ besides $w_\eta$ and $m_{0}$, and $t_{\eta_-}$
to denote the third vertex of $\Omega(h_0,h_\eta,w_\eta,\ell_-)$ besides $w_\eta$ and $q_{0}$.

Similarly to \eqref{Trhoareaincrease-position-s+}, we deduce that
the angle of $[m_{0},t_{\eta,+}]$ and 
$\arc{m_{0}w_\eta}$ tends to $\frac{\pi}2-\xi(\ell_+)$.
Since $\delta_{0,+}<\pi/2$ holds according to \eqref{Trhoareaincrease-delta+-} for the angle $\pi-\delta_{0,+}>\frac{\pi}2$ of 
$[m_{0},m_{\eta}]$ and $\arc{m_{0}w_\eta}$, it follows that $t_{\eta,+}$ lies in the arc $\arc{m_{\eta}w_\eta}$
of $h_\eta$ between $m_{\eta}$ and $w_\eta$, and hence 
\begin{equation}
\label{Trhoareaincrease-Xiell+Theta+}
\Omega(h_\eta,h_0,w_\eta,\ell_+)\subset \Theta_{\eta,+}.
\end{equation} 

Similarly to \eqref{Trhoareaincrease-position-s-}, we deduce that
the angle of $[q_{0},t_{\eta,-}]$ and 
$\arc{q_{0}w_\eta}$ tends to $\frac{\pi}2+\xi(\ell_-)$.
Since $\delta_{0,-}<\pi/2$ according to \eqref{Trhoareaincrease-delta+-} for the angle  of 
$[q_{0},q_{\eta}]$ and $\arc{q_{0}w_\eta}$, it follows that $q_{\eta}$ lies in the arc $\arc{t_{\eta,-}w_\eta}$
of $h_\eta$ between $t_{\eta,-}$ and $w_\eta$, and hence 
\begin{equation}
\label{Trhoareaincrease-Xiell-Theta-}
\Theta_{\eta,-}\subset\Omega(h_0,h_\eta,w_\eta,\ell_-) .
\end{equation} 
We conclude from \eqref{Trhoareaincrease-Xiell+Theta+}, \eqref{Trhoareaincrease-Xiell-Theta-} and $\ell_+\geq \ell_-$ that
(cf. \eqref{horocycle-triangle-area-monotone})
$$
V\left(\Theta_{\eta,+}\right)-V\left(\Theta_{\eta,-}\right)\geq
V\left(\Omega(h_\eta,h_0,w_\eta,\ell_+)\right)-V\left(\Omega(h_0,h_\eta,w_\eta,\ell_-)\right)\geq 0,
$$
proving 
\eqref{Trhoareaincrease-Theta}, and in turn \eqref{Trhoareaincrease-eta0}.
\end{proof}

\begin{proof}[Proof of Theorem \ref{thm:pal_hconv}] Since $C_w(\frac{w}2)=B(p,\frac{w}2)$, we deduce from \eqref{C6Gamma}, 
\eqref{DeltainGamma} and Proposition~\ref{Trhoareaincrease}
 that
\begin{equation}
\label{Bw2-larger-Tw}
V\left(B\left(p,\frac{w}2\right)\right)=6V\left(\Gamma_w\left(\frac{w}2\right)\right)>
6V\left(\Delta_w\left(\frac{w}2\right)\right)\geq
6V(\Delta_w(r))=V(T_w).
\end{equation}

Now let $K\subset \hyp^2$ be an h-convex body of minimal Lassak width at least $w$, and hence
$r(K)\geq r(T_w)$ by Theorem~\ref{Blaschke:horocyclic}. If $r(K)\geq \frac{w}2$, then \eqref{Bw2-larger-Tw} yields that $V(K)>V(T_w)$; therefore, we assume that $r(K)<\frac{w}2$.
For $\varrho=r(K)$, we may assume that $B(p,\varrho)$ is the incircle into $K$.
Now Proposition~\ref{three-spikes} yields the existence of $u_1,u_2,u_3\in K$ with $d(u_j,p)=w-\varrho$, $j=1,2,3$, such that the spikes $\Sigma_1,\Sigma_2,\Sigma_3$ with apexes $u_1,u_2,u_3$, respectively,  are pairwise disjoint. 
We deduce from  Lemma~\ref{spike-properties} that
\begin{equation}
\label{thm:pal_hconv-C0}  
V(B(p,\varrho)\cup\Sigma_1\cup \Sigma_2\cup \Sigma_3)=V(C_w(\varrho)). 
\end{equation}
Therefore, combining \eqref{thm:pal_hconv-C0},
\eqref{C6Gamma}, \eqref{DeltainGamma} and Proposition~\ref{Trhoareaincrease}
 implies that
\begin{equation}
\label{K-spikes-larger-Tw}
V(K)\geq V(B(p,\varrho)\cup\Sigma_1\cup \Sigma_2\cup \Sigma_3)=6V\left(\Gamma_w\left(\varrho\right)\right)\geq
6V\left(\Delta_w\left(\varrho\right)\right)\geq
6V(\Delta_w(r))=V(T_w).
\end{equation}

Let us assume that $V(K)=V(T_w)$.
We deduce from \eqref{Bw2-larger-Tw} that $r(K)<\frac{w}2$, and hence we have equality everywhere in \eqref{K-spikes-larger-Tw}. Equality in \eqref{DeltainGamma}
yields that $r(K)=\varrho=r(T_w)$.
It follows that the closure of each spike $\Sigma_i$ covers one third of the boundary of $B(p,r(T_w))$. Since the spikes $\Sigma_1,\Sigma_2,\Sigma_3$ are pairwise disjoint, it follows that
$A=B(p,\varrho)\cup\Sigma_1\cup \Sigma_2\cup \Sigma_3$ is congruent to $T_w$. As $A\subset K$, we conclude that $A=K$. 
\end{proof}

\section{The stability of the isominwidth inequality for  h-convex bodies}
\label{sec-isomin-stab}

The goal of this section is to prove a stability version of 
Theorem~\ref{thm:pal_hconv}. Recall that $\delta$ denotes the Hausdorff distance in $\hyp^2$ as defined before Theorem \ref{thm:pal_hconv-stab0}.

\begin{theorem}
\label{thm:pal_hconv-stab}
For $w>0$, if
$K\subset \hyp^2$ is an h-convex body of minimal Lassak width at least $w$ and $V(K)\leq (1+\varepsilon)V(T_w)$ for $\varepsilon\in[0,1]$, then there exists an isometry $\Phi$ of $\hyp^2$ such that
$$
\delta(K,\Phi T_w)\leq c\sqrt{\varepsilon}
$$
where $c>0$ is an explicitly calculable constant depending on $w$.
\end{theorem}

We use the notation set up in Section~\ref{ssec:setup} proving the isominwidth inequality, thus
 $r=r(T_w)$. We recall that for $\varrho\in\left(r,\frac{w}2\right]$, $\partial\Gamma_w(\varrho)$ contains the shorter circular arc of $\partial B(p,\varrho)$ between $m(\varrho)$ and $v(\varrho)$, and define 
 \begin{equation}
\label{alpharho-def}
\alpha(\varrho)=\angle(m(\varrho),p,v(\varrho))\in\left(0,\frac{\pi}3\right].
\end{equation}
If $\varrho=r$, then we set $v(r)=m(r)$ and $\alpha(r)=0$. We note that $\alpha(\frac{w}2)=\frac{\pi}3$ as
$C_w(\frac{w}2)=B(p,\frac{w}2)$.

The estimates
in Proposition~\ref{area-between-Gamma-Delta} and in Proposition~\ref{alpha-rho-r} form the basis of the proof of Theorem~\ref{thm:pal_hconv-stab}.

\begin{prop}
\label{area-between-Gamma-Delta}
For $w>0$, if $\varrho\in[r(T_w),\frac{w}2]$, then
\begin{equation}
\label{area-between-Gamma-Delta-alpha}
V\left(\Gamma_w(\varrho)\backslash \Delta_w(\varrho)\right)\geq c \alpha(\varrho)^2,
\end{equation}
where $c>0$ is a calculable constant depending on $w$.
\end{prop}
\begin{proof}
All the objects we consider during our argument are contained in $B(p,w)$. We use the Poincar\'e disk model during certain parts of the proof where $p=o$, and hence (cf. \eqref{hypdist-origin-est})
\begin{equation}
\label{area-between-Gamma-Delta-theta}
B(p,w)\subset\theta B^2 \mbox{ \ for }\theta=\frac{e^{2w}-1}{e^{2w}+1}.
\end{equation}
In addition, $B(p,\varrho)$ is a Euclidean circular disk of center $o$ and of radius (cf. \eqref{hypdist-origin})
\begin{equation}
\label{area-between-Gamma-Delta-Brhoradius}
s=\frac{e^{\varrho}-1}{e^{\varrho}+1}.
\end {equation}
We prove \eqref{area-between-Gamma-Delta-alpha} in two steps.\\

\noindent{\bf Step 1.} If $\varrho\in\left(r,\frac{w}2\right]$, then
\begin{equation}
\label{area-between-Gamma-Delta-Step1}
V\left(\Gamma_w(\varrho)\backslash \Delta_w(\varrho)\right)\geq c_1 \alpha(\varrho)^3
\end{equation}
where $c_1>0$ is a calculable constant depending on $w$.

According to Lemma~\ref{horo-sphere-sphere}, there exists an open horocyclic arc $\sigma$ between $m(\varrho)$ and $v(\varrho)$ that lies between $[m(\varrho),v(\varrho)]$ and the shorter arc $\kappa$ of $\partial B(p,\varrho)$ connecting $m(\varrho)$ and $v(\varrho)$. We deduce via
Lemma~\ref{horo-hyper-sphere} that $\sigma\subset \Gamma_w(\varrho)\backslash \Delta_w(\varrho)$, and it is sufficient to verify that
\begin{equation}
\label{area-between-Gamma-Delta-Step01}
V\left(\Theta\right)\geq c_1 \alpha(\varrho)^3
\end{equation}
for the bounded set $\Theta=B(p,\varrho)\backslash \Xi$ bounded by $\kappa$ and $\sigma$ where $\Xi$ is the horoball containing $\sigma$ in its boundary. We use the Poincare disk model to prove \eqref{area-between-Gamma-Delta-Step01} by setting $p=o$, and hence $B(p,\varrho)\subset \theta B^2$
(cf. \eqref{area-between-Gamma-Delta-theta}) and the union of $\Xi$
and its ideal point $i\in\partial B^2$ are Euclidean circular disks.
We set $a=\frac{v(\varrho)+m(\varrho)}2$, and write $a_\kappa\in \partial B(p,r)$ and $a_\sigma\in\partial\Xi$ to denote the midpoints of $\kappa$ and $\sigma$, and $a'_\kappa\in \partial B(p,r)$ to denote the point opposite to $a_\kappa$.
Thus the perpendicular bisector of the common secant connecting $v(\varrho)$ and $m(\varrho)$ of $B(p,r)$ and $\Xi$ contains $i,a'_\kappa,a,a_\sigma,a_\kappa$, and
$$
\|a-a_\kappa\|\cdot \|a-a'_\kappa\|=
\|a-v(\varrho)\|\cdot \|a-m(\varrho)\|=\|a-a_\sigma\|\cdot \|a-i\|.
$$
We deduce that
\begin{equation}
\label{aeta-asigma-quotient}
\frac{\|a-a_\sigma\|}{\|a-a_\kappa\|}=\frac{\|a-a'_\kappa\|}{\|a-i\|}=1-\frac{\|a'_\kappa-i\|}{\|a-i\|}\leq 1-\frac{1-\theta}2.
\end{equation}
Let the ellipse $E$ be the image of $\Xi$ by the affine transformation of  $\R^2$ that leaves the points of the Euclidean line $l$ passing through $v(\varrho)$ and $m(\varrho)$ fixed, and maps $a_\kappa\in \partial B(p,r)$
into $a_\sigma\in\partial\Xi$.
For the Euclidean half-plane $l^+$ bounded by $l$ and containing $a_\sigma$, we have
$l^+\cap \Xi\subset l^+\cap E$ because the boundary of an ellipse may have at most 4 common points with a circle counting multiplicities. We deduce via \eqref{aeta-asigma-quotient} that
$$
|\Theta|\geq |l^+\cap B(p,\varrho)|-|l^+\cap E|=
\left(1-\frac{\|a-a_\sigma\|}{\|a-a_\kappa\|}\right)\cdot |l^+\cap B(p,\varrho)| \geq \frac{1-\theta}2\cdot |l^+\cap B(p,\varrho)|.
$$
As the Euclidean radius of $B(p,\varrho)$ is $s=\frac{e^\varrho-1}{e^\varrho+1}$ (cf. \eqref{area-between-Gamma-Delta-Brhoradius}) and $\angle(v(\varrho),o,m(\varrho))=\alpha(\varrho)$, it follows that
$$
|\Theta|\geq \frac{1-\theta}2\cdot |l^+\cap B(p,\varrho)|=
\frac{1-\theta}2\cdot\frac{s^2}2\cdot 
(\alpha(\varrho)-\sin\alpha(\varrho))\geq
\frac{1-\theta}4
\left(\frac{e^\varrho-1}{e^\varrho+1}\right)^2
 \frac{\alpha(\varrho)^3}6.
$$
Thus, Lemma~\ref{hyp-Euc-volume} yields \eqref{area-between-Gamma-Delta-Step01}, and in turn
\eqref{area-between-Gamma-Delta-Step1} in Step~1.\\

\noindent{\bf Step 2.} For some $\varrho_0\in\left(r,\frac{w}2\right)$ depending on $w$, if $\varrho\in(r,\varrho_0)$, then
\begin{equation}
\label{area-between-Gamma-Delta-Step2}
V\left(\Gamma_w(\varrho)\backslash \Delta_w(\varrho)\right)\geq c_2 \alpha(\varrho)^2
\end{equation}
where $c_2>0$ depends on $w$.

We choose $\varrho_0\in(r,\frac{w}2)$ in a way such that
if $\varrho\in(r,\varrho_0)$, then
\begin{equation}
\label{area-between-Gamma-Delta-Step21}
d(q(\varrho),v(\varrho))\geq \frac{d(q(r),m(r))}2.
\end{equation}
Let $\widetilde{\Xi}$ be the horoball whose boundary contains the horocyclic arc on $\partial \Gamma_w(\varrho)$ between $v(\varrho)$ and $q(\varrho)$, and hence
$\Gamma_w(\varrho)\subset \widetilde{\Xi}$.
First we prove that
if $\varrho\in(r,\varrho_0)$, then
\begin{equation}
\label{area-between-Gamma-Delta-Step22}
B(m(\varrho),c_3\alpha(\varrho)^2)\subset \widetilde{\Xi}
\end{equation}
for some $c_3>0$ depending on $w$. To prove \eqref{area-between-Gamma-Delta-Step22}, we use the Poincar\'e disk model assuming that $p=o$, and hence $B(p,\varrho)$ is a Euclidean circular disk of center $o$ and of radius $s=\frac{e^{\varrho}-1}{e^{\varrho}+1}$ (cf. \eqref{area-between-Gamma-Delta-Brhoradius}), and
$\widetilde{\Xi}$ is a Euclidean circular disk of radius $\frac{1+s}2$. Let $u$ be  Euclidean center of $\widetilde{\Xi}$, thus $\|u\|=\frac{1-s}2$.
It follows that $m(\varrho)+t B^2\subset \widetilde{\Xi}$ 
for $t=\frac{1+s}2-\|m(\varrho)-u\|$ where
$\angle(u,o,m(\varrho))=\pi-\alpha(\varrho)$ and
the law of cosines applied to the Euclidean triangle with vertices $m(\varrho)$, $o$ and $u$ implies that
\begin{align*}
t&=\frac{1+s}2-\sqrt{s^2+\left(\frac{1-s}2\right)^2+
2s\cdot \frac{1-s}2\cdot\cos \alpha(\varrho)}\\
&=\frac{1+s}2-\sqrt{\left(\frac{1+s}2\right)^2-
s(1-s)(1-\cos \alpha(\varrho))}.
\end{align*}
Here $\sqrt{1-x}<1-\frac{x}2$ for $x\in(0,1)$ and 
$1-\cos \alpha(\varrho)>\frac{\alpha(\varrho)^2}4$ yield that
$t>\frac{s(1-s)}{4(1+s)}\cdot\alpha(\varrho)^2$.
Since 
$\frac{e^{r}-1}{e^{r}+1}\leq s\leq \frac{e^{w}-1}{e^{w}+1}$,
we conclude \eqref{area-between-Gamma-Delta-Step22} by
Lemma~\ref{hyp-Euc-dist}.

Now let $\sigma'$ be the horocyclic arc on $\partial \Delta_w$ between $q(\varrho)$ and $m(\varrho)$, and let $\tilde{h}$ be the infinite horocyclic arc of $\partial \widetilde{\Xi}$ emanating from $q(\varrho)$ and passing through
$v(\varrho)$, and hence $\tilde{h}$ contains the horocyclic arc bounding $\Gamma_w(\varrho)$. If $\tilde{m}\in\tilde{h}$ satisfies that $d(q(\varrho),\tilde{m})=d(q(\varrho),m(\varrho))<w$, then
$d(\tilde{m},m(\varrho))\geq c_3\alpha(\varrho)^2$
by \eqref{area-between-Gamma-Delta-Step22}, and hence
$\angle (\tilde{m},q(\varrho),m(\varrho))\geq c_4\alpha(\varrho)^2$ for a $c_4>0$ depending on $w$ by the hyperbolic law of sines (cf. Lemma~\ref{triangle}).
Thus \eqref{horocycle-triangle-straight-triangle-angle} implies that the angle of the horocyclic arcs $\sigma'$ and $\tilde{h}$ at $q(\varrho)$ is at least $c_4\alpha(\varrho)^2$.

In turn, let $v'\in\sigma'$ be the point such that 
$d(q(\varrho),v')=d(q(\varrho),v(\varrho))\geq \frac{d(q(r),m(r))}2$ (cf. \eqref{area-between-Gamma-Delta-Step21}), and let $\Theta'$ be the part of $\Gamma_w(\varrho)\backslash \Delta_w(\varrho)$ in $B(q(\varrho),d)$ for $d=d(q(\varrho),v(\varrho))$; namely, $\Theta'$ is bounded by the horcycle arc of $\tilde{h}$ between  $q(\varrho)$ and $v(\varrho)$, the horocyclic arc of $\sigma'$ between $q(\varrho)$ and $v'$, and the shorter arc of 
$\partial B(q(\varrho),d)$. We deduce from \eqref{horocycle-triangle-straight-triangle-area}
that the area of $\Theta'$ is the same as the area of the circular sector  of $B(q(\varrho),d)$ between $[q(\varrho),v(\varrho)]$ and $[q(\varrho),v']$
where $\angle (v(\varrho),q(\varrho),v')\geq c_4\alpha(\varrho)^2$ by \eqref{horocycle-triangle-straight-triangle-angle}. It follows that (cf. \eqref{area-between-Gamma-Delta-Step21})
$$
V\left(\Gamma_w(\varrho)\backslash \Delta_w(\varrho)\right)\geq V(\Theta')\geq c_4\alpha(\varrho)^2(\cosh d-1)\geq 
c_4\alpha(\varrho)^2\left(\cosh \frac{d(q(r),m(r))}2-1\right),
$$
completing the proof of 
\eqref{area-between-Gamma-Delta-Step2} of Step~2.

Finally, combining
\eqref{area-between-Gamma-Delta-Step1} and \eqref{area-between-Gamma-Delta-Step2}
yields \eqref{area-between-Gamma-Delta-alpha}.
\end{proof}

Lemma~\ref{alpha-varrho-monotone}
is needed in the proof of Proposition~\ref{alpha-rho-r}

\begin{lemma}
\label{alpha-varrho-monotone}
Given $w>0$, $\alpha(\varrho)$ is strictly monotone increasing for 
$\varrho\in[r,\frac{w}2]$.
\end{lemma}
\begin{proof}
It is equivalent to verify that if $r<\varrho<\varrho'<\frac{w}2$, then
\begin{equation}
\label{alpha-rho-rhoprime}
\alpha(\varrho)<\alpha(\varrho').
\end{equation}
As $q(\varrho')\in [p,q(\varrho)]$, considering spikes with apexes $q(\varrho')$ and $q(\varrho)$
corresponding to $B(p,\varrho)$ (cf. Lemma~\ref{spike-properties}) shows that there exists a $\tilde{v}$ in the open shorter arc of $\partial B(p,\varrho)$ between $v(\varrho)$ and $m(\varrho)$ such that the supporting horocycle $\tilde{h}$ to $B(p,\varrho)$ at $\tilde{v}$ passes through $q(\varrho')$, and hence
\begin{equation}
\label{alpha-alphatilde}
\alpha(\varrho)<\angle(m(\varrho),p,\tilde{v}).
\end{equation}
On the other hand, let $v'\in \partial B(p,\varrho')$ such that  $\tilde{v}\in[p,v']$. Then the horoball bounded by the supporting horocycle $h'$ to $B(p,\varrho')$ at $v'$ contains
$\tilde{h}$ as they have the same ideal point; therefore, $q(\varrho')\in[p,q']\backslash\{q'\}$ for some $q'\in h'$. We deduce via Lemma~\ref{spike-properties} that
\begin{equation}
\label{alpha-rho-prime-low}
\alpha(\varrho')>\angle(m(\varrho'),p,v')=\angle(m(\varrho),p,\tilde{v}).
\end{equation}
Finally, combining \eqref{alpha-alphatilde} and \eqref{alpha-rho-prime-low} yields
\eqref{alpha-rho-rhoprime}.
\end{proof}

\begin{prop}
\label{alpha-rho-r}
For $w>0$, if $\varrho\in[r,\frac{w}2]$ for $r=r(T_w)$, then
\begin{equation}
\label{area-between-Gamma-Delta-eta}
\alpha(\varrho)\geq c\cdot(\varrho-r)
\end{equation}
where $c>0$ is a calculable constant depending on $w$.
\end{prop}
\begin{proof}
We deduce from Lemma~\ref{alpha-varrho-monotone} that it is sufficient to prove Proposition~\ref{alpha-rho-r}
if $\varrho\in(r(T_w),\varrho_0)$
for
$\varrho_0\in\left(r(T_w),\frac{w}2\right)$
where $\varrho_0$ depends on $w$.
We recall that
$\aleph(T_w)\in(0,\frac{\pi}2)$ is the angle of the horocyclic arc on $\partial\Gamma_w(r)$ 
and $[q(r),p]$ at $q(r)$, and we consider
\begin{equation}
\label{widetilde-aleph}
\widetilde{\aleph}=\frac{\aleph(T_w)}3<\frac{\pi}6.
\end{equation}
We also recall that $q(\varrho)\in [p,q(r)]$ and $d(q(\varrho),q(r))=\varrho-r$.

Let $\tilde{h}$ be the supporting horocycle to $B(p,r)$ at a point $\tilde{v}\in \Gamma_w(r)\cap \partial B(p,r)$ such that $q(\varrho)\in \tilde{h}$. In addition, let $\tilde{q}\in\tilde{h}$ satisfy that the arc of $\tilde{h}$ between $\tilde{v}$ and $\tilde{q}$ contains $q(\varrho)$, and 
$$
d(\tilde{v},\tilde{q})=d(m(r),q(r)).
$$
If $v'\in\partial B(p,\varrho)$ satisfies that $\tilde{v}\in[p,v']$, then $q(\varrho)$ lies in the interior of the horoball whose boundary is the supporting horocycle to $B(p,\varrho)$ at $v'$.
It follows that
\begin{equation}
\label{alpha-tilde-alpha}   
\alpha(\varrho)=\angle(v(\varrho),p,m(r))>\tilde{\alpha}
\mbox{ \ for }
\tilde{\alpha}=\angle(\tilde{v},p,m(r)).
\end{equation}
Now  the triangle with vertices $p,\tilde{v},\tilde{q}$ is obtained from the triangle with vertices $p,m(r),q(r)$ by a rotation of angle $\tilde{\alpha}$ around $p$, and hence
\begin{equation}
\label{tilde-q-tilde-alpha}   
\angle(\tilde{q},p,q(r))=\tilde{\alpha}\mbox{ \ and \ }
d(\tilde{q},p)=d(q(r),p)=w-r.
\end{equation}
We observe that if $\gamma\in\left(0,\frac{\pi}2\right]$ and $t\in(0,w]$, then the convexity of the function $\sinh t$ yields that
\begin{equation}
\label{sinh-t-w}
\frac{\gamma}2\leq \sin\gamma\leq \gamma\mbox{ \ and \ }
t\leq \sinh t\leq \frac{\sinh w}{w}\cdot t.
\end{equation}
Combining \eqref{alpha-tilde-alpha}, \eqref{tilde-q-tilde-alpha} and \eqref{sinh-t-w}, and using the law of sines (cf. Lemma~\ref{triangle}) in the two halves of the triangle $[p,\tilde{q},q(r)]$ implies that
\begin{equation}
\label{tilde-q-qr-alpha}   
d(\tilde{q},q(r))=2{\rm arcsinh}\left(\sinh(w-r)\cdot\sin \frac{\tilde{\alpha}}2\right)\leq w\cdot\tilde{\alpha}\leq w\cdot \alpha(\varrho).
\end{equation}
On the other hand, we consider the triangle $[q(\varrho),\tilde{q},q(r)]$. Let us choose $\varrho_0\in\left(r,\frac{w}2\right]$ such that if $\varrho\in(r,\varrho_0]$, then (cf. \eqref{widetilde-aleph}, Definition~\ref{xiell-def} and 
\eqref{xiell-de-eq})
\begin{align}
\label{rho01}
\angle(p,q(r),\tilde{q})&\geq \frac{\pi}2-\widetilde{\aleph},\\
\label{rho02}
\xi(2(\varrho-r))&\leq \widetilde{\aleph},\\
\label{rho03}
V(B(q(r),\varrho-r))&\leq\widetilde{\aleph}.
\end{align}
We distinguish two cases.\\

\noindent{\bf Case 1.} 
$d(q(r),\tilde{q})\geq \varrho-r$.

In this case, \eqref{tilde-q-qr-alpha} directly yields Porposition~\ref{alpha-rho-r}.\\

\noindent{\bf Case 2.} 
$d(q(r),\tilde{q})\leq \varrho-r=
d(q(\varrho),q(r))$.

In this case, the triangle inequality yields that
$d(q(\varrho),\tilde{q})\leq 2(\varrho-r)$. Let $g$ be
the tangent half-line to the horocycle
$\tilde{h}$ at $\tilde{q}$ such that $\tilde{g}$ intersects the interior of the triangle $[q(\varrho),\tilde{q},q(r)]$.
As $q(\varrho)\in \tilde{h}$, it follows from the condition \eqref{rho02} that 
the angle $\xi\left(d(q(\varrho),\tilde{q})\right)$ (cf. Definition~\ref{xiell-def}) of $\tilde{g}$ and $[\tilde{q},q(\varrho)]$ at $\tilde{q}$ is at most $\widetilde{\aleph}$.
On the other hand, the angle of $\tilde{g}$ and 
$[\tilde{q},p]$ at $\tilde{q}$ is $\aleph(T_w)$ by the contruction of $\tilde{q}$, and hence the angle of 
$\tilde{g}$ and 
$[\tilde{q},q(\varrho)]$ at $\tilde{q}$ is less than $\frac{\pi}2-\aleph(T_w)$; therefore,
\begin{equation}
\label{qr-tildeq-qrho}
\angle(q(r),\tilde{q},q(\varrho))\leq \frac{\pi}2-\aleph(T_w)+\widetilde{\aleph}=
\frac{\pi}2-2\widetilde{\aleph}.
\end{equation}
We observe that 
$[q(\varrho),\tilde{q},q(r)]\subset B(q(r),\varrho-r)$ by the condition in Case~2, and hence
the area formula for the triangle 
$[q(\varrho),\tilde{q},q(r)]$ (cf. Lemma~\ref{triangle}) implies that
$$
\widetilde{\aleph}\geq v\left( B(q(r),\varrho-r)\right)\geq 
\pi-\angle(\tilde{q},q(r),q(\varrho))-\angle(q(r),\tilde{q},q(\varrho))-\angle(q(r),q(\varrho),\tilde{q}).
$$
Here $\angle(\tilde{q},q(r),q(\varrho))=\angle(\tilde{q},q(r),p<\frac{\pi}2$, thus \eqref{qr-tildeq-qrho} yields that
$$
\angle(q(r),q(\varrho),\tilde{q})\geq \widetilde{\aleph}.
$$
On the other hand, $\angle(q(r),q(\varrho),\tilde{q})\leq \frac{\pi}2+ \widetilde{\aleph}<\pi-\widetilde{\aleph}$ by \eqref{rho01}. We deduce from 
\eqref{sinh-t-w} and applying the law of sines in the triangle $[\tilde{q},q(r),q(\varrho)]$
that
\begin{align*}
d(\tilde{q},q(r))&\geq \frac{w}{\sinh w}\cdot \sinh d(\tilde{q},q(r))=
\frac{w}{\sinh w}\cdot
\frac{\sin\angle(q(r),q(\varrho),\tilde{q})}{\sin\angle(q(\varrho),q(r),\tilde{q})}\cdot
\sinh d(q(\varrho),q(r))\\
&\geq \frac{w}{\sinh w}\cdot \sin\widetilde{\aleph}\cdot (\varrho-r).
\end{align*}
Combining the last estimate with \eqref{tilde-q-qr-alpha} 
completes the proof of Porposition~\ref{alpha-rho-r}
if $\varrho\in(r(T_w),\varrho_0)$, and the case $\varrho\in\left[\varrho_0,\frac{w}2\right]$ follows from Lemma~\ref{alpha-varrho-monotone}.
\end{proof}

To prove Theorem~\ref{thm:pal_hconv-stab}, we still need two  technical statements about horocycles, like Lemma~\ref{angle-horocycle-arc}, Lemma~\ref{Horo-equidistant} and
Lemma~\ref{Horo-secant}.
For a horocyclic arc $\sigma\subset \hyp^2$, we write $\ell_H(\sigma)$ to denote its length.

\begin{lemma}
\label{angle-horocycle-arc}
For $w>0$ and $\varrho\in(0,w]$, let $h$ be the supporting horocycle of $B(p,\varrho)$ at an $a\in\partial B(p,\varrho)$. If 
$\angle(a,p,q_i)\leq \frac{\pi}2$ for $q_1,q_2\in h$ and $i=1,2$, then the arc $\sigma$ of $h$ between $q_1$ and $q_2$  satisfies
$$
\ell_H(\sigma)\leq 
c\cdot \angle(q_1,p,q_2) \mbox{ \ for
$c=e^w+1$.}
$$
\end{lemma}
\begin{proof}
 We use the Poincar\'e disk model such that $p=o$. It follows that
 $\|a\|=\frac{e^\varrho-1}{e^\varrho+1}$ (cf. \eqref{hypdist-origin}), and $h$ is a Euclidean circle of radius $\frac{1+\|a\|}2<1$ whose center $u$ satisfies $\|u\|=\frac{1-\|a\|}2$
 and $o$ lies on the Euclidean segment between $a$ and $u$. Let $b_1,b_2\in h$ such that $b_2=-b_1$, and hence $b_1,b_2$ as vectors are orthogonal to $a$
 and satisfy $\|b_i\|=\sqrt{\|a\|}$, $i=1,2$. We deduce that $\|z\|\leq \theta=\sqrt{\frac{e^w-1}{e^w+1}}$ for any $z\in\sigma$.
 To estimate $\|q_1-q_2\|$, we actually estimate $\ell_E(\sigma)$ where $\ell_E(\cdot)$ is the Euclidean length of a circular arc.  Let $q'_1,q'_2\in\partial B^2$ such that $q_i$ is contained in the Euclidean segment between $o$ and $q'_i$, $i=1,2$, and let $\sigma'$ be the shorter circular arc of $\partial B^2$ connecting $q'_1$ and $q'_2$, and hence $\ell_E(\sigma')=\angle(q_1,p,q_2)$. For $x\in\sigma'$, we write $\pi(x)$ to denote the radial projection of $x$ onto $\sigma$, and $\varphi(x)$ to denote the angle  of the tangent line to $B^2$ at $x$ and the Euclidean tangent line to $h$ at $\pi(x)$. It follows that
 \begin{equation}
\label{q1q2-sigma-sigmaprime}
\|q_1-q_2\|\leq \ell_E(\sigma)=
\int_{\sigma'}\frac{\|\pi(x)\|}{\cos\varphi(x)}\,dx.
 \end{equation}
  Now for any $x\in\sigma'$,
  $\varphi(x)$ coincides the angle of the Euclidean segments between $\pi(x)$ and $o$ and $\pi(x)$ and $u$. 
We observe that the angle of the Euclidean vectors $\pi(x)\in \sigma$ and $u$ is at least $\frac{\pi}2$, and hence there exists a $y_x$ in the Euclidean segment between $\pi(x)$ and $u$ such that the line through $o$ orthogonal to the vector $\pi(x)$ passes through $y$. We deduce that
$$
\frac{\|\pi(x)\|}{\cos\varphi(x)}
=\|\pi(x)-y_x\|\leq \|\pi(x)-u\|\leq 1,
$$
thus \eqref{q1q2-sigma-sigmaprime} yields that
$\|q_1-q_2\|\leq \ell_E(\sigma')=\angle(q_1,p,q_2)$.
We conclude that
$$
d(q_1,q_2)\leq \frac2{1-\theta^2}\cdot \angle(q_1,p,q_2)=(e^w+1)\cdot \angle(q_1,p,q_2)
$$
for
$\theta=\sqrt{\frac{e^w-1}{e^w+1}}$ 
by $q_1,q_2\in\theta B^2$  and
Lemma~\ref{hyp-Euc-dist}.
\end{proof}

Lemma~\ref{Horo-equidistant} and
Lemma~\ref{Horo-secant} are well-known (cf. Berger \cite[Chapter 19]{Ber87} and Ratcliffe \cite[Chapters 3 and 4]{Rat19}).

\begin{lemma}
\label{Horo-equidistant}
Let $\Xi'\subset\Xi\subset \hyp^2$ for horoballs $\Xi'\neq \Xi$ sharing the same ideal point $i$,
and for $x\in\partial \Xi$, let $\pi(x)\in\partial \Xi'$ be 
 the closest point of $\Xi'$ to $x$; or in other words,
$\pi(x)$ is
the intersection of the line through $x$ and $i$ (and orthogonal to $\partial \Xi$ and $\partial \Xi'$) with $\partial \Xi'$.

Then there exists an $\eta>0$
 such that $d(x,\pi(x))=\eta$ holds for any $x\in\partial \Xi$,
 and if $\sigma\subset\partial \Xi$ is a bounded horocyclic arc, then $\sigma'=\pi(\sigma)$ satisfies
$\ell_H(\sigma)=e^\eta\ell_H(\sigma')$.
 \end{lemma}

\begin{lemma}
\label{Horo-secant}
If $a,b\in h$, $a\neq b$ for a horocycle $h\subset \hyp^2$, then the arc $\sigma$ of $h$ between $a$ and $b$ satisfies
$$
2\sinh\ell_H(\sigma)=d(a,b).
$$
 \end{lemma}

 To estimate Hausdorff distance in the proof of Theorem~\ref{thm:pal_hconv-stab}, we use that
 if compact $X,Y\subset \hyp^2$ are convex, then 
 \begin{equation}
\label{Hausdorff-boundary}
\delta(X,Y)=\delta(\partial X,\partial Y)=\max\left\{\delta(\partial X, Y),
\delta(X,\partial Y)
\right\}
\end{equation}
 according Lemma~\ref{supporting-hyperplane-convex}.

\begin{proof}[Proof of Theorem~\ref{thm:pal_hconv-stab}]
For $w>0$, it is sufficient to prove that  if
$K\subset \hyp^2$ is an h-convex body of minimal Lassak width at least $w$ and $V(K)\leq (1+\varepsilon)V(T_w)$ for $\varepsilon\in[0,\varepsilon_0)$, then there exists an isometry $\Phi$ of $\hyp^2$ such that
\begin{equation}
 \label{pal_hconv-stab-epsi0} 
\delta(K,\Phi T_w)\leq c\sqrt{\varepsilon}
\end{equation}
where $c,\varepsilon_0>0$ are explicitly calculable constants depending on $w$ (cf. \eqref{varepsilon0-area} and \eqref{varepsilon0-alpharho}).

We recall that $C_w(\frac{w}2)=B(p,\frac{w}2)$, and\eqref{C6Gamma}, 
\eqref{DeltainGamma} and Proposition~\ref{Trhoareaincrease}
yield that
\begin{equation}
\label{Bw2-larger-Tw0}
V\left(B\left(p,\frac{w}2\right)\right)=6V\left(\Gamma_w\left(\frac{w}2\right)\right)>
6V\left(\Delta_w\left(\frac{w}2\right)\right)\geq
6V(\Delta_w(r))=V(T_w).
\end{equation}
 One of the conditions on $\varepsilon_0>0$ is that 
\begin{equation}
\label{varepsilon0-area}
(1+\varepsilon_0)V(T_w)\leq V\left(B\left(p,\frac{w}2\right)\right),
\end{equation}
and hence $\varrho=r(K)<\frac{w}2$. Let $B(p,\varrho)$ be the incircle into $K$ where
$\varrho\geq r=r(T_w)$ by Theorem~\ref{Blaschke:horocyclic}. 
As $\varrho<w/2$, $\partial K\cap\partial B(p,\varrho)$ contains no pair of opposite points of $\partial B(p,\varrho)$.
Therefore, Lemma~\ref{inscribed-ball} yields that $p$ is contained in the convex hull of $\tilde{z}_1,\tilde{z}_2,\tilde{z}_3\in\partial K\cap\partial B(p,\varrho)$, and no two of $\tilde{z}_1,\tilde{z}_2,\tilde{z}_3$ are opposite.
Let $\widetilde{T}=\Xi_1\cap\Xi_2\cap\Xi_3$ be the horocyclic triangle where $\Xi_j$ is the horoball containing $B(p,\varrho)$ and satisfying $\tilde{z}_j\in\partial \Xi_j$, and hence $K\subset \widetilde{T}$ as $K$ is h-convex.
We write $\tilde{q}_1,\tilde{q}_2,\tilde{q}_3$ to denote the vertices of $\widetilde{T}$ where $\tilde{q}_m=\partial \Xi_j\cap\partial \Xi_k\cap\Xi_m$, $\{j,k,m\}=\{1,2,3\}$. 
Next, let $\widetilde{\Sigma}_m\subset \widetilde{T}$ be the spike with apex $\tilde{q}_m$ corresponding to $B(p,\varrho)$. 

Now the proof of Proposition~\ref{three-spikes} yields the existence of $u_j\in \widetilde{\Sigma}_j\cap K$ with $d(u_j,p)=w-\varrho$, $j=1,2,3$, therefore, 
$\Sigma_j\subset \widetilde{\Sigma}_j$ for the 
spike $\Sigma_j$ with apex $u_j$ and  corresponding to $B(p,\varrho)$
(cf.  Lemma~\ref{spike-properties}).
In particular, 
the spikes $\Sigma_1,\Sigma_2,\Sigma_3$   are pairwise disjoint, and 
$\widetilde{C}=B(p,\varrho)\cup\Sigma_1\cup \Sigma_2\cup \Sigma_3\subset K$ satisfies that
(cf.  Lemma~\ref{spike-properties}) 
\begin{equation}
\label{tildeC-Cw}  
V(\widetilde{C})=V(C_w(\varrho)). 
\end{equation}
For the rest of of the argument, we write $\varphi\ll \psi$ or $\psi\gg \varphi$ for two quantities $\varphi,\psi>0$ if $\varphi\leq c\cdot \psi$ for a calculable constant $c>0$ depending only on $w$. It follows from \eqref{tildeC-Cw},
Proposition~\ref{Trhoareaincrease} and Proposition~\ref{area-between-Gamma-Delta} that
\begin{align*}
\varepsilon\cdot V(T_w)&\geq V(K)-V(T_w)\geq V(\widetilde{C})-V(T_w)=6V\left(\Gamma_w(\varrho)\right)-6V\left(\Delta_w(r)\right)\\
&\geq 
V\left(\Gamma_w(\varrho)\backslash \Delta_w(\varrho)\right)\gg \alpha(\varrho)^2;
\end{align*}
therefore,
\begin{equation}
\label{alpharho-upper}
\alpha(\varrho)\leq c_0\sqrt{\varepsilon}
\end{equation}
for a calculable constant $c_0>0$.
In turn, Proposition~\ref{alpha-rho-r} yields that
\begin{equation}
\label{rho-minus-r-upper}
\varrho-r\ll\sqrt{\varepsilon}.
\end{equation}
According to \eqref{alpharho-upper}, we may choose $\varepsilon_0$ in a way such that 
\begin{equation}
\label{varepsilon0-alpharho}
c_0\sqrt{\varepsilon_0}\leq\frac{\pi}{72},\mbox{ \ and hence if $\varepsilon\leq \varepsilon_0$, then \ } 
\alpha(\varrho)\leq \frac{\pi}{72}.
\end{equation}

In the following, if $x\neq y$ are contained in a supporting horocycle $h$ to $B(p,\varrho)$, then $\arc{xy}$ denotes the arc of $h$ between $x$ and $y$. If $\{i,j,k\}=\{1,2,3\}$, then let $\tilde{y}_{ij},\tilde{y}_{ik}\in\partial B(p,\varrho)$ be the two endpoints of the circular arc $B(p,\varrho)\cap\partial \Sigma_i$ in a way such that $\tilde{y}_{ij}$ is closer to $\tilde{z}_j$ than to $\tilde{z}_k$, and hence $\arc{\tilde{y}_{ij}u_i}$
and $\arc{\tilde{y}_{ik}u_i}$ are the two horocyclic arcs bounding $\Sigma_i$. It follows from the definition $\alpha(\varrho)$
(cf. \eqref{alpharho-def})
and Lemma~\ref{spike-properties} that
\begin{equation}
\label{yijpyik}
\angle(\tilde{y}_{ij},p,u_i)=\angle(y_{ik},p,u_i)=\frac{\pi}3-\alpha(\varrho),
\end{equation}
and as $\tilde{z}_j$ lies on the arc of $\partial B(p,\varrho)$ between $\Sigma_i$ and $\Sigma_k$, we deduce that
\begin{equation}
\label{zipzj}
\begin{array}{rccccl}
\frac{2\pi}3-4\alpha(\varrho)&\leq& \angle(\tilde{z}_{i},p,\tilde{z}_{j})&\leq& \frac{2\pi}3+6\alpha(\varrho),&
\mbox{ \ and hence \ }\\
\frac{\pi}3-2\alpha(\varrho)&\leq &\angle(\tilde{z}_{i},p,\tilde{q}_{j})&\leq &\frac{\pi}3+3\alpha(\varrho)&<\frac{\pi}2.
\end{array}
\end{equation}
It follows from $\varrho\leq \frac{w}2$ and \eqref{zipzj} that there exists a calculable constant $R_w>0$ such that
\begin{equation}
\label{Rwbound}
d(p,\tilde{q}_i)\leq R_w,\;i=1,2,3.
\end{equation}
As for $\{i,j,k\}=\{1,2,3\}$, either $\tilde{z}_i$ lies on the shorter arc of $\partial B(p,\varrho)$ between $\tilde{y}_{ji}$ and $\tilde{y}_{ki}$, or $\tilde{y}_{ji}=\tilde{y}_{ki}=\tilde{z}_i$, we deduce from 
\eqref{yijpyik}, 
\eqref{zipzj} and \eqref{alpharho-upper} that
\begin{equation}
\label{ziyji}
\angle(\tilde{z}_{i},p,\tilde{y}_{ji})\leq 6\alpha(\varrho)\ll\sqrt{\varepsilon}.
\end{equation}

One of the observations we use in our argument below is that if 
$d(p,x)=d(p,v)\leq R_w$ and
$\angle(x,p,v)\leq 12\alpha(\varrho)$ (cf. \eqref{Rwbound} and \eqref{ziyji}), then \eqref{alpharho-upper} yields that
\begin{equation}
\label{rotate-x}
d(x,v)\ll\sqrt{\varepsilon}.
\end{equation}

Using the notation set up at the beginning of
Section~\ref{sec-isomin},
we may assume that $m_i(\varrho)\in \partial B(p,\varrho)$ used in the definition of $C_w(\varrho)$.
We recall that if $\{i,j,k\}=\{1,2,3\}$, then
$\angle(m_i(\varrho)),p,m_j(\varrho))=\frac{2\pi}3$, and $C_w(\varrho)$ is the $h$-convex hull of $B(p,\varrho)$ and $q_1(\varrho),q_2(\varrho),q_3(\varrho)$ where $p\in[m_i(\varrho),q_i(\varrho)]$ and $d(p,q_i(\varrho))=d(p,u_i)=w-\varrho$ for $i=1,2,3$.
We may also assume by \eqref{alpharho-upper} and \eqref{zipzj} that 
\begin{equation}
\label{zi-standard}
\tilde{z}_1=m_1(\varrho)
\mbox{ and }
\angle(\tilde{z}_{2},p,m_2(\varrho))\leq 6\alpha(\varrho)\ll\sqrt{\varepsilon}
\mbox{ and }
\angle(\tilde{z}_{3},p,m_3(\varrho))\leq 6\alpha(\varrho)\ll\sqrt{\varepsilon}.
\end{equation}
Concerning $T_w$, we may also assume that $B(p,r)$ is the incircle of $T_w$ for $r=r(T_w)$, and 
$\partial T_w\cap\partial B(p,r)=\{m_1(r),m_2(r),m_3(r)\}$ where $m_i(r)\in[p,m_i(\varrho)]$, and the vertices of $T_w$ are
$q_1(r),q_2(r),q_3(r)$ with
$p\in[m_i(r),q_i(r)]$ for $i=1,2,3$.

After this much preparation, we complete the proof of
Theorem~\ref{thm:pal_hconv-stab}
in three steps.\\

\noindent{\bf Step 1.} 
\begin{equation}
\label{Haus-tildeC-K}
\delta(K,\widetilde{C})\ll\sqrt{\varepsilon}.
\end{equation}
As $\widetilde{C}\subset K\subset \widetilde{T}$, it is sufficient to prove that 
$\delta(\widetilde{C},\widetilde{T})\ll\sqrt{\varepsilon}$. In turn, \eqref{Hausdorff-boundary}
yields that it is sufficient to verify that if $\{i,j,k\}=\{1,2,3\}$, then
\begin{equation}
\label{Haus-KC-step1-claim}
\delta\left(\arc{\tilde{z}_{i}\tilde{q}_{j}},\arc{\tilde{y}_{ji}u_j}\right)
\ll\sqrt{\varepsilon}.
\end{equation}
Let $\tilde{q}'_j$ be the image of $\tilde{q}_j$ by the rotation around $p$ that maps $\tilde{z}_{i}$ into $\tilde{y}_{ji}$, and hence
\eqref{Rwbound}, \eqref{ziyji} and
\eqref{rotate-x}  imply that
\begin{equation}
\label{ziqj-yjiqprimej}
\delta\left(\arc{\tilde{z}_{i}\tilde{q}_{j}},\arc{\tilde{y}_{ji}\tilde{q}'_{j}}\right)
\ll\sqrt{\varepsilon}.
\end{equation}
On the other hand,
we deduce from Lemma~\ref{angle-horocycle-arc},
\eqref{yijpyik} and 
\eqref{zipzj} that
$$
\left|\ell_H\left(\arc{\tilde{y}_{ji}\tilde{q}'_{j}}\right)-\ell_H\left(\arc{\tilde{y}_{ji}u_j}\right)\right|=
\left|\ell_H\left(\arc{\tilde{z}_{i}\tilde{q}_{j}}\right)-\ell_H\left(\arc{\tilde{y}_{ji}u_j}\right)\right|
\ll\sqrt{\varepsilon};
$$
therefore,
\begin{equation}
\label{yjiqprimej-yjiuj}
\delta\left(\arc{\tilde{y}_{ji}\tilde{q}'_{j}},\arc{\tilde{y}_{ji}u_{j}}\right)
\ll\sqrt{\varepsilon}.
\end{equation}
Combining \eqref{ziqj-yjiqprimej} and \eqref{yjiqprimej-yjiuj} yields \eqref{Haus-KC-step1-claim}, and in turn the estimate
\eqref{Haus-tildeC-K} of Step~1.\\

\noindent{\bf Step 2.}
\begin{equation}
\label{Haus-tildeC-Cw}
\delta(\widetilde{C},C_w(\varrho))\ll\sqrt{\varepsilon}.
\end{equation}
If $\{i,j,k\}=\{1,2,3\}$, then
let $y_{ij}(\varrho)$ and $y_{ik}(\varrho)$ be the endpoints of the circular arc of $\partial B(p,\varrho)$ on the boundary of the spike with apex $q_{i}(\varrho)$ corresponding to $B(p,\varrho)$ in a way such that 
$y_{ij}(\varrho)$ is closer to $m_j(\varrho)$ than $y_{ik}(\varrho)$, and hence \eqref{alpharho-upper},
\eqref{yijpyik} and 
\eqref{zi-standard} imply
\begin{equation}
\label{yijrho-tildeyij}
\angle\left(y_{ij}(\varrho),p,\tilde{y}_{ij})\right)
\leq 12\alpha(\varrho)
\ll\sqrt{\varepsilon}.
\end{equation}
We observe that $u_j$ is the image of $q_j(\varrho)$ by the rotation around $p$ that maps $y_{ji}(\varrho)$ into $\tilde{y}_{ji}$, and hence
\eqref{Rwbound}, 
\eqref{rotate-x} and \eqref{yijrho-tildeyij} imply that
\begin{equation}
\label{ziqj-yjiqprimej-rho}
\delta\left(\arc{y_{ij}(\varrho)q_j(\varrho)},\arc{\tilde{y}_{ji}u_j}\right)
\ll\sqrt{\varepsilon}.
\end{equation}
Therefore, \eqref{Hausdorff-boundary}
yields the estimate \eqref{Haus-tildeC-Cw} of Step~2.\\

\noindent{\bf Step 3.}
\begin{equation}
\label{Haus-Cwrho-Tw}
\delta(C_w(\varrho),T_w)\ll\sqrt{\varepsilon}.
\end{equation}
 We deduce via 
\eqref{Hausdorff-boundary}
 that it is sufficient to verify that if $\{i,j,k\}=\{1,2,3\}$, then
\begin{equation}
\label{Haus-Crho-Tw-step3-claim}
\delta\left(\arc{y_{ji}(\varrho)q_j(\varrho)},\arc{m_i(r)q_j(r)}\right)
\ll\sqrt{\varepsilon}.
\end{equation}
In this final part of the argument, if $x$ is a point of the supporting horocycle $h_i$ to $B(p,r)$ at $m_i$, then
we also write $\arc{m_ix}$ to denote the horocyclic arc between $m_i,x$. 

As $\angle\left(y_{ij}(\varrho),p,m_j(\varrho)\right)
=\alpha(\varrho)
\ll\sqrt{\varepsilon}$ (cf. \eqref{alpharho-upper}), \eqref{rho-minus-r-upper} and the triangle inequality yield that
$\left|d(y_{ij}(\varrho),q_i(\varrho))-
d(m_j(r),q_i(r))\right|\ll\sqrt{\varepsilon}$, and in turn we deduce via Lemma~\ref{Horo-secant} that
\begin{equation}
\label{length-yjiqjrho-mirqjr}
\left|\ell_H\left(\arc{y_{ji}(\varrho)q_j(\varrho)}\right)-\ell_H\left(\arc{m_i(r)q_j(r)}\right)\right|
\ll\sqrt{\varepsilon}.
\end{equation}

Abusing the usual notation for derivative, let $q_j(\varrho)'$ be the image of $q_j(\varrho)$ by the rotation around $p$ that maps $y_{ji}(\varrho)$ into $m_i(\varrho)$, and hence
\eqref{Rwbound}, 
\eqref{rotate-x} and $\angle\left(y_{ij}(\varrho),p,m_j(\varrho)\right)
=\alpha(\varrho)
\ll\sqrt{\varepsilon}$ imply that
\begin{equation}
\label{yjiqj-miqjprime}
\delta\left(\arc{y_{ij}(\varrho)q_j(\varrho)},\arc{m_i(\varrho)q_j(\varrho)'}\right)
\ll\sqrt{\varepsilon}.
\end{equation}
Let $q_j(\varrho)''$ be the closest point of $h_i$ to $q_j(\varrho)'$, and hence
Lemma~\ref{Horo-equidistant} and \eqref{rho-minus-r-upper} yield that
\begin{align}
\label{miqjprimerho-miqjprimer-delta}
\delta\left(\arc{m_i(\varrho)q_j(\varrho)'},\arc{m_i(r)q_j(\varrho)''}\right)
&\ll\sqrt{\varepsilon}\\
\label{miqjprimerho-miqjprimer-ellH}
\left|\ell_H\left(\arc{m_i(\varrho)q_j(\varrho)'}\right)-\ell_H\left(\arc{m_i(r)q_j(\varrho)''}\right)\right|&
\ll\sqrt{\varepsilon}.
\end{align}
We deduce from 
\eqref{length-yjiqjrho-mirqjr} and
\eqref{miqjprimerho-miqjprimer-ellH} that
\begin{equation}
\label{mirqjr-mirqjprimeprime}
\delta\left(\arc{m_i(r)q_j(\varrho)''},\arc{m_i(r)q_j(r)}\right)
\ll\sqrt{\varepsilon};
\end{equation}
therefore, combining
\eqref{yjiqj-miqjprime},
\eqref{miqjprimerho-miqjprimer-delta} and
\eqref{mirqjr-mirqjprimeprime}
leads to \eqref{Haus-Crho-Tw-step3-claim}, and in turn to the estimate \eqref{Haus-Cwrho-Tw} of
Step~3.\\

Finally, combining the estimates in Steps~1, 2 and 3 yields
 \eqref{pal_hconv-stab-epsi0}, and in turn
 Theorem~\ref{thm:pal_hconv-stab}.
\end{proof}

\section*{Acknowledgments}
We thank the anonymous referees for their very thorough review of our manuscript and their helpful comments that substantially improved the quality of the paper.

\end{document}